\documentclass[reqno,a4paper,12pt]{amsart}
\usepackage{hyperref}
\usepackage{url}
\usepackage{amscd,amssymb,amsmath,amsfonts,mathrsfs,yhmath}
\usepackage{relsize}

\usepackage{color}

\definecolor{darkgreen}{cmyk}{1,0,1,.2}
\definecolor{m}{rgb}{1,0.1,1}

\newcommand\verytiny{\fontsize{4pt}{4,8pt}\selectfont}
\newcommand\supp{\operatorname{supp}}
\newcommand\rank{\operatorname{rank}}

\newcommand\Kp{{\mathscr{K}}}

\renewcommand\lq{\leqslant}
\newcommand\gq{\geqslant}
\renewcommand\th{\theta}

\newcommand\X{\mathcal{X}}
\newcommand\Y{\mathcal{Y}}

\newcommand\CC{\mathcal{C}}

\newcommand\A{\mathcal{A}}

\newcommand\mv{{{MV}}}

\newcommand\diag{\operatorname{diag}}

\newcommand\Ad{\operatorname{Ad}}

\newcommand\ds{\displaystyle}
\renewcommand\S{\mathbb{S}}

\newcommand\C{\mathbb C}
\newcommand\Q{\mathbb Q}
\newcommand\K{\mathcal{K}}
\newcommand\F{\mathcal{F}}
\newcommand\Si{\Sigma}
\newcommand\De{\Delta}

\renewcommand\H{\mathscr{H}}
\newcommand\de{\delta}
\newcommand\N{\mathbb N}

\newcommand\Id{\mathcal{I}d}

\newcommand\R{\mathbb R}
\newcommand\st{\text{ such that }}
\newcommand\Z{\mathbb Z}
\newcommand\T{\mathcal{T}}
\newcommand\TT{\mathcal{T}}
\renewcommand\O{\mathcal{O}}
\newcommand\MM{\mathcal{M}}
\renewcommand\L{\mathcal{L}}
\newcommand\D{\mathcal D}
\newcommand\DD{\mathcal D}
\newcommand\G{\mathcal{G}}
\newcommand\ts{{\otimes}}

\newcommand\si{\sigma}
\newcommand\la{\lambda}
\newcommand{\aeq}{\stackrel{(\lambda,h)}{\sim}}

\newcommand\Ga{{\Gamma}}

\newcommand\ga{{\gamma}}
\newcommand\al{{\alpha}}
\newcommand\lto{{\longrightarrow}}
\newcommand\defi{{\stackrel{\text{def}}{=\!=}}}

\newcommand\M{{M}}
\newcommand\U{\operatorname{U}}
\renewcommand\P{\operatorname{P}}
\newcommand\ka{\kappa}
\newcommand\ue{\operatorname{U}_n^{\varepsilon,r}}
\newcommand\pe{\operatorname{P}_n^{\varepsilon,r}}
\newcommand\eps{\varepsilon}
\newcommand\erp{$\eps$-$r$-projection }
\newcommand\eru{$\eps$-$r$-unitary }

\theoremstyle{plain}
\newtheorem{theorem}{Theorem}[section]
\newtheorem{proposition}[theorem]{Proposition}
\newtheorem{corollary}[theorem]{Corollary}
\newtheorem*{conj}{Conjecture}
\newtheorem{lemma}[theorem]{Lemma}
\newtheorem{definition}[theorem]{Definition}
\theoremstyle{definition}
\newtheorem{remark}[theorem]{Remark}
\newtheorem{example}[theorem]{Example}
\newtheorem{notation}[theorem]{Notation}

\begin{document}
\title[Quant.  Index, Novikov Conj. \& Coarse  Decomposability]{Quantitative index, Novikov Conjecture and Coarse  Decomposability}

 \author[H. Oyono-Oyono]{Herv\'e Oyono-Oyono}
 \address{Universit\'e de Lorraine, Metz , France}
 \email{herve.oyono-oyono@math.cnrs.fr}
\author[G. Yu]{Guoliang Yu }
 \address{Texas A\&M University, USA}
 \email{guoliangyu@math.tamu.edu}
\thanks{Yu is partially  supported by NSF 2247313.}
\begin{abstract} 
 We define for families of finite metric spaces quantitative assembly map estimates that take into account propagation phenomena for pseudo-differential calculus. We relate these estimates to the Novikov conjecture and we show that they fit nicely under coarse decompositions. As an application, we provide a geometric  proof of Novikov conjecture for groups with finite decomposition complexity. 
\end{abstract}
\maketitle

\begin{flushleft}{\it Keywords:
    quantitative operator $K$-theory,  spaces with finite asymptotic dimension, group with finite complexity decomposition, Novikov conjecture, Coarse Baum-Connes conjecture.}
\medskip

\end{flushleft}

\tableofcontents
\section*{Introduction}

Since the pioneering work of Atiyah and Singer on the index theory of elliptic differential operators, operator algebras and their $K$-theory have become a central tool in the study of topological invariants in differential geometry. The most spectacular example is probably the Novikov conjecture on the homotopy invariance of higher signatures which can be stated as follows.  Let $\Ga$ be a finitely generated group and  let $M$ be a compact oriented manifold with fundamental group $\Ga$. Let $B_\Ga$ be the classifying space for $\Ga$ and let $f_\Ga:M\to B_\Ga$ be the classifying map. We denote by $\mathbb{L}(M)\in H^*(M,\Q)$ the Pointrjagin-Hirzebruch class of $M$.  Recall that we have a natural isomorphism $H^*(\Ga,\Q)\cong H^*(B_\Ga,\Q)$, so any $x$ in $H^*(\Ga,\Q)$ can be pulled back to an element $f^*_\Ga(x)$ in $H^*(M,\Q)$. The rational numbers  $$\langle f^*_\Ga(x),\mathbb{L}(M)\cap [M]\rangle$$  are called  higher signatures of $M$
(here $[M]$ stands for the fundamental class for $M$).
\begin{conj}[Novikov]
The higher signatures are oriented  homotopy invariant.
\end{conj}
By introducing in his work on the Novikov conjecture the bivariant $K$-theory \cite{kas}, Kasparov provided a far reaching generalization to the concept of index with a wide range of important applications to  geometry. He proved the Novikov conjecture for discrete subgroups of Lie groups. His approach has been extended successfully to the case of groups actings on an euclidian building \cite{ks1}, of "bolic" groups (including Gromov hyperbolic groups) \cite{ks2} and had  culminated in the proof of Novikov conjecture for groups which are  amenable at infinity \cite{hr,y1}.

Following the development by Roe of a higher index theory for elliptic differential operators on non-compact complete riemannian manifolds \cite{roe}, the 90's have seen the emergence of a more geometric approach. It has generated a great deal of interest  as it has highlighted geometric obstructions for the Baum-Connes conjecture leading to the only known counterexamples \cite{hls}. It has also striking applications to secondary invariants \cite{wxy,xy}. 
The foundational principle  is the fact that the higher indices have their receptacles in the $K$-theory of a canonical $C^*$-algebra, called the Roe algebra. The Roe algebra  only depends on the large scale geometry of the underlying space. It can be viewed as an algebra  of order zero pseudo-differential calculus and inherits from the geometry a propagation structure. For a finitely generated group (viewed as a metric space using any word length) the computability of the Roe algebra using higher indices  implies (under some very mild condition) the Novikov conjecture.
A prominent  achievement of this is the proof of the Novikov conjecture for finitely generated groups with finite asymptotic dimension \cite{y2} by proving that the $K$-theory of a localized version of the  Roe algebra can be computed using a "cut-and-pasting" techniques. This was done by introducing for these algebras a quantitative version of the $K$-theory which takes only into account propagation up to a given order. At each order there is a Mayer-Vietoris type decomposition  which allowed to compute the quantitative $K$-theory and then, the $K$-theory is obtained by letting the order goes to infinity. Quantitative $K$-theory has been generalized in \cite{oy2} to the setting of filtered $C^*$-algebras by  the two authors. Thereafter, they   introduced in \cite[Section 3.1]{oy3}  geometric controlled assembly maps that compute locally the $K$-theory of the Roe algebra. For a finite metric space $X$ and a $C^*$-algebra $A$, these  maps have source the $K$-homology with coefficients in $A$ of the Rips complex of $X$ and range the quantitative $K$-theory of $A\ts \Kp(\ell^2(X))$. Here $\Kp(\ell^2(X))$ stands for  the algebra of matrices with entries in $X$ and the filtration is provided by the metric.  Quantitative assembly maps (QAM) estimates for metric spaces 
where stated for purpose of computing the $K$-theory of Roe algebras from the geometric controlled assembly maps by uniform approximations \cite[Section 4.2]{oy3}. It was then shown that if $\Ga$ is a finitely generated group such that  the
QAM-estimates are uniformly satisfied by  the family of its  finite subsets, then $\Ga$ satisfies the Novikov conjecture. It is worth pointing out that unlike the proof of the Novikov conjecture for groups amenable at infinity, this QAM-estimates only involve finite dimension analysis and finite linear algebra.

The aim of this paper is to establish a stability result  for uniform QAM-estimates under geometric decomposition   that we shall explain  in more details now.  Geometric decomposability in the context of metric spaces has been brought in by 
by Guentner, Tessera and Yu in their study of topological rigidity for manifolds \cite{gty1}. For a given positive number $R$, a family of discrete metric spaces $\X$ is said to be  $R$-decomposable with respect to another family $\Y$ if for any metric space $X$ in $\X$, there exists a decomposition 
 $X=X^{(1)}\cup X^{(2)}$ such that $X^{(i)}$ is for  $i=1,2$  the 
$R$-disjoint union (i.e at mutual distance greater than $R$) of metric spaces in $\Y$. We are now in position to state the  main result of this paper  (Theorem \ref{theorem-main}). Let $\X$ be a family of finite metric spaces. Assume that for any $R>0$, there exists a family of finite metric spaces
$\Y$  such that 
\begin{itemize}
\item $\X$ is  $R$-decomposable with respect to $\Y$;
\item the family $\widetilde{\Y}$ of all subsets of all metric spaces in $\Y$ satisfies uniformly the QAM-estimates.
\end{itemize} 
Then the  family $\X$ satisfies (under some uniform boundedness condition) uniformly the QAM-estimates. 
 A special case of interest is the class of families with finite decomposition complexity  \cite{gty1}. 
 A class  $\CC$  of families of  proper discrete metric spaces is said to be closed under
coarse decomposition if any   family  $\X$ of discrete proper metric spaces   which is  $R$-decomposable with respect to a family $\Y$ in $\CC$  for any positive number $R$ lies indeed in $\CC$. Let us define the class $\CC_{fdc}$ as the smallest class of families that contains uniformly bounded families of  discrete metric spaces
and which is closed under coarse decomposability. As   uniformly bounded families of  discrete metric spaces satisfy obviously uniformly the QAM-estimates, we obtain as a consequence of our main result that 
any family $\X$  in $\CC_{fdc}$ with uniformly bounded geometry 
 satisfies uniformly the QAM-estimates (Theorem \ref{theorem-QAM-fdc}).  
 If $X$ is a discrete proper metric space, we say that $X$ has finite decomposition complexity if  the single family $\{X\}$ is in $\CC_{fdc}$ \cite{gty1}. Examples of such metric spaces are provided by metric spaces with finite asymptotic dimension.
 As a consequence of Theorem \ref{theorem-QAM-fdc}, we get that if $\Ga$ is a finitely generated group with  finite decomposition complexity which admits a classifying space with finite homotopy type, then
 $\Ga$ satisfies the Novikov conjecture. This result was already known as group with finite decomposition complexity are amenable \cite{gty2}  but as already emphasized, our proof does not involve infinite dimension analysis.

Let us now explain the general idea of the proof of the main result of this paper.
In \cite{oy4}  was stated a controlled Mayer-Vietoris exact sequence in quantitative $K$-theory. It was indeed motivated by $K$-theory computability of Roe type algebras under the geometric decompositions described above and thus result in  six-term controlled exact sequences on the right hand-side of the quantitative assembly maps.  These decompositions also give rise to Mayer-Vietoris exact sequences in $K$-homology on the left-hand side of the quantitative assembly maps. Our proof then amounts to prove the compatibility of these two sequences under the quantitative assembly maps. The major obstacle is that the two boundary maps do not have the same nature. In $K$-homology, the boundary maps is constructed using homotopy theory whereas in a quantitative $K$-theory, we use an approximating lifting property for almost unitaries. The first step to circumvent this obstacle is to show that as it is the case for usual assembly maps,  local quantitative coarse assembly map factorize through  quantitative index maps. These quantitative index maps are indeed defined for any compact metric space $Z$ and is valued for a fixed non degenerate standard $Z$-module $\H_Z$ in the quantitative $K$-theory of the algebra  $\Kp(\H_Z)$ of compact operators on $\H_Z$ (or in $A\ts\Kp(\H_Z)$ for the version with coefficients). The second step is to prove that for finite
simplicial complexe and for small propagation (depending only on the dimension), the quantitative index maps are (up to some rescaling)   isomorphisms. This is done once again by using  controlled Mayer-Vietoris exact sequence in quantitative $K$-theory. Checking the compatibility of the boundary maps is a quite tedious chore but it is worth  noticing that it can be straightforwardly generalized to the equivariant case (in particular with respect to groupoid actions). The proof of our main result then  amounts to show that for any finite metric space $X$ in $\X$,  then the  map
induced in quantitative $K$-theory by some canonical morphism 
$\Kp(\H_{P_d(X)})\to\Kp(\ell^2(X))$ is an isomorphism up to uniformly controlled rescaling  if the same holds for any metric spaces in $\widetilde{\Y}$. Geometric decompositions give also rise to controlled Mayer-Vietoris exact sequence to compute the quantitative 
$\Kp(\H_{P_d(X)})$.
The two controlled exact sequences being now intertwined by natural morphisms, this is achieved by using a controlled version of the five lemma.

By analogy with  metric spaces, the   notion of geometric decomposition can be also defined for groupoids \cite{gwy2,o}. An interesting perspective would be to extend the main result of the paper to prove the permanence property of the Baum-Connes conjecture for \'etale groupoids under geometric decomposition  \cite[Theorem 6.8]{o}  without the assumption of the existence of a $\ga$-element in the sense of  \cite[Proposition 5.20 and Remark 5.21]{tuhyp}. To reach this goal, we need to develop the analogous of the theory of standard module in the framework of cocompact proper action of groupoids (see \cite{bs} for the case of ample groupoids).

\medbreak

{\it Outline of the paper.} 

In Section 1, we review from \cite{oy2,oy3} the main features of quantitative $K$-theory for filtered $C^*$-algebras.
We draw attention of the reader on the fact that the definition 
of quantitative  $K$-theory that we give in this paper is slightly different from the original one given in [11][Section 3.1].
We have changed the range of the control $\eps$ and allowed larger homotopies in the equivalence relations between almost projections and almost unitaries. The purpose for this is to define later on in this section the rescaling operator which allows to reduce the control $\eps$ in quantitative $K$-theory.

The main theorem is stated in Section 2. We first recall from
 \cite{oy3} the construction of the geometric controlled assembly map which can be viewed as  a geometric version of the controlled Baum-Connes assembly maps (see \cite{oy2}). Next we  give the formulation of the quantitative assembly maps estimates for a proper discrete metric space. We recall the link with the coarse Baum-Connes conjecture and with the Novikov conjecture. We then state the main result of this paper.  We apply this result to metric spaces with finite decomposition complexity introduced in \cite{gty2} and we obtain has a consequence  the Novikov conjecture for finitely generated groups with finite decomposition  complexity   (under the standard assumption on classifying spaces).

 Section 3 is devoted to the definition and to the study of the quantitative index map.
 After some reminders about non-degenerated standard modules and propagation for theirs operators, we introduce for a compact metric space $X$ and a non-degenerated standard $X$-modules $\H_X$ the fundamental class in quantitative $K$-theory.  This is a family of rank $1$ projections in $C(X,\Kp({\H_X}))$ with arbitrary small propagation and which are moreover compatible with respect to the structure maps in quantitative $K$-theory.  This leads to the definition of quantitative indices. We then study some naturality properties for this 
 map and we state the fundamental theorem for  quantitative indices : if $X$ is a finite simplicial complexe, then for small propagation depending only  on the dimension, quantitative indices are (up to some controlled rescaling)  isomorphisms. This theorem will be proved in Section 4. We end the section by  proving that the geometric quantitative assembly maps factorizes through  quantitative index maps.
 
 In Section 4, we introduce the central tool  of this paper : the controlled six terms exact sequence in quantitative $K$-theory associated to a Mayer-Vietoris pair decomposition for a filtered $C^*$-algebra \cite{oy4}.  We show that a classical  Mayer-Vietoris decomposition for a compact metric space gives rise to a Mayer-Vietoris pair for the right hand side of the quantitative index map. We prove by an arduous computation that the Mayer-Vietoris boundary maps in $K$-homology
and in quantitative $K$-theory are intertwined by the quantitative index maps. As a consequence, we prove by using a controlled version of the 
 five Lemma that for small propagation, the quantitative index is an isomorphism up to some rescaling.

 In Section 5,  we collect results from Section 3 and 4 to show that the proof  of the main theorem 
 reduces to replace in the QAM-estimates the geometric assembly maps by the morphisms
induced in quantitative $K$-theory by some canonical map 
$\Kp(\H_{P_d(X)})\to\Kp(\ell^2(X))$.  Noticing that geometric decompositions give rise to compatible  Mayer-Vietoris pair decompositions both for $\Kp(\H_{P_d(X)})$ and $\Kp(\ell^2(X))$,
this us done 
 by using once again a five Lemma type argument for the associated 
  controlled six terms exact sequences in quantitative $K$-theory.

\section{Overview on quantitative $K$-theory}
In this section, we recall the basic concepts of quantitative $K$-theory for filtered $C^\ast$-algebras and collect  the main results of \cite{oy2} concerning quantitative $K$-theory that we shall use throughout  this paper. 
The structure of filtered $C^*$-algebras allows us to talk about the propagation scale of elements in the $C^\ast$-algebras.
Roughly speaking,  quantitative $K$-theory is the abelian groups of $K$-theory elements at each given scale.   
\begin{definition}
A filtered $C^*$-algebra $A$ is a $C^*$-algebra equipped with a family
$(A_r)_{r>0}$ of  closed linear subspaces   indexed by positive numbers such that:
\begin{itemize}
\item $A_r\subseteq A_{r'}$ if $r\lq r'$;
\item $A_r$ is stable under involution;
\item $A_r\cdot A_{r'}\subseteq A_{r+r'}$;
\item the subalgebra $\ds\bigcup_{r>0}A_r$ is dense in $A$.
\end{itemize}
If $A$ is unital, we also require that the identity  $1$ is an element of $ A_r$ for every positive
number $r$. The elements of $A_r$ are said to have {\bf propagation $r$}.
\end{definition}
 Let $A$ and $A'$ be respectively  $C^*$-algebras filtered by
$(A_r)_{r>0}$ and  $(A'_r)_{r>0}$. A homomorphism of $C^*$
-algebras $\phi:A\lto A'$
is a {\bf filtered homomorphism} (or a {\bf homomorphism of  filtered $C^*$-algebras}) if  $\phi(A_r)\subseteq A'_{r}$ for any
positive number $r$.

\medskip

If $A$ is not unital, let us denote by $A^+$ its unitarization, i.e.,
$$A^+=\{(x,\lambda);\,x\in A\,,\lambda\in \C\}$$  with the product $$(x,\lambda)(x',\lambda')=(xx'+\lambda x'+\lambda' x,\la\la')$$ for all $(x,\lambda)$ and $(x',\lambda')$ in $A^+$. Then ${A^+}$ is filtered with
$${A^+_r}=\{(x,\lambda);\,x\in A_{r}\,,\lambda\in \C\}.$$
We also define $\rho_A:A^+\to\C;\, (x,\lambda)\mapsto \lambda$.

\subsection{Definition of quantitative $K$-theory}

Let $A$ be a unital filtered $C^*$-algebra. For any  positive
numbers $r$ and $\eps$, we call
\begin{itemize}
\item an element $u$ in $A$  an $\eps$-$r$-unitary if $u$
  belongs to $A_r$,  $\|u^*\cdot
  u-1\|<\eps$
and  $\|u\cdot u^*-1\|<\eps$. The set of $\eps$-$r$-unitaries on $A$ will be denoted by $\operatorname{U}^{\varepsilon,r}(A)$.
\item an element $p$ in $A$   an $\eps$-$r$-projection    if $p$
  belongs to $A_r$,
  $p=p^*$ and  $\|p^2-p\|<\eps$. The set of $\eps$-$r$-projections on $A$ will be denoted by $\operatorname{P}^{\varepsilon,r}(A)$.
\end{itemize}
Notice that an $\eps$-$r$-unitary is invertible, and that if $p$ is an \erp in $A$, then it has a spectral gap around $1/2$ and then gives rise by functional calculus to a  projection $\ka_{0}(p)$  in  $A$ such that
 $\|p-\ka_{0}(p)\|< 2\eps$.

Recall  from \cite[Lemma 1.7]{oy3}
the following result that  will be  used quite  extensively throughout  the paper.
\begin{lemma}\label{lemma-almost-closed}
 Let $A$ be a  $C^*$-algebra filtered by $(A_r)_{r>0}$.
\begin{enumerate}

\item  If $p$  is an \erp in $A$ and  $q$ is a self-adjoint element of $A_r$ such
  that
   $\|p-q\|<\frac{\eps-\|p^2-p\|}{4}$, then $q$ is an  $\eps$-$r$-projection. In
   particular, if $p$ is an \erp in $A$ and if $q$ is a self-adjoint  element in
   $A_r$ such that $\|p-q\|< \eps$, then $q$ is a
   $5\eps$-$r$-projection in $A$ and $p$ and $q$ are connected by a
   homotopy of $5\eps$-$r$-projections.
\item  If $A$ is unital, $u$ is an \eru and  $v$ is an element of $A_r$ such that
  $\|u-v\|<\frac{\eps-\|u^*u-1\|}{3}$, then $v$ is an $\eps$-$r$-unitary. In
  particular, if $u$ is an \eru and $v$ is an element of $A_r$ such that
  $\|u-v\|<\eps$, then $v$ is a $4\eps$-$r$-unitary  in $A$ and $u$
  and $v$ are connected by a homotopy of $4\eps$-$r$-unitaries.
 \item  If $p$  is a projection  in $A$ and  $q$ is a self-adjoint element of $A_r$ such
  that
   $\|p-q\|<\frac{\eps}{4}$, then $q$ is an  $\eps$-$r$-projection. 
   \item  If $A$ is unital,  $u$ is a unitary in $A$  and  $v$ is an element of $A_r$ such that
  $\|u-v\|<\frac{\eps}{3}$, then $v$ is an $\eps$-$r$-unitary.
\end{enumerate}
\end{lemma}
We get as a corollary:
 \begin{corollary}\label{cor-example-homotopy}
 Let $u$ be an $\eps$-$r$-unitary in a unital filtered $C^*$-algebra $A$, then $\diag(u,u^*)$ and
 $I_2$ are homotopic as   $3\eps$-$2r$-unitaries  in   $M_2(A)$.
 \end{corollary}

%
%
%
%
%
%


 For any   integer $n$, we set  $\ue(A)=\operatorname{U}^{\varepsilon,r}(M_n(A))$ and
$\pe(A)=\operatorname{P}^{\varepsilon,r}(M_n(A))$.
For any  unital filtered $C^*$-algebra $A$, any
 positive  numbers $\eps$ and $r$ and  any positive integer $n$, we consider the inclusions
$$\P_n^{\eps,r}(A)\hookrightarrow \P_{n+1}^{\eps,r}(A);\,p\mapsto
\begin{pmatrix}p&0\\0&0\end{pmatrix}$$ and
$$\U_n^{\eps,r}(A)\hookrightarrow \U_{n+1}^{\eps,r}(A);\,u\mapsto
\begin{pmatrix}u&0\\0&1\end{pmatrix}.$$ This allows us to  define
 $$\U_{\infty}^{\eps,r}(A)=\bigcup_{n\in\N}\ue(A)$$ and
$$\P_{\infty}^{\eps,r}(A)=\bigcup_{n\in\N}\pe(A).$$

For a unital filtered $C^*$-algebra $A$, $\eps\in(0,1/100)$ and $r>0$, we define the
following
equivalence relations on $\P_\infty^{\eps,r}(A)\times\N$ and on  $\U_\infty^{\eps,r}(A)$:
\begin{itemize}
\item if $p$ and $q$ are elements of $\P_\infty^{\eps,r}(A)$, $l$ and
  $l'$ are positive integers, $(p,l)\sim(q,l')$ if there exists a
  positive integer $k$ and an element $h$ of
  $\P_\infty^{25\eps,r}(A[0,1])$ such that $h(0)=\diag(p,I_{k+l'})$
and $h(1)=\diag(q,I_{k+l})$.
\item if $u$ and $v$ are elements of $\U_\infty^{\eps,r}(A)$, $u\sim v$ if
  there exists an element $h$ of
  $\U_\infty^{25\eps,2r}(A[0,1])$ such that $h(0)=u$
and $h(1)=v$.
\end{itemize}

If $p$ is an  element of $\P_\infty^{\eps,r}(A)$ and  $l$ is an integer, we
denote by $[p,l]_{\eps,r}$ the equivalence class of $(p,l)$ modulo  $\sim$
and if $u$ is an element of $\U_\infty^{\eps,r}(A)$ we denote by
$[u]_{\eps,r}$ its  equivalence class  modulo  $\sim$.
\begin{definition} Let $r$ and $\eps$ be positive numbers with
  $\eps<1/100$.
We define:
\begin{enumerate}
\item $K_0^{\eps,r}(A)=\P_\infty^{\eps,r}(A)\times\N/\sim$ for $A$ unital and
$$K_0^{\eps,r}(A)=\{[p,l]_{\eps,r}\in K_0({A^+}) \st
\rank \kappa_0(\rho_{A}(p))=l\}$$ for $A$ non unital ($\kappa_0(\rho_{A}(p))$ being the spectral projection associated to $\rho_A(p)$);
\item $K_1^{\eps,r}(A)=\U_\infty^{\eps,r}({A^+})/\sim$, with $A=A^+$ if $A$ is already unital.
\end{enumerate}
\end{definition}
 Then $K_0^{\eps,r}(A)$ turns to be an abelian group  \cite[Lemma 1.15]{oy2} where
 $$[p,l]_{\eps,r}+[p',l']_{\eps,r}=[\diag(p,p'),l+l']_{\eps,r}$$  for any  $[p,l]_{\eps,r}$ and $[p',l']_{\eps,r}$ in $K_0^{\eps,r}(A)$. According to \cite[Lemma 1.15]{oy2},  $K_1^{\eps,r}(A)$ is
 equipped with a structure of abelian group such that
$$[u]_{\eps,r}+[u']_{\eps,r}=[\diag(u,v)]_{\eps,r},$$ for
any  $[u]_{\eps,r}$ and $[u']_{\eps,r}$ in $K_1^{\eps,r}(A)$.

We draw attention of the readers on the fact that this definition of controlled 
$K$-theory differs slightly from the original one given in \cite{oy2}[Section 3.1]. We have 
increased control of homotopies for purpose of rescaling (see Section \ref{subsec-rescaling}).
As quantitative objects (see Section \ref{subsec-quant-object}), the two objects are related by a controlled isomorphism.
Furthermore, controlled morphisms and isomorphisms and controlled exact sequences for both objects are in natural correspondance with each other.

  Recall from \cite[Corollaries 1.19 and 1.21]{oy2} that
for any positive numbers  $r$ and $\eps$ with $\eps<1/100$, then
$$K_0^{\eps,r}(\C)\to\Z;\,[p,l]_{\eps,r}\mapsto \rank\kappa_0(p)-l$$
is an isomorphism and
$K_1^{\eps,r}(\C)=\{0\}$.

 \medskip

 We have for any filtered $C^*$-algebra $A$ and any  positive numbers
$r$, $r'$, $\eps$ and $\eps'$  with
  $\eps\lq\eps'<1/100$ and $r\lq r'$  natural group homomorphisms
\begin{itemize}
\item $\iota_0^{\eps,r}:K_0^{\eps,r}(A)\lto K_0(A);\,
[p,l]_{\eps,r}\mapsto [\kappa_0(p)]-[I_l]$ (where  $\kappa_0(p)$ is the spectral projection associated to $p$);
\item $\iota_1^{\eps,r}:K_1^{\eps,r}(A)\lto K_1(A);\,
  [u]_{\eps,r}\mapsto [u]$  ;
\item $\iota_*^{\eps,r}=\iota_0^{\eps,r}\oplus \iota_1^{\eps,r}$;
\item $\iota_0^{\eps,\eps',r,r'}:K_0^{\eps,r}(A)\lto K_0^{\eps',r'}(A);\,
[p,l]_{\eps,r}\mapsto [p,l]_{\eps',r'};$
\item $\iota_1^{\eps,\eps',r,r'}:K_1^{\eps,r}(A)\lto K_1^{\eps',r'}(A);\,
  [u]_{\eps,r}\mapsto [u]_{\eps',r'}$.
\item $\iota_*^{\eps,\eps',r,r'}=\iota_0^{\eps,\eps',r,r'}\oplus\iota_1^{\eps,\eps',r,r'}$
\end{itemize}
If some of the indices $r,r'$ or $\eps,\eps'$ are equal, we shall not
repeat it in $\iota_*^{\eps,\eps',r,r'}$. Furthermore, when 
the range (resp. source, source and range) is obvious, we shall write
$\iota_*^{\eps,r,-}$ (resp. $\iota_*^{-,\eps',r'}$, $\iota_*^{-,-}$) instead of
$\iota_*^{\eps,\eps',r,r'}$.

Notice that with this definition, the family of morphisms
$$(\iota_*^{\eps,r}:K_*^{\eps,r}(A)\to K_*(A))_{\eps\in(0,1/100),r>0}$$ gives rise to an isomorphism
\begin{equation}\label{eq-direct-limit}\lim_{r>0}K_*^{\eps,r}(A)\stackrel{\simeq}{\lto} K_*(A),\end{equation}where the inductive limit on the 
left hand side is taken with respect to the inductive system
$$\iota_*^{\eps,r,r'}:K_*^{\eps,r}(A)\to K_*^{\eps,r'}(A)$$for $r'\gq r>0$.

 If  $\phi:A\to B$ is a  homomorphism of  filtered $C^*$-algebras, then since $\phi$ preserves $\eps$-$r$-projections
 and $\eps$-$r$-unitaries, it obviously induces  for any positive number $r$ and any
$\eps\in(0,1/100)$ a
group homomorphism $$\phi_*^{\eps,r}:K_*^{\eps,r}(A)\longrightarrow
K_*^{\eps,r}(B).$$  Moreover the quantitative $K$-theory is homotopy invariant with respect to homotopies that preserves propagation
\cite[Lemma 1.26]{oy2}.
There is also a quantitative version of Morita equivalence \cite[Proposition 1.28]{oy2}.
\begin{proposition}\label{prop-morita}
If $A$ is a filtered algebra and $\H$ is a separable Hilbert space, then the homomorphism
$$A\to \Kp(\H)\otimes A;\,a\mapsto \begin{pmatrix}a&&\\&0&\\&&\ddots\end{pmatrix}$$
induces a ($\Z_2$-graded) group isomorphism (the Morita equivalence)
$$\MM_A^{\eps,r}:K_*^{\eps,r}(A)\to K_*^{\eps,r}(\Kp(\H)\otimes A)$$
for any positive number $r$ and any
$\eps\in(0,1/100)$.
\end{proposition}

\subsection{Rescaling the control}\label{subsec-rescaling}
In this subsection, we describe a rescaling process that allows to fix the control in quantitative $K$-theory. 

Recall that if $q$ is an $\eps$-$r$ projection, then $\kappa_0(q)$ stands for the spectral projection corresponding to $q'$. If $u$ is an $\eps$-$r$ unitary, we set $\kappa_1(u)=(u^*u)^{-1/2}u$ (then  $\kappa_1(u)$ is a unitary).
\begin{lemma}\label{lemma-rescaling-projection}
There exists a non increasing function $$(0,1/8)\to [1,+\infty);\,\eps\mapsto l_\eps$$ such that for any unital filtered $C^*$-algebra $A$,  any $\eps$ in $(0,1/8)$ and  any $r>0$,  then the following holds
\begin{enumerate}
\item for any $\eps$-$r$-projection $q$ in $A$, there exists an 
$\eps/2$-$l_\eps r$-projection $q'$ such that $\|\ka_0(q)-q'\|<\eps/8$.
\item for any $\eps$-unitary $u$ in $A$,  there exists an $\eps/2$-$l_\eps r$-unitary $u'$ such that $\|\ka_1(u)-u'\|<\eps/6$.
\end{enumerate}
Moreover, if $J$ is an ideal in $A$ and if $q-1$ (resp. $u-1$) lies in $J$, then we can choose $q'$ (resp. $u'$) such that $q'-1$ (resp. $u'-1$) lies also in $J$.
\end{lemma}
\begin{proof}
Let us first prove the even case.
Let us fix $\al$ in $\left(0,1-{1}/{\sqrt{2}}\right)$. Then the spectrum of $q$ is included in
$[-\al,\al]\cup[1-\al,1+\al]$ for any $\eps$-$r$-projection $q$ with $\eps$ in $(0,1/8)$. Let us fix $\ka_0:\R\to\R$ a continuous function such $\ka(t)=0$ for $t$ in 
$[-\al,\al]$ and $\kappa_0(t)=1$ for $t$ in $[1-\al,1+\al]$ and let
$(Q_N)_{N\in\N}$ be a polynomial sequence such that 
$(Q_N)_{N\in\N}$ converge uniformly to $\ka_0$ on $[-\al,\al]\cup[1-\al,1+\al]$. For $\eps$ in $(0,1/8)$, let $N_\eps$ be the smallest integer such that $\|Q_{N_\eps}(t)-\ka_0(t)\|<\eps/8$ for all $t$ in $[-\al,\al]\cup[1-\al,1+\al]$. Set then $l_\eps=\max\{d^oQ_k;\, k=0,\ldots,N_\eps\}$. Then $q'=Q_{N_\eps}(q)$ belongs to $A_{l_\eps r}$ and
$\|\ka_0(q)-q'\|<\eps/8$. In particular, according to the third point of Lemma \ref{lemma-almost-closed}, we see that $q'$ is an $\eps/2$-$l_\eps r$-projection.

\smallbreak

In the same way for the odd case, using the power serie of $t\mapsto (1+t)^{-1/2}$, we see that there exists:
\begin{itemize}
\item a non increasing function $(0,1/4)\lto \N;\,\eps\mapsto l_\eps$
\item for any $\eps$ in $(0,1/4)$ a polynomial function $Q_\eps$ 
\end{itemize}such that
\begin{itemize}
\item The degree of $Q_\eps$ is $l_\eps$;
\item $Q_\eps(1)=1$
\item $|Q_\eps(t)-t^{-1/2}|<\eps/8$ for all $t$ in $[3/4,5/4]$.
\end{itemize}
If we  set $u'=Q_\eps(u^*u)u$, then we have that  
$\|\ka_1(u)-u'\|<\eps/6$ and in view of the fourth  point of Lemma \ref{lemma-almost-closed}, we see that $u'$ is an $\eps/2$-$(2l_\eps+1) r$-unitary.
\end{proof}
If $q$ is an $\eps$-$r$-projection  in $A$, with $\eps$ in $(0,1/200)$, we deduce from the first point of 
Lemma \ref{lemma-almost-closed} that two elements that satisfy the condition of the first item of the lemma are homotopic as $\frac{5}{4}\eps$-$r$-projections and hence there exists a well defined  homomorphism 
$$L_{A,0}^{\eps,r}:K_0^{\eps,r}(A)\lto K_0^{\eps/2,l_\eps r}(A),$$ natural with respect to homomorphisms of filtered $C^*$-algebras, such that if $A$ is unital, then
$L_{A,0}^{\eps,r}([q,k]_{\eps,r})=[q',k]_{\eps/2,l_\eps r}$ for  any $\eps$-$r$-projection $q$ in some $M_n(A)$, any integer $k$ and any $\eps/2$-$l_\eps r$-projection $q'$ in  $M_n(A)$
such that $\|\ka_0(q)-q'\|<\eps/8$.

In the same way, 
there exists a well defined natural homomorphism 
$$L_{A,1}^{\eps,r}:K_1^{\eps,r}(A)\lto K_1^{\eps/2,l_\eps r}(A)$$ such that if $A$ is unital, then
$L_{A,1}^{\eps,r}([u]_{\eps,r})=[u']_{\eps/2,l_\eps r}$ for  any $\eps$-$r$-unitary $u$ in some $M_n(A)$  and any $\eps/2$-$l_\eps r$-unitary  $u'$ in  $M_n(A)$
satisfying $\|\ka_1(u)-u'\|<\eps/8$.

For any $\eps$ in $(0,1/200)$ and $r>0$, let us set $l_{rs,\eps}=l_\eps$ and  $$L_{A,*}^{\eps,r}=L_{A,0}^{\eps,r}\oplus L_{A,1}^{\eps,r}:K_*^{\eps,r}(A)\lto K_*^{\eps/2,l_{rs,\eps}r} (A).$$
We clearly have that
$$L_{A,*}^{\eps',r'}\circ \iota_*^{\eps,r,-}=\iota_*^{-,\eps'/2,l_{rs,\eps'} r'}\circ L_{A,*}^{\eps,r}$$
for any $\eps$, $\eps'$, $r$ and $r'$ such that
$0<\eps\lq\eps'<1/200$, $0<r\lq r'$ and $l_{rs,\eps} r\lq l_{rs,\eps'}r'$.
 \begin{lemma}\label{lemma-rescaling}
  For any unital filtered $C^*$-algebra $A$,  any $\eps$ in $(0,1/200)$ and  any $r>0$,  then  
  \begin{enumerate}
  \item $\iota_*^{\eps/2,\eps,l_{\eps}r}\circ L_{A,*}^{\eps,r}=\iota_*^{\eps,r,l_{\eps}r}$;
  \item $L_{A,*}^{\eps,r}\circ \iota_*^{\eps/2,\eps,r}=\iota_*^{\eps/2,l_\eps r}$.
  \end{enumerate}
  \end{lemma}
  \begin{proof} We can assume without loss of generality that $A$ is unital.
  Let us prove the first point in the even case.
  Let $q$ be an $\eps$-$r$-projection in some $M_n(A)$ and let $q'$ be an $\eps/2$-$l_\eps r$-projection in
  $M_n(A)$  such that $\|\kappa_0(q)-q'\|<\eps/8$ as in Lemma \ref{lemma-rescaling-projection}.
  Then 
 \begin{eqnarray*}
 \|q-q'\|&\lq&\|q-\kappa_0(q)\|+\|\kappa_0(q)-q'\|\\
 &<&\frac{17\eps}{8}.
 \end{eqnarray*}
In view of  the first point of 
Lemma \ref{lemma-almost-closed}, we obtain that 
$q$ and $q'$ are homotopic as  $25\eps$-$l_\eps r$-projections and we deduce the result for  the first point in the even case. The odd case is quite similar.

For the second point, let $q$ be an $\eps/2$-$r$-projection in some $M_n(A)$ and let us consider $q$ as a 
$\eps$-$r$-projection. Pick then 
an $\eps/2$-$l_\eps r$-projection $q'$ in
  $M_n(A)$  such that $\|\kappa_0(q)-q'\|<\eps/8$.
   Then 
 \begin{eqnarray*}
 \|q-q'\|&\lq&\|q-\kappa_0(q)\|+\|\kappa_0(q)-q'\|\\
 &<&\frac{9\eps}{8}
 \end{eqnarray*}
and hence 
$q$ and $q'$ are homotopic as  $\frac{25\eps}{2}$-$l_\eps r$-projections and we deduce the result of  the second point in the even case, the  odd case being similar.
\end{proof}
\subsection{Quantitative objects}\label{subsec-quant-object}
In order to study the  properties of quantitative $K$-theory, we introduce the following concept of quantitative object.

\begin{definition}
 A quantitative object is a family   $\O=(O^{\eps,r})_{0<\eps<1/100,r>0}$ of abelian groups,
together with a family of
group homomorphisms  $$\iota_{\O}^{\eps,\eps',r,r'}:O^{\eps,r}\to O^{\eps',r'}$$ for $0<\eps\lq \eps'<1/100$ and $0<r\lq r'$ such that
\begin{itemize}
 \item $\iota_{\O}^{\eps,\eps,r,r}=Id_{O^{\eps,r}}$ for any  $0<\eps<1/100$ and $r>0$;
 \item $\iota_{\O}^{\eps',\eps'',r',r''}\circ\iota_{\O}^{\eps,\eps',r,r'}=\iota_{\O}^{\eps,\eps'',r,r''}$ for any $0<\eps\lq \eps'\lq \eps''<1/100$ and $0<r\lq r'\lq r''$;
\end{itemize}
\end{definition}
As for structure maps attached to quantitative $K$-theory,
when some of the indices $r,r'$ or $\eps,\eps'$ are equal, we shall not
repeat it in $\iota_\O^{\eps,\eps',r,r'}$.
\begin{example}\
Of course, our prominent example will be quantitative $K$-theory $\K_*(A)=(K^{\eps,r}_*(A))_{0<\eps<1/100,r>0}$ of a filtered $C^*$-algebras $A$  with structure maps $\iota_*^{\eps,\eps',r,r'}:K^{\eps,r}_*(A)\lto K^{\eps',r'}_*(A)$.
%
\end{example}
In order to define  the relevant notion of morphisms between quantitative objects, we need to introduce control pairs.
 \begin{definition} A control pair  is a pair $(\lambda,h)$, where
\begin{itemize}
 \item $\lambda \gq 1$;
\item  $h:(0,\frac{1}{100\lambda})\to [1,+\infty);\, \eps\mapsto h_\eps$   is a non-increasing map.
\end{itemize}\end{definition}
 The set of control pairs is equipped with a partial order:
 $(\lambda,h)\lq (\lambda',h')$ if $\lambda\lq\lambda'$ and $h_\eps\lq h'_\eps$
for all $\eps\in (0,\frac{1}{100\lambda'})$. If $(\lambda,h)$ and $(\lambda',h')$ are two control pairs, define
$$h*h': \left(0,\frac{1}{100\lambda\lambda'}\right)\to (0,+\infty);\, \eps\mapsto h_{\lambda'\eps}h'_\eps.$$
Then $(\lambda\lambda',h*h')$  is again  a control pair.

\begin{definition}
 Let $(\lambda,h)$ be a control pair and let  $\O=(O^{\eps,r})_{0<\eps<1/100,r>0}$ and $\O'=(O'^{\eps,r})_{0<\eps<1/100,r>0}$ be
quantitative objects. A $(\lambda,h)$-controlled morphism
$$\F:\O\to\O'$$ is a family $\F=(F^{\eps,r})_{0<\eps< \frac{1}{100\lambda},r>0}$ of  groups homomorphisms
 $$F^{\eps,r}:\O^{\eps,r} \to \O'^{\lambda\eps,h_\eps r}$$ such that for any positive
numbers $\eps,\,\eps',\,r$ and $r'$ with
$0<\eps\lq\eps'< \frac{1}{100\lambda},\, r\lq r'$ and $h_\eps r\lq h_{\eps'}r'$, we have
$$F^{\eps',r'}\circ \iota_\O^{\eps,\eps',r,r'}=\iota_{\O'}^{\lambda\eps,\lambda\eps', h_\eps r,h_{\eps'}r'}\circ F^{\eps,r}.$$
\end{definition}
When  it is not necessary  to specify the control pair, we  just say that $\F$ is a controlled morphism.
If $\O=(O^{\eps,r})_{0<\eps<1/100,r>0}$ is a quantitative object, let us define the identity $(1,1)$-controlled morphism
$\Id_\O=(Id_{O^{\eps,r}})_{0<\eps< 1/100,r>0}:\O\to\O$. Observe that if $A$ and $B$ are filtered
$C^*$-algebras and if $\F:\K_*(A)\to\K_*(B)$ is a $(\lambda,h)$-controlled morphism, then $\F$ induces
 a  morphism $F:K_*(A)\to K_*(B)$ uniquely defined
by $\iota^{\eps,r}_*\circ F^{\eps,r}=F\circ\iota^{\eps,r}_*$.

In some situation,  as  for instance  controlled boundary maps of controlled Mayer-Vietoris pair (see Section \ref{subsec-MV-pairs}), we deal with  family 
$F^{\eps,r}:\O^{\eps,r} \to \O'^{\lambda\eps,h_\eps r}$ of group morphisms  defined indeed only up to  a certain order.
\begin{definition}
 Let $(\lambda,h)$ be a control pair, let  $\O=(O^{\eps,s})_{0<\eps<1/100,s>0}$ and $\O'=(O'^{\eps,s})_{0<\eps<1/100,s>0}$ be
quantitative objects and let $r$ be a positive number.
A $(\lambda,h)$-controlled morphism of order $r$ 
$$\F:\O\to\O'$$ is a family $\F=(F^{\eps,s})_{0<\eps< \frac{1}{100\lambda},0<s<\frac{r}{h_\eps}}$ of  groups homomorphisms
 $$F^{\eps,s}:\O^{\eps,s} \to \O'^{\lambda\eps,h_\eps s}$$ such that for any positive
numbers $\eps,\,\eps',\,s$ and $s'$ with
$0<\eps\lq\eps'< \frac{1}{100\lambda},\,s\lq s'$ and $h_\eps s\lq h_{\eps'}s'<r$, we have
$$F^{\eps',s'}\circ \iota_\O^{\eps,\eps',s,s'}=\iota_{\O'}^{\lambda\eps,\lambda\eps', h_\eps s,h_{\eps'}s'}\circ F^{\eps,s}.$$
\end{definition}
 Let   $\O=(O^{\eps,r})_{0<\eps<1/100,r>0}$, $\O'=(O'^{\eps,r})_{0<\eps<1/100,r>0}$ and $\O''=(O''^{\eps,r})_{0<\eps<1/100,r>0}$
be quantitative objects, let $$\F=(F^{\eps,r})_{0<\eps<\frac{1}{100\alpha},r>0}:\O\to\O'$$
 be
a $(\alpha,k)$-controlled morphism, and let
 $$\G=(G^{\eps,r})_{0<\eps<\frac{1}{100\alpha'},r>0}:\O'\to\O''$$
be a $(\alpha',k')$-controlled morphism. Then $\G\circ\F:\O\to \O''$ is the $(\alpha'\alpha,k'*k)$-controlled
morphism defined by the family
\begin{equation}\label{equ-composition}(G^{\alpha\eps,k_{\eps}r}\circ F^{\eps,r}:O^{\eps,r}\to
{O''}^{\alpha'\alpha\eps,k_{\eps}k'_{\alpha\eps}r})_{0<\eps<\frac{1}{100\alpha\alpha'},r>0}.\end{equation}Notice that if  
  $\F:\O\to\O'$ and     $\G:\O'\to\O''$  are respectively
a $(\alpha,k)$-controlled morphism  and  $(\alpha',k')$-controlled morphism of order $r$,  
then equation (\ref{equ-composition}) defines a  $(\alpha'\alpha,k'*k)$-controlled
morphism  $\G\circ\F:\O\to \O''$ of order $r$.

\begin{notation}
  Let   $\O=(O^{\eps,r})_{0<\eps<1/100,r>0}$ and $\O'=(O'^{\eps,r})_{0<\eps<1/100,r>0}$ be  quantitative objects,
let $\F=(F^{\eps,r})_{0<\eps<\frac{1}{100\al},r>0}:\O\to\O'$ (resp. $\G=(G^{\eps,r})_{\frac{1}{100\al'},r>0}:\O\to\O'$) be an
$(\alpha,k)$-controlled morphism (resp. an $(\alpha',k')$-controlled morphism) and let $(\la,h)$ be a control pair. We write
$\F\aeq\G$ if
\begin{itemize}
 \item $(\alpha,k)\lq (\lambda,h)$ and $(\alpha',k')\lq (\lambda,h)$.
\item for every $\eps$ in $(0,\frac{1}{100\lambda})$ and $r>0$, we have
$$\iota^{\alpha\eps,\lambda\eps,k_{\eps}r,h_\eps r}_{\O'}\circ F^{\eps,r}=\iota^{\alpha'\eps,\lambda\eps,k'_{\eps}r,h_\eps r}_{\O'}\circ G^{\eps,r}.$$
\end{itemize}
\end{notation}
%
%
%
%

Recall from \cite{oy3} the definition of  a controlled isomorphisms.

 \begin{definition}
  Let $(\lambda,h)$ be a control pair, and let $\F:\O\to\O'$  be a $(\alpha,k)$-controlled morphism with
$(\alpha,k)\lq (\lambda,h)$.
 $ \F$ is called  $(\lambda,h)$-invertible or a $(\lambda,h)$-isomorphism if there exists a controlled morphism
$\G:\O'\to\O$ such that
$\G\circ\F\aeq \Id_{\O}$ and
 $\F\circ\G\aeq \Id_{\O'}$. The controlled morphism $\G$ is called a $(\lambda,h)$-inverse for $\G$.
 \end{definition}

In particular, if $A$ and $B$ are filtered
$C^*$-algebras and if $\F:\K_*(A)\to\K_*(B)$ is a $(\lambda,h)$-isomorphism, then the induced  morphism $F:K_*(A)\to K_*(B)$  is an isomorphism
and its inverse is induced by  a controlled morphism (indeed induced by any $(\lambda,h)$-inverse for $\F$).

In order to define the quantitative index map, we need to consider the special case of controlled morphisms with constant source (see Section \ref{subsec-quantitative-index}).
\begin{definition}\label{def-quant-morphism}Let $Z$ be  an abelian  group and let $\O=(O^{\eps,s})_{0<\eps<1/100,s>0}$ be a quantitative object. A  projective morphism 
 $\F:Z\to \O$ is a family of group homomorphisms $$\F=(F^{\eps,s}:Z\to\O^{\eps,s})_{0<\eps<1/100,s>0}$$
such that for every positive numbers $\eps,\,\eps'$, $s$ and $s'$ with $\eps\lq\eps'<1/100$ and $s\lq s'$, then
$F^{\eps',s'}=\iota_\O^{\eps,\eps',s,s'}\circ F^{\eps,s}$.
\end{definition}
\begin{remark}\label{rem-quant-morphism-small-propagation}
Let 
$$\F=(F^{\eps,s}:Z\to\O^{\eps,s})_{0<\eps<1/100,s>0}$$ be a projective morphism.
\begin{enumerate}
\item  let us fix $\eps_0$ in $(0,1/100)$ and $s_0>0$. Then $\F$ is completely determined
by the knowledge of
$(F^{\eps,s})_{0<\eps<\eps_0,0<s<s_0}$.
\item  for any fixed $\eps$ in $(0,1/100)$, the projective morphism  $F$ induces a morphism
\begin{eqnarray*}
Z&\lto &\varprojlim_{s>0}O^{\eps,s}\\ 
x&\mapsto&(F^{\eps,s}(x))_{s>0},
\end{eqnarray*}
where the right hand side is the projective limit with respect to the inductive system $(\iota_\O^{\eps,s,s'})_{0<s\lq s'}$.
\end{enumerate}
\end{remark}
 If $\Lambda:Z'\to Z$ is a homomorphism of abelian group, then $$\F\circ\Lambda\defi(F^{\eps,s}\circ\Lambda)_{0<\eps<1/100,s>0} :Z\lto\O$$ is a projective morphism and if $\G:\O\to\O'$ is a $(\al,k)$-controlled morphism, then
  $$\G\circ\F\defi \left(G^{\frac{\eps}{\alpha},\frac{s}{k_{{\eps}/{\alpha}}}}\circ F^{\frac{\eps}{\alpha},\frac{s}{k_{{\eps}/{\alpha}}}}\right)_{0<\eps<1/100,s>0}:Z\lto\O'$$ is a  projective morphism.
  In order to study the behaviour of  quantitative objects for small propagation, we introduce the following notion of $h$-isomorphism.
 \begin{definition}\label{def-h-isomorphism}
  Let $(0,1/200)\to [1,+\infty);\, \eps\mapsto h_\eps$ be a non decreasing function. A  projective morphism 
 $$\F=(F^{\eps,s})_{0<\eps<1/100,s>0}:Z\to \O$$ is called a $h$-isomorphism if for any $\eps$ in $(0,1/200)$,
 \begin{itemize}
 \item $F^{\eps,s}:Z\to\O^{\eps,r}$ is one-to-one for any positive number $r$ with $r<1/h_\eps$;
 \item For any positive number $r$ with $r<1/h^2_\eps$ and any $y$ in $O^{\eps,r}$, there exists $x$ in $Z$ such that
 $F^{\eps,h_\eps r}(x)=\iota_\O^{\eps,r,h_\eps r}(y)$.
 \end{itemize}
 In other words, this means that $F^{\eps, r}$ is one-to-one onto $\iota_\O^{\eps,r/h_\eps,r}(O^{\eps,r/h_\eps})$ for any positive number $r$ with $r<1/h_\eps $.
 \end{definition}
 The restriction  to $(0,1/200)$ of the range of the control is for controlled rescaling reasons that will become clear later on  when stating the properties of the quantitative index map.
 The meaning of this definition is  that up to a rescaling, $Z$ and $O^{\eps,r}$ are isomorphic for small propagation.   In view of the second point of Remark \ref{rem-quant-morphism-small-propagation}, this implies in particular that for any  $\eps$ in $(0,1/200)$, the projective morphism $\F$ induces an isomorphism
 $$Z\stackrel{\simeq}{\lto} \varprojlim_{s>0}O^{\eps,s}.$$




Let us recall now the definition of the controlled exact sequence for quantitative objects. This controlled exact sequence is
an important tool in computing quantitative $K$-theory for filtered $C^\ast$-algebras.

\begin{definition}\label{def-exact-seq}Let $(\lambda,h)$ be a control pair,
\begin{itemize}
 \item
Let $\O=(O_{\eps,s})_{0<\eps< \frac{1}{100},s>0},\, \O'=(O'_{\eps,s})_{0<\eps< \frac{1}{100},s>0}$  and
$\O''=(O''_{\eps,s})_{0<\eps< \frac{1}{100},s>0}$ be quantitative objects, let
$$\F=(F^{\eps,s})_{0<\eps< \frac{1}{100\al},s>0}:\O\to\O'$$ be a  $(\alpha,k)$-controlled morphism and
  let   $$\G=(G^{\eps,s})_{0<\eps<\frac{1}{100\alpha'},s>0}:\O'\to \O''$$
be a $(\alpha',k')$-controlled morphism with $(\alpha',k')\lq (\lambda,h)$.
Then the composition
$$\O\stackrel{\F}{\to}O'\stackrel{\G}{\to}\O''$$ is said to be $(\lambda,h)$-exact at $\O'$ if
 $\G\circ\F=0$ and if
 for any $0<\eps<\frac{1}{100\lambda\alpha}$, any $s>0$  and  any $y$ in
$O'^{\eps,s}$  such that  $G^{\eps,s}(y)=0$ in $O''_{\al'\eps,k'_\eps r}$, there exists an
  element $x$ in $O^{\lambda\eps,h_\eps s}$
such that $$F^{\lambda\eps,h_{\eps}s}(x)=\iota_{\O'}^{\eps,\alpha\lambda\eps,r,k_{\lambda\eps}h_\eps s}(y)$$ in
$O'^{\alpha\lambda\eps,k_{\lambda\eps}h_\eps s}$.
\item  A sequence of controlled morphisms
$$\cdots\O_{k-1}\stackrel{\F_{k-1}}{\lto}\O_{{k}}\stackrel{\F_{k}}{\lto}
\O_{{k+1}}\stackrel{\F_{k+1}}{\lto}\O_{{k+2}}\cdots$$ is called $(\lambda,h)$-exact if for every $k$,
the composition   $$\O_{{k-1}}\stackrel{\F_{k-1}}{\lto}\O_{{k}} \stackrel{\F_{k}}{\lto}
\O_{{k+1}}$$ is $(\lambda,h)$-exact at $\O_{{k}}$.
\item{a controlled morphism $\F:\O\to\O'$ is called $\mathbf{(\la,h)}${\bf-injective} if $$0\to \O\stackrel{\F}{\to}\O'$$ is
$(\la,h)$-exact;}
\item {a controlled morphism $\F:\O\to\O'$ is called $\mathbf{(\la,h)}${\bf-surjective } if $$\O\stackrel{\F}{\to}\O'\to 0$$ is
$(\la,h)$-exact.}
\end{itemize}\end{definition}

It is then obvious that if a controlled morphism $\F$ is a $(\la,h)$-isomorphism, then there exists a controlled pairs $(\la',h')$  depending only on $(\la,h)$ such that
$\F$  is $(\la',h')$-injective and $(\la',h')$-surjective. Conversely, if
$\F$ is $(\la',h')$-injective and $(\la'',h'')$-surjective, then there exists a controlled pair 
$(\la,h)$ depending only on $(\la',h')$ and $(\la'',h'')$ such that $\F$ is a $(\la,h)$-isomorphism.

The notion of $(\lambda,h)$-exactness of a composition and of a sequence can  be extended to the setting of  controlled morphism of order $r$.
\begin{definition}Let $(\lambda,h)$ be a control pair,

Let $\O=(O_{\eps,s})_{0<\eps< \frac{1}{100},s>0},\, \O'=(O'_{\eps,s})_{0<\eps< \frac{1}{100},s>0}$  and
$\O''=(O''_{\eps,s})_{0<\eps< \frac{1}{100},s>0}$ be quantitative objects, let
$$\F=(F^{\eps,s})_{0<\eps< \frac{1}{100\al},0<s<\frac{r_1}{h_\eps}}:\O\to\O'$$ be a  $(\alpha,k)$-controlled morphism of order $r_1$ and
  let   $$\G=(G^{\eps,s})_{0<\eps<\frac{1}{100\alpha'},0<s<\frac{r_2}{h_\eps}}:\O'\to \O''$$
be a $(\alpha',k')$-controlled morphism order $r_2$ with $(\alpha',k')\lq (\lambda,h)$.
For $r$ in $(0,\min\{r_1,r_2\})$, the composition
$$\O\stackrel{\F}{\to}O'\stackrel{\G}{\to}\O''$$ is said to be $(\lambda,h)$-exact at $\O'$ and at order $r$ if
\begin{itemize}
\item for any $0<\eps<\frac{1}{100\lambda\alpha}$, any  $0<s<\frac{r}{k_\eps h_{\al\eps}}$ and any $x$ in
$\O_{\eps,s}$, then $G^{\al\eps,k_{\al \eps} s}\circ F^{\eps,s}(x)=0$ in $\O''_{\al\al'\eps,k_\eps k'_{\al\eps}s}$.
\item for any $0<\eps<\frac{1}{100\lambda\alpha}$, any  $0<s<\frac{r}{k_{\la\eps} h_{\eps}}$
    and  any $y$ in
$O'^{\eps,s}$  such that  $G^{\eps,s}(y)=0$ in $O''_{\al'\eps,k'_\eps r}$, there exists an
  element $x$ in $O^{\lambda\eps,h_\eps s}$
such that $$F^{\lambda\eps,h_{\eps}s}(x)=\iota_{\O'}^{\eps,\alpha\lambda\eps,r,k_{\lambda\eps}h_\eps s}(y)$$ in
$O'^{\alpha\lambda\eps,k_{\lambda\eps}h_\eps s}$.
\end{itemize}
\end{definition}

\subsection{Six-term controlled  exact sequence}

Examples of controlled exact sequences in quantitative $K$-theory are provided by controlled six term exact sequences associated to a completely filtered extensions of $C^*$-algebras \cite[Section 3]{oy2}.

\begin{definition}\label{def-completely-filtered-extension}
Let $A$ be a $C^*$-algebra filtered by $(A_r)_{r>0}$, let $J$ be an
ideal of $A$ and set $J_r=J\cap A_r$. The  extension of $C^*$-algebras $$0\to J\to A\to A/J\to 0$$ is called a completely filtered extension of $C^*$-algebras  if 
 the bijective continuous linear map $$A_r/J_r\lto (A_r+J)/J$$ induced by the inclusion $A_r\hookrightarrow\ A$  is  a complete  isometry i.e for any integer $n$, any positive number $r$ and any $x$ in $M_n(A_r)$,   then  $$\inf_{y\in M_n(J_r)}\|x+y\|=\inf_{y\in M_n(J)}\|x+y\|.$$\end{definition}
 Notice that  in this case, the ideal $J$ is  filtered by $(A_r\cap J)_{r>0}$  and $A/J$ is filtered by $(A_r+J)_{r>0}$.
 A particular case of completely   filtered extension of $C^*$-algebra is the case of a filtered and semi-split extension of $C^*$-algebras \cite[Lemma 3.3]{oy2} (or a semi-split extension of filtered  algebras) i.e extension 
$$0\to J\to A\to A/J {\to} 0,$$ such that
\begin{itemize}
\item $A$ is filtered by $(A_r)_{r>0}$;
\item   there
exists a completely positive  and completely contractive (if $A$ is not unital)  cross-section  $$s:A/J\to A$$ such that
$$s(A_r+J)\subseteq A_r$$  for any
number $r>0$. 
\end{itemize}
For any extension of $C^*$-algebras $$0\to J\to A\to A/J\to 0$$  we denote by $\partial_{J,A,*}:K_{*}(A/J)\to K_{*}(J)$ the associated
(odd degree) boundary map.
\begin{proposition}\label{prop-bound}
There exists a control pair $(\alpha_\DD,k_\DD)$  such that for any      completely filtered extension of   $C^*$-algebras
 $$0\longrightarrow J \longrightarrow A \stackrel{q}{\longrightarrow}
A/J\longrightarrow 0,$$ there exists an $(\alpha_\DD,k_\DD)$-controlled  morphism of odd degree
 $$\DD_{J,A,*}=(\partial_ {J,A,*}^{\eps,r})_{0<\eps<\frac{1}{100\alpha_\DD},r>0}: \K_*(A/J)\to
  \K_*(J)$$
 which induces   $\partial_{J,A,*}:K_{*}(A/J)\to K_{*}(J)$ in $K$-theory.
\end{proposition}
Moreover  the controlled boundary map enjoys the expected naturality properties with respect to completely filtered extensions.
We have in the setting of  controlled exactness  the analogue of the six terms exact  sequence in $K$-theory.
\begin{theorem}\label{thm-six term}  There exists a control  pair $(\lambda,h)$  such that for any  completely  filtered extension of    $C^*$-algebras $$0  \longrightarrow J
\stackrel{\jmath}{\longrightarrow} A \stackrel{q}{\longrightarrow}
A/J\longrightarrow 0,$$   the following six-term sequence is $(\lambda,h)$-exact
$$\begin{CD}
\K_0(J) @>\jmath_*>> \K_0(A)  @>q_*>>\K_0(A/J)\\
    @A\DD_{J,A,*} AA @.     @V\DD_{J,A,*} VV\\
\K_1(A/J) @<q_*<< \K_1(A)@<\jmath_*<< \K_1(J)
\end{CD}
$$
\end{theorem}

Moreover, by construction  (see \cite[Section 3.2]{oy2}), the controlled boundary map  is compatible with the rescaling defined  in Section \ref{subsec-rescaling}.
\begin{proposition}\label{prop-compatibility-boundary-rescaling}
For any  completely  filtered extension of    $C^*$-algebras $$0  \longrightarrow J
\stackrel{\jmath}{\longrightarrow} A \stackrel{q}{\longrightarrow}
A/J\longrightarrow 0,$$  if we set $(\al,k)=(\al_\D,k_\D)$ and $l=l_{rs}$ (with notation of Section \ref{subsec-rescaling}, then we have
\begin{eqnarray*}
\partial_ {J,A,*}^{\eps/2,l_\eps r}\circ L_{A/J,*}^{\eps,r}&=&\iota_*^{-,\al\eps/2,k_{\eps/2}l_\eps r}\circ L_{J,*}^{\al\eps,k_\eps r}\circ \partial_ {J,A,*}^{\eps,r}\\
&=& L_{J,*}^{\al\eps,\frac{k_{\eps/2} l_\eps r}{l_{\al\eps}}}\circ \partial_ {J,A,*}^{\eps,\frac{k_{\eps/2} l_\eps r}{k_{\eps}l_{\al\eps}}}\circ\iota_{*}^{\eps,r,-}
\end{eqnarray*} 
for any $\eps$ in $\left(0,\frac{1}{200\al}\right)$ and any $r>0$.
\end{proposition}

 \subsection{Tensorisation in $KK$-theory and controlled morphisms}
 Let $A$ be a $C^*$-algebra and let $B$ be a filtered $C^*$-algebra filtered by $(B_r)_{r>0}$.  Let us define $A\ts B_r$ as the closure of the algebraic tensor product of $A$ and $B_r$ in the spatial tensor
product $A\ts B$. Then the $C^*$-algebra $A\ts B$ is filtered by $(A\ts B_r)_{r>0}$. 
 If $f:A_1\to A_2$ is a homomorphism of $C^*$-algebras, let us set 
 $$f_B:A_1\ts B\to A_2\ts B;\, a\ts b\mapsto f(a)\ts b.$$ Recall from \cite{kas} that for any $C^*$-algebras  $A_1$, $A_2$ and $B$, G. Kasparov defined a tensorization map
$$\tau_B:KK_*(A_1,A_2)\to KK_*(A_1\ts B,A_2\ts B).$$ If $B$ is a filtered $C^*$-algebra, then for any $z$ in $KK_*(A_1,A_2)$  the morphism
$$K_*(A_1\ts B)\lto K_*(A_2\ts B);\, x\mapsto x\ts_{A_1\ts B}\tau_B(z)$$  is induced by a control morphism  \cite[Theorem 4.4]{oy2}.
  \begin{theorem}\label{thm-tensor}
  There exists a control pair $(\alpha_\TT,k_\TT)$ with $k_\TT>1$ such that 
  \begin{itemize}
  \item for any filtered   $C^*$-algebra $B$; 
  \item for any $C^*$-algebras $A_1$ and $A_2$;
  \item for any element $z$ in $KK_*(A_1,A_2)$,
  \end{itemize}
  There exists an $(\alpha_\TT,k_\TT)$-controlled morphism  
  $$\TT_B(z)=(\tau_{B}^{\eps,r}(z))_{\eps\in\left(0,\frac{1}{100\al_\T}\right),r>0}:\K_*(A_1\ts B)\to \K_*(A_2\ts B)$$  with same degree as $z$ that satisfies the following:
 \begin{enumerate}
 \item $\TT_B(z):\K_*(A_1\ts B)\to \K_*(A_2\ts B)$ induces the right multiplication by $\tau_B(z)$ in $K$-theory;
\item  For any  elements $z$ and $z'$ in $KK_*(A_1,A_2)$. Then
$$\TT_B(z+z')=\TT_B(z)+\TT_B(z').$$
\item Let $A'_1$ be a filtered $C^*$-algebra  and  let $f:A_1\to A'_1$ be a homomorphism of  $C^*$-algebras. Then we have 
$\T_B(f^*(z))=\T_B(z)\circ f_{B,*}$  for all $z$ in $KK_* (A'_1,A_2)$.
\item Let $
A'_2$ be a $C^*$-algebra and  let $g:A'_2\to A_2$ be a homomorphism of filtered $C^*$-algebras. Then we have
$\TT_B(g_{*}(z))=g_{B,*}\circ \TT_B(z)$
for any $z$ in $KK_*(A_1,A'_2)$.
\item $\TT_B([Id_{A_1}])\stackrel{(\alpha_\TT,k_\TT)}{\sim} \Id _{\K_*({A_1}\ts B)}$.
\item For any $C^*$-algebra $D$ and any element $z$ in $KK_*(A_1,A_2)$, we have  $\TT_B(\tau_D(z))=\TT_{B\ts D}(z)$.
\item For any  semi-split extension of  $C^*$-algebras $0\to J\to A\to A/J\to 0$    with corresponding   element   $[\partial_{J,A}]$  of
$KK_1(A/J,J)$
 that implements the boundary map,   we have
$$\TT_B([\partial_{J,A}])=\DD_{J\ts B,A\ts B}.$$\end{enumerate}
\end{theorem}
Moreover, $\TT_B$ is compatible with Kasparov products \cite[Theorem 4.5]{oy2}.
\begin{theorem}\label{thm-product-tensor} There exists a control pair $(\lambda,h)$ such that the following holds :

  let  $A_1,\,A_2$ and $A_3$  be separable $C^*$-algebras and let $B$ be a  filtered $C^*$-algebra. Then for any
$z$ in $KK_*(A_1,A_2)$ and  any  $z'$ in
$KK_*(A_2,A_3)$, we have
$$\TT_B(z\ts_{A_2} z')\aeq  \TT_B(z') \circ\TT_B(z).$$
\end{theorem}

We also have the following naturallity  result with respect to the filtered $C^*$-algebras.
\begin{proposition}\label{prop-naturality}
Let $A_1$ and $A_2$ be separable $C^*$-algebras, let $B_1$ and $B_2$ be filtered $C^*$-algebras
and let $f:B_1\to B_2$ be a morphism of filtered $C^*$-algebras.  Let us consider 
$f_{A_1}:A_1\ts B_1\to A_1\ts B_2$ and $f_{A_2}:A_2\ts B_1\to A_2\ts B_2$  the morphisms of filtered $C^*$-algebra induced by $f$ on tensor product. Then $$f_{A_2,*}\circ \TT_{B_1}(z)=\TT_{B_2}(z)\circ f_{A_1,*}$$  for any $z$ in $KK_*(A_1,A_2)$.
\end{proposition}
\begin{remark}\label{remark-naturality}
We can extend Proposition \ref{prop-naturality} to the slightly more general following situation. Let  $f:B_1\to B_2$  be a homomorphism of $C^*$-algebras and assume that there exist a positive number $\th$  with $\th\gq 1$ and a non negative number $\beta$ such
that $f(B_{1,r})\subseteq B_{2,\th r+\beta}$ for all positive number $r$. Then $f$ induces for every positive numbers $\eps$ and $r$ with $\eps<1/100$ and for any $C^*$-algebra $A$ a  morphism $$f^{\eps,r}_{A,*}:K^{\eps,r}_*(A\ts B_1)\lto K^{\eps,\th r+\beta}_*(A\ts B_2).$$
With notations  of Proposition \ref{prop-naturality}, we have with $(\al,k)=(\al_\T,k_\T)$
$$\iota_*^{-,\al\eps,k_{\eps}(\th r+\beta)}\circ f^{\al\eps,k_{\eps} r}_{A_2,*}\circ \tau_{B_1}^{\eps,r}(z)=\tau_{B_2}^{\eps,\th r+\beta}(z)\circ f^{\eps,r}_{A_1,*}$$
for all positive numbers $\eps$ and $r$ with $\eps<\frac{1}{100\al}$.
\end{remark}

By construction of  $\TT_B$ (see \cite[Section 3.2]{oy2}) and in view of Proposition
\ref{prop-compatibility-boundary-rescaling}, this quantitative version of Kasparov transformation is compatible with the rescaling defined  in Section \ref{subsec-rescaling}.
\begin{proposition}\label{prop-compatibility-kasparov-rescaling}

Let us  set $(\al,k)=(\alpha_\TT,k_\TT)$ and $l=l_{rs}$ (with notation of Section \ref{subsec-rescaling}. Then for any filtered   $C^*$-algebra $B$,
   for any $C^*$-algebras $A_1$ and $A_2$ and any 
 element $z$ in $KK_*(A_1,A_2)$, we have
\begin{eqnarray*}
\tau^{\eps/2,l_\eps r}_B(z)\circ L_{A_1\ts B,*}^{\eps,r}&=&\iota_*^{-,\al\eps/2,k_{\eps/2}l_\eps r}\circ L_{A_2\ts B,*}^{\al\eps,k_\eps r}\circ \tau_ {B}^{\eps,r}(z)\\
&=& L_{A_2\ts B,*}^{\al\eps,\frac{k_{\eps/2} l_\eps r}{l_{\al\eps}}}\circ \tau_ {B}^{\eps,\frac{k_{\eps/2} l_\eps r}{k_{\eps}l_{\al\eps}}}(z)\circ\iota_{*}^{\eps,r,-}
\end{eqnarray*} 
for any $\eps$ in $\left(0,\frac{1}{200\al}\right)$ and any $r>0$.
\end{proposition}

\section{Quantitative assembly maps estimates and coarse decomposition}

In \cite{oy3}, the authors developed  a geometric version of the controlled Baum-Connes assembly maps (see \cite{oy2}) and of the related   quantitative statements. It was proved that if   the family of finite subsets of a proper discrete metric space  $\Si$ satisfies  uniformly these geometric quantitative statements, then $\Si$ satisfies the coarse Baum-Connes conjecture. Due to the   descent principle \cite{roe}, in the particular case  of a finitely generated group with finite type classifying space, this implies the Novikov conjecture. In this section, we recall the setting for these geometric statement and we state the main result of this paper: uniform geometric quantitative statements for families of finite metric spaces with uniformly bounded  geometry are  stable under coarse decompositions. We apply this result to metric spaces with finite geometric complexity introduced in \cite{gty2}. We prove that the family of finite subsets of such a metric space with bounded geometry satisfies
uniformly the geometric quantitative statements.

\subsection{Geometric controlled assembly maps}\label{sub-sec-geo-ass-map}
Let $\Si$ be a proper discrete metric space. For any $C^*$-algebra $A$, set  $A_\Si=A\ts \Kp(\ell^2(\Si))$.
Then the distance $d$ on $\Si$ induces  a filtration on $A_{\Si}$ in the following way: let $r$ be a positive number and
 $T=(T_{\si,\si'})_{(\si,\si')\in\Si^2}$ be an element in $A_{\ell^2(\Si)}$, with $T_{\si,\si'}$ in $A$ for any $\si$ and $\si'$ in $\Si^2$. Then 
 $T$ has propagation   $r$ if $T_{\si,\si'}=0$ for $\si$ and $\si'$ in $\Si$ such that $d(\si,\si')>r$. We define 
 the  Rips complex of degree $d$ of $\Si$ as  the set $P_d(\Si)$ of probability measure on $\Si$ with support of diameter less than $d$. Then $P_d(\Si)$ is a locally finite simplicial complex and  is locally compact when endowed   with the simplicial topology.  Let us define then $$K_*(P_d(\Si),A)\defi \lim_{Z\subseteq P_d(\Si)} KK_*(C(Z),A),$$ where $Z$ runs through compact subsets of $P_d(\Si)$.
In \cite{oy3} was  constructed  for any  $C^*$-algebra $A$ local quantitative coarse assembly maps 
$$\nu_{\Si,A,*}^{\eps,r,d}: K_*(P_d(\Si),A)\to
K_*^{\eps,r}(A_{\Si}),$$ with $d>0,\,\eps\in(0,1/100)$ and $r\gq r_{d,\eps}$, where the map $$(0,+\infty)\times (0,1/100)\to(0,+\infty);(\eps,d)\mapsto r_{d,\eps}$$ is non increasing in $\eps$,  non decreasing in $d$ and satisfies $ r_{d,\eps}>2d$ for all positive numbers $\eps$ and $d$ with $\eps<1/100$. Let us recall now the definition of these quantitative assembly maps: 

\medskip

Notice first  that  any $x$ in $P_d(\Si)$ can be written as a finite convex combination
$$x=\sum_{\si\in\Si}\lambda_\si(x)\de_\si$$ where
\begin{itemize}
 \item $\delta_\si$ is the Dirac probability measure  at $\si$ in $\Si$.
\item for every $\si$ in $\Si$, the coordinate function $$\lambda_\si=\lambda_{\Si,\si}: P_d(\Si)\to[0,1]$$ is continuous and its  support is the closure of the set of probability measures on $\Si$ for which $\si$ is in the support.\end{itemize}
Let $X$ be a compact  subset of $P_d(\Si)$.
Let us define $$P_X:C(X)\ts\ell^2(\Si)\to C(X)\ts\ell^2(\Si)$$  by
\begin{equation}\label{equ-def-PX}(P_X\cdot h)(x,\si)=\lambda^{1/2}_\si(x)\sum_{\si'\in\Si}h(x,\si')\lambda^{1/2}_{\si'}(x),\end{equation} for any $h$ in $C(X)\ts\ell^2(\Si)$.
Since $\sum_{\si\in\Si}\lambda_\si=1$, it is straightforward to check that $P_X$ is a   projection 
in $C(X)\ts \Kp(\ell^2(\Si))=C(X)_{\Si}$ with  propagation less than $2d$ 
and hence gives rise in particular to a class $[P_X]$ in  $K_0^{\eps,r}(C(X)_{\Si})$ for any $\eps$ in $(0,1/100)$ and any positive number $r$ with $r\gq 2d$. Let us set $r_{d,\eps}=2dk_{\eps/\alpha}$, for any positive number $\eps$ and $r$ with $\eps<1/100$, where $(\al,k)=(\al_\T,k_\T)$ is the control pair of Theorem \ref{thm-tensor}.
For  any  $C^*$-algebra $A$, any $\eps\in(0,1/100)$ and any $r\gq r_{d,\eps}$, the map
\begin{eqnarray}\label{eq-quant-ind}
 KK_*(C(X),A)&\longrightarrow& K_*^{\eps,r}(A_{\Si}))\\
\nonumber x&\mapsto &\left(\tau_{\Kp(\ell^2(\Si))}^{\eps/\alpha,r/k_{\eps/\alpha}}(x)\right)
([P_X,0]_{\eps/\al,r/k_{\eps/\alpha}})
\end{eqnarray} is compatible with inductive limit over  
compact subsets of $P_d(\Si)$ and hence gives rise to a local coarse assembly  map
$$\nu_{\Si,A,*}^{\eps,r,d} :K_* (P_d(\Si),A){\longrightarrow}K_*(A_{\Si}).$$

 The  map  $\nu_{\Si,\bullet,*}^{\eps,r,d}$ is natural   for $C^*$-algebras and induces in $K$-theory the index map, i.e the maps $\iota_*^{\eps,r}\circ\nu_{\Si,A,*}^{\eps,r,d}$  is up   to Morita equivalence  given for any compact subset $Z$ of $P_d(\Si)$ by the morphism in the inductive limit $KK_*(C(Z),A)\to K_*(A)$ induced by the map  $Z\to \{pt\}$.
Moreover, the maps  $\nu_{\Si,A,*}^{\bullet,\bullet,\bullet}$ are compatible with  structure morphisms  and with inclusion of Rips complexes:
\begin{itemize}
 \item $\iota_*^{-,\eps',r'}\circ \nu_{\Si,A,*}^{\eps,r,d}=\nu_{\Si,A,*}^{\eps',r',d}$ for any
positive numbers $\eps,\,\eps',\,r,\,r'$ and $d$ such that $\eps\lq\eps'<1/100,\,r_{d,\eps}\lq r,\,r_{d,\eps'}\lq r'$ and
$r\lq r'$;
\item $\nu_{\Si,A,*}^{\eps,r,d'}\circ q_{d,d',*}=\nu_{\Si,A,*}^{\eps,r,d}$ for any
positive numbers $\eps,\,r,\,d$ and $d'$ such that $\eps<1/100,\,d\lq d'$ and $r_{d',\eps}\lq r$, where $$q_{d,d',*}=q^\Si_{d,d',*}:K_*(P_{d}(\Si),A)\to K_*(P_{d'}(\Si),A)$$  is the homomorphism induced by the restriction
  from $P_{d'}(\Si)$ to $P_d(\Si)$.
\end{itemize} 
The construction of the map $\nu_{\Si,A,*}^{\eps,r,d}$ has to be compare with the construction of the coarse assembly map in 
\cite[Section 2.3]{sty}.
\begin{lemma}\label{lemma-subset}
Let $\Sigma$ be a proper discrete  metric space and let $\Sigma'$ be a subset of $\Sigma$.
Let us consider :
\begin{itemize}
\item $\jmath^\sharp_{\Sigma,\Sigma'}:\Kp(\ell^2(\Si'))\hookrightarrow \Kp(\ell^2(\Si))$ the natural inclusion;
\item $\jmath_{P_d(\Si),P_d(\Si'),*}:K_*(P_{d}(\Si'),A)\to K_*(P_d(\Si),A)$   the homomorphism induced by the restriction
  from $P_{d}(\Si)$ to $P_d(\Si')$.
 \end{itemize}
 Then we have the equality $$\nu_{\Si,A,*}^{\eps,r,d}\circ \jmath_{P_d(\Si),P_d(\Si'),*}=\jmath^\sharp_{\Sigma,\Sigma',A,*}\circ \nu_{\Si',A,*}^{\eps,r,d}$$ for any
positive numbers $\eps,\,r,\,d$   such that $\eps<1/100$,  and $r_{d,\eps}\lq r$, where 
$$\jmath^\sharp_{\Sigma,\Sigma',A}:A_{\Si'}\longrightarrow A_{\Si}$$ is the filtered morphism induced by $\jmath^\sharp_{\Si,\Si'}$.
\end{lemma}
\begin{proof}
Let $X$ be a compact subset of $P_d(\Si')\subseteq P_d(\Si)$ and let $P_X$ in $C(X)\ts \Kp(\ell^2(\Si))$  and $P'_X$ 
in $C(X)\ts \Kp(\ell^2(\Si'))$ be  the projections of 
equation (\ref{equ-def-PX}) respectively obtained from $\Si$ and $\Si'$. It  is straightforward to check
that \begin{equation}\label{equation-inclusion}\jmath^\sharp_{\Si,\Si',C(X)}(P'_{X})=P_X.\end{equation} Now the result is a consequence of the naturality in the filtered $C^*$-algebra  of $\TT_\bullet$ stated in Proposition \ref{prop-naturality}.
\end{proof}
Let us set $l=l_{rs}$ (with notation of Section \ref{subsec-rescaling} and let $X$
be a compact subset of $P_d(\Si)$. It is  clear that 
$L_{C(X)\ts \Kp(\ell^2(\Si))}^{\eps,r}([P_X,0]_{\eps,r})=[P_X,0])_{\eps/2,l_\eps r}$ for $r>2d$  and hence according to  Proposition \ref{prop-compatibility-kasparov-rescaling}, the local quantitative assembly maps are compatible with the rescaling map.
\begin{proposition}\label{prop-compatibility-assembly-rescaling}
Let $\Si$ be a proper discrete metric space and let $A$ be a $C^*$-algebra $A$.
For any  positive numbers $\eps$, $r$ and $d$ with $\eps<1/200$ and
$r_{d,\eps}\lq
r$, let us set $k'_\eps=\frac{l_{\eps/\al} k_{\frac{\eps}{2\al}}}{k_{\eps/\al}}$.
Then, we have
$$ \nu_{\Si,A,*}^{\eps/2,k'_\eps r,d}=
\iota_*^{-,\eps/2,k'_\eps r,d}\circ L_{A_\Si,*}^{\eps,r}\circ \nu_{\Si,A,*}^{\eps,r,d}.$$
\end{proposition}
\subsection{Quantitative Assembly Maps estimates}
 Let us consider 
for 
$d,d',r,r'$ and $\eps$ positive numbers with $d\lq d'$, $\eps< 1/100$,  
$r_{d,\eps'}\lq
r'$ and  $ r\lq
r'$ the following statements:
\begin{description}
\item[$QI_{\Si,A,*}(d,d',r,\eps)$]  for any element $x$ in
  $K_*(P_d(\Si),A)$, then 
  $\nu_{\Si,A,*}^{\eps,r,d}(x)=0$
  in $K_*^{\eps,r}(A\ts \Kp(\ell^2(\Si)))$ implies that
  $q_{d,d',*}(x)=0$ in    $K_*(P_{d'}(\Si),A)$.
\item[$QS_{\Si,A,*}(d,r,r',\eps)$] for every $y$
  in  $K_*^{\eps,r}(A_\Si)$, there exists an element $x$ in $K_*(P_d(\Si),A)$
   such that
$$\nu_{\Si,A,*}^{\eps,r',d}(x)=\iota_*^{\eps,r,r'}(y).$$
\end{description}
\begin{definition}\
\begin{itemize}
\item Let $\Si$ be a proper discrete metric space. We say that $\Si$ satisfies the Quantitative Assembly Map  estimates (QAM-estimates)
if
\begin{itemize}
\item for every positive numbers $\eps$, $r$ and $d$ with $\eps<\frac{1}{200}$ and $r\gq r_{d,\eps}$, there exists a positive number $d'$ with $d'\gq d$ such that  $QI_{\Si,A,*}(d,d',r,\eps)$ holds for any $C^*$-algebras $A$;
\item for every positive numbers $\eps$, $r$  with $\eps<\frac{1}{200}$ there exists positive numbers $d$ and $r'$ with
$r_{d,\eps}\lq
r'$ and  $ r\lq
r'$ and such that $QS_{\Si,A,*}(d,r,r',\eps)$ holds for any  $C^*$-algebras $A$;
\end{itemize}
\item Let $\Y$ be a family of proper discrete metric spaces. We say that the family  $\Y$ satisfies uniformly  the  QAM-estimates    if
\begin{itemize}
\item for every positive numbers $\eps$, $r$ and $d$ with $\eps<\frac{1}{200}$ and $r\gq r_{d,\eps}$, there exists a positive number $d'$ with $d'\gq d$ such that  $QI_{\Si,A,*}(d,d',r,\eps)$ holds for any $C^*$-algebras $A$ and any proper discrete metric space $\Si$ in $\Y$;
\item for every positive numbers $\eps$, $r$  with $\eps<\frac{1}{200}$ there exists positive numbers $d$ and $r'$ with
$r_{d,\eps}\lq
r'$ and  $ r\lq
r'$ and such that $QS_{\Si,A,*}(d,r,r',\eps)$ holds for any  $C^*$-algebras $A$ and any proper discrete metric space $\Si$ in $\Y$.
\end{itemize}\end{itemize}
\end{definition}
\begin{remark}
According to Lemma \ref{lemma-rescaling} and Proposition \ref{prop-compatibility-assembly-rescaling}, it  is enough to check the QAM-estimates for a fixed $\eps$ in $(0,1/200)$:

\smallbreak
 Assume that there exists $\eps$ in $(0,1/200)$ such that
 \begin{itemize}
 \item for every positive numbers  $r$ and $d$ with $r\gq r_{d,\eps}$, there exists a positive number $d'$ with $d'\gq d$ such that  $QI_{\Si,A,*}(d,d',r,\eps)$ holds for any $C^*$-algebras $A$;
\item for every positive number $r$, there exists positive numbers $d$ and $r'$ with
$r_{d,\eps}\lq
r'$ and  $ r\lq
r'$ and such that $QS_{\Si,A,*}(d,r,r',\eps)$ holds for any  $C^*$-algebras $A$.
\end{itemize}
Then $\Si$ satisfies the $QAM$-estimates.

We have  a similar statement in the uniform case.
\end{remark}

The next two  results, proved in  \cite[Theorem 4.10 and 4.11]{oy3}, provide numerous examples of discrete proper metric space that satisfy the QAM-estimates.

\begin{theorem}\label{thm-qam-bc}
Let $\Ga$ be a finitely generated group provided with any word metric. Assume that $\Ga$ satisfies the Baum-Connes conjecture. Then $\Ga$ as a metric space satisfies the QAM-estimates.\end{theorem}

Before stating the second  result, we recall the definition of coarse embedability. 
\begin{definition}
A proper  discrete metric  space  $(\Si,d)$   coarsely embeds into a Hilbert space $\H$
if there exist two functions $\rho_\pm:(0,+\infty)\to (0,+\infty)$  and a map $f:\Si\to \H$ such that
\begin{itemize}
\item  $\displaystyle \lim_{x\mapsto +\infty} \rho_\pm(\si)=+\infty$.
\item $\rho_-(d(\si,\si'))\lq \|f(\si)-f(\si')\|\lq \rho_+(d(\si,\si'))$ for all $\si$ and $\si'$ in $\Si$.
\end{itemize}
\end{definition}
Recall that  a discrete metric space has bounded geometry if every ball of a given radius as uniformly bounded cardinal.
\begin{theorem}\label{thm-qam-hilbert}
Let $\Si$ be a discrete metric space with bounded geometry that coarsely embeds into a Hilbert space, 
then $\Si$ satisfy   the QAM-estimates.
\end{theorem}
The QAM-estimates are also related to the Novikov conjecture in the following way \cite[Theorem 5.1]{oy3} .
\begin{theorem}\label{thm-qam-bcc}
Let $(\Si,d)$ be a proper discrete metric space and assume that
the family of finite subsets of $\Si$ satisfies uniformly  the  QAM-estimates.
Then $\Si$  satisfies the coarse Baum-Connes conjecture.
\end{theorem}
In view of the descent principle   \cite{roe}, we have the following consequence.
\begin{corollary}\label{corollary-Novikov}
Let $\Ga$ be a finitely generated group which admits a finite type classifying space  and let us equip $\Ga$  with any word metric.  Assume that 
the family of finite subsets of $\Ga$ satisfies uniformly  the  QAM-estimates.
Then $\Ga$ satisfies the Novikov conjecture.
\end{corollary}

Notice that Theorems \ref{thm-qam-bc}, \ref{thm-qam-hilbert} and \ref{thm-qam-bcc}  were originally stated with a rescaling which is now taken into account by  enlarging   homotopies in the definition of quantitative $K$-theory.


Next we show that for a proper discrete metric space $\Si$ we can reduce QAM-estimates for $\Si$ to uniform  QAM-estimates
for the family of finite subsets of $\Si$.

\begin{proposition}
Let $\Si$ be a  discrete proper metric space  and assume that 
the family of finite subsets of $\Si$ satisfies uniformly the QAM-estimates. Then $\Si$ satisfies the $QAM$-estimates.
\end{proposition}

\begin{proof}
Let us fix $\si_0$ in$\Si$ and set $$\Si_n=\{\si\in \Si\text{ such that }d(\si,\si_0)<n\}.$$ 
According to \cite[Proposition 1.28]{oy2}, the sequence of inclusions 
$$\jmath^\sharp_{\Si_n,\Si}:\Kp(\ell^2(\Si_n))\hookrightarrow \Kp(\ell^2(\Si))$$
induces an isomorphism $$\lim_n  K^{\eps,r}(A_{\Si_n})\stackrel{\simeq}{\longrightarrow} K^{\eps,r}(A_{\Si}).$$ In the same way the sequence of restrictions of $P_d(\Si)$ to $P_d(\Si_n)$ induces an isomorphism
$$\lim_n  K_*(P_d(\Si_n),A )\stackrel{\simeq}{\longrightarrow} K_*(P_d(\Si),A )$$ and  according to Lemma \ref{lemma-subset}, we see that if $QI_{\Si_n,A,*}(d,d',r,\eps)$ holds for every $n$, then $QI_{\Si,A,*}(d,d',r,\eps)$ is also satisfied. We can proceed similarly for the $QS$-statement \end{proof}

\subsection{The main theorem}
In this subsection,  we state the main Theorem of this paper. This theorem will be proved in Section \ref{section-QAM-coarse-decomposability}.
Let us first recall from \cite{gty2} the definition of coarse decomposability for discrete proper metric spaces.
\begin{definition}
Let $\X$ and $\Y$ be families of discrete metric spaces and let $R$ be a positive number.
The family $\X$ is $R$-decomposable with respect to $\Y$ if for any $X$ in $\X$, there exists
$X^{(1)}$ and $X^{(2)}$ two subsets of $X$ such that
\begin{itemize}
\item $X=X^{(1)}\cup X^{(2)}$;
\item for $i=1,2$, then $X^{(i)}$ is the $R$-disjoint union of metric spaces in $\Y$, i.e there exist $(X_j^{(i)})_{j\in J_i}$ a family of metric space in $\Y$ such that  $X^{(i)}=\bigsqcup_{j\in J_i} X_j^{(i)}$ and $d(X_j^{(i)},X_{j'}^{(i)})> R$ if $j\neq j'$ (in this situation, we shall write  $X^{(i)}=_R\bigsqcup_{j\in J_i} X_j^{(i)}$).
\end{itemize}
\end{definition} 
We also need to generalize the concept of bounded geometry to families of discrete proper metric spaces.

\begin{definition}
Let $\X$  be a family of discrete proper metric spaces. Then $\X$ has uniformly bounded geometry if for every 
positive number $r$, there exists an integer  $n_r$ such that for any metric space in $\X$, balls of radius $r$ have cardinal bounded by $n_r$.
\end{definition}
If $\X$ is a family of finite metric spaces, we denote by $\widetilde{\X}$ the family of all subsets of all metric spaces in $\X$.
\begin{theorem}\label{theorem-main}
Let $\X$ be a family of finite metric spaces  and 
assume that for every positive number $R$, there exists a family $\Y$ of finite metric spaces with uniformly bounded geometry such that
\begin{itemize}
\item $\X$ is $R$-decomposable with respect to $\Y$;
\item the family $\widetilde{\Y}$ satisfies uniformly the $QAM$-estimates;
\end{itemize}
Then  the family $\X$ satisfies uniformly the $QAM$-estimates.
\end{theorem}
This theorem will be proved in Section \ref{section-QAM-coarse-decomposability}.

\subsection{Application to metric space with finite decomposition  complexity}
Spaces with finite geometric decomposition were introduced in \cite{gty1,gty2}.  As an application of Theorem \ref{theorem-main}, we show that family of finite subsets  of a space with finite decomposition  complexity satisfies uniformy the QAM-estimates.
Recall first the concept of coarse decomposition for a class of proper discrete metric spaces.
\begin{definition}
Let $\CC$ be a class of family of  proper discrete metric spaces. We say that $\CC$ is closed under
coarse decomposition if the following is satisfied : let $\X$ be a  family  of discrete proper metric spaces   such that for every positive number $R$, there exists 
a family $\Y$ in $\CC$ such that $\X$ is $R$-decomposable with respect to $\Y$.  Then $\X$ belongs to $\CC$.
\end{definition}
\begin{example}\label{example-ubg-families}
The class of families of discrete metric spaces with uniformly bounded geometry is closed under coarse decomposition.
\end{example}

Let us define as in \cite{gty1,gty2} the class $\CC_{fdc}$ as the smallest class of family that contains uniformly bounded families of  discrete metric spaces
and which is closed under coarse decomposability.
  Typical examples of family of metric spaces in  $\CC_{fdc}$ are provided by spaces with finite asymptotic dimension.
Recall that for a metric  space $X$ and a positive number $r$, a cover $(U_i)_{i\in\N}$ has $r$-multiplicity $n$ if any ball of radius $r$ in $X$ intersects at most $n$ elements in $(U_i)_{i\in\N}$.

\begin{definition}
Let $\Sigma$ be a proper discrete metric set.  Then $\Sigma$ has {\bf finite asymptotic dimension} if there exists an integer $m$ such that for any positive number  $r$, there exists a  uniformly bounded cover $(U_i)_{i\in\N}$ with finite $r$-multiplicity $m+1$.  The smallest integer that satisfies this condition is called the {\bf asymptotic dimension} of $\Si$.
\end{definition}
We have  the following characterization of finite asymptotic dimension \cite{gr}.
\begin{proposition}\label{proposition-finite-asymptotic-dim}
Let $\Si$ be a proper discrete metric space and let $m$ be an integer. Then the following assertions are equivalent:
\begin{enumerate}
\item $\Si$ has asymptotic dimension $m$;
\item For every positive number $r$ there exist $m+1$ subsets $X^{(1)},\ldots,X^{(m+1)}$   of $\Si$ such that
\begin{itemize}
\item $\Si=X^{(1)}\cup\ldots\cup  X^{(m+1)}$;
\item for $i=1,\ldots m+1$, then $X^{(i)}$ is the $r$-disjoint union of a uniformly bounded family  $(X^{(i)}_k)_{k\in\N}$ of subsets of $X^{(i)}$, i.e  $X^{(i)}=\cup_{k\in\N}X^{(i)}_k,\, d_i(X^{(i)}_k,X^{(i)}_l)\gq r$ if $k\neq l$ and there exists a positive number $C$ such $diam\, X^{(i)}_k\lq C$ for all integer $k$.
\end{itemize}
\end{enumerate}
\end{proposition}
\begin{example}\label{example-tree}
If $T$ is a tree, then $T$ has asymptotic dimension equal to $1$.
\end{example}

Let $\Si$ be a proper metric space with  asymptotic dimension $m$, then there exists a sequence of positive numbers $(R_k)_{k\in\N}$  and for any  integer $k$ a cover $(U^{(k)}_i)_{i\in\N}$ of $\Si$ such that (\cite[Page 27]{gr} and \cite{y2})
\begin{itemize}
\item $R_{k+1}>4R_k$ for every  integer $k$;
\item  $U^{(k)}_i$ has diameter less than $R_k$ for every integers $i$ and $k$;
\item for any integer $k$, the $R_k$-multiplicity of $(U^{(k+1)}_i)_{i\in\N}$ is $m+1$.
\end{itemize}
The sequence $(R_k)_{k\in\N}$ is called {\bf the $\mathbf{m}$-growth} of $\Si$.

\begin{lemma}
Let  $m$ be an integer, let $(R_k)_{k\in\N}$ be a sequence of positive numbers such that  $R_{k+1}>4R _k$ for every   integer $k$.
Let $(\Si_i)_{i\in\N}$ be a family of proper metric spaces 
 with asymptotic dimension $m$ and $m$-growth  $(R_k)_{k\in\N}$. Then the family $(\Si_i)_{i\in\N}$
belongs to $\CC_{fdc}$. 
\end{lemma}
\begin{proof}
 Let us equip $\Si=\bigsqcup_{i\in\N}\Si_i$ with a distance $d_\Si$ such that the inclusion $\Si_i\hookrightarrow\Si$ is isometric for any  integer $i$ and $d_\Si(\Si_i,\Si_j)\gq i+j$ for all integers $i$ and $j$ with $i\neq j$.
 Then $\Si$ has asymptotic dimension $m$ and hence according to \cite{dz}, the metric space $\Si$ embeds uniformly in a product 
 of trees $\prod_{j=1}^n T_j$. Let $d$ be the metric  on $X=\prod_{j=1}^n T_j$  and $d_i$ the distance on $\Si_i$ when $i$ runs through integers. Then there exists two non-decreasing functions
 $\rho_\pm:[0,+\infty)\to [0,+\infty)$   and for every integer $i$ a map $f_i:\Si_i\to \prod_{j=1}^n T_j$ such that
  \begin{itemize}
 \item $\lim_{r\mapsto +\infty}\rho_{\pm}(r)=+\infty$;
 \item  $\rho_-(d_i(x,y))\lq d(f_i(x),f_i(y))\lq \rho_+(d_i(x,y))$ for all integer $i$ and all $x$ and $y$ in $\Si_i$.
 \end{itemize}
 If $n=1$, then $X$ is a tree and then the result holds in view   Example \ref{example-tree}. A straightforward induction shows that if $\Si$ embeds uniformly in a product of $n$ trees, then the family $(\Si_i)_{i\in\N}$
belongs to $\CC_{fdc}$. 
  \end{proof}

\begin{theorem}\label{theorem-QAM-fdc}Let $\X$ be a family of finite metric space in $\CC_{fdc}$ with uniformly bounded geometry,
then   $\X$ satisfies uniformly the QAM-estimates.\end{theorem}
\begin{proof} Let  $\CC'_{fdc}$ be the class of families of finite metric spaces with uniformly bounded geometry  that belongs to $\CC_{fdc}$. Then the class $\CC'_{fdc}$ is indeed the smallest class of families of finite metric spaces such that 
\begin{itemize}
\item $\CC'_{fdc}$  contains  all families of uniformly bounded  finite metric spaces    with uniformly bounded geometry;
\item $\CC'_{fdc}$ is closed under coarse decomposition. 
\end{itemize}
Its also straightforward to check that $\widetilde{\X}$ belongs to $\CC'_{fdc}$ if  $\X$ does.
Notice first that if $\X$ is a family of uniformly bounded  metric spaces  with   uniformly bounded geometry 
then $\X$ satisfies uniformly the QAM-estimates: 
\begin{itemize}
\item the $QI$-statement is a consequence on one hand of  that there exists a positive number $d$ such that $P_d(X)$ is contractible for every $X$ in $\X$ and on the other hand of that $\nu_{X,A,*}^{\bullet,\bullet,\bullet}$ induces in $K$-theory the map
induced by $P_d(X)\to\{pt\}$.
\item since there exists a positive number $r$ such that $\Kp(\ell^2(X))_r=\Kp(\ell^2(X))$ for all $X$ in $\X$, the $QS$-statement is a consequence of the isomorphism of equation \eqref{eq-direct-limit}.\end{itemize}
If a family of finite metric space $\X$ is $R$-decomposable with respect to a a family of finite metric space $\Y$, then
$\widetilde{\X}$ is clearly also $R$-decomposable with respect to the family of finite metric space $\widetilde{\Y}$.
Then the class of family $\X$ of finite metric spaces with uniformly bounded geometry such that $\widetilde{\X}$ satisfies the QAM-estimates   is according to Theorem \ref{theorem-main} and in view of Example \ref{example-ubg-families} closed under coarse decomposition and contains 
family of metric spaces  with uniformly bounded diameter and uniformly bounded geometry. Hence it contains  $\CC'_{fdc}$.
\end{proof}
As a  consequence of this theorem, we can deduce the  Novikov conjecture for group with finite asymptotic dimension
which admits a classifying space of finite type (see \cite{y2} for the original proof). The argument of the proof can be extended to the case
of groups with finite decomposition complexity. This class of groups was introduced in \cite{gty1}   in their study of topological rigidity for manifolds.
\begin{definition}
A proper discrete metric space $\Si$ has finite decomposition complexity if the single family $\{\Si\}$ is in $\CC_{fdc}$.
\end{definition}

Since the family of finite subsets of a discrete  metric space with  finite decomposition complexity is in $\CC_{fdc}$, we deduce from
Theorem \ref{theorem-QAM-fdc} the following.

\begin{theorem}
Let $\Si$ be a  metric space  with bounded geometry and  finite decomposition  complexity, then the family of finite subsets of $\Si$  satisfies uniformly the QAM-estimates.
\end{theorem}
As a consequence of Corollary \ref{corollary-Novikov}, we deduce

\begin{corollary}
Let $\Ga$ be a finitely generated group with finite decomposition complexity and which admits a classifying space of finite type,
then $\Ga$ satisfies Novikov conjecture.
\end{corollary}
Notice that the Novikov conjecture was indeed already known in this situation since  the property (A) holds for  finitely generated groups with finite decomposition complexity  \cite{gty2}[Theorem 4.3] and hence under the extra assumptions of the corollary,  the Novikov conjecture is satisfied \cite{y1}. But as mentioned before, our proof does not involve infinite dimension analysis.
\section{$K$-homology and small propagation}
In this section,  we define a quantitative version of the index map for  $K$-homology of a compact metric space $X$. This quantitative index map  is  valued	 in the quantitative $K$-theory for  the compact operators algebra  of any non degenerate $X$-standard module  and induces  in $K$-theory the usual index map. As we shall see,  when $X$ is a  finite simplicial complexe, then  the quantitative index map is a  $h$-isomorphism in sense of Definition \ref{def-h-isomorphism}

\subsection{$X$-modules}\label{section-nds-mod}

Recall that if  $X$ is  a compact  space, then an $X$-module is a pair $(\H_X,\rho_X)$, where $\H_X$ is a separable Hilbert space and $\rho_X:C(X)\to \L(\H_X)$ is a representation of $C(X)$ on $\H_X$.
An $X$-module $(\H_X,\rho_X)$ is called 
\begin{itemize}
\item non degenerate if $\rho_X(1)=Id_{\H_X}$;
\item standard if no non-zero function  acts as a compact operator on $\H_X$ (in particular $\rho_X$ is  one-to-one).
\end{itemize}
As usual, if $(\H_X,\rho_X)$ is an $X$-module, then  for $f$ in $C(X)$ and $\xi$ in $\H_X$ we denote for short $\rho_X(f)\cdot\xi$ by $f\cdot\xi$. 

\begin{example}Let  $\mu$ be a borelian  measure on $X$ such that $supp\,\mu=X$. Then $L^2(X,\mu)$ is a non degenerate  $X$-module for the representation given by the pointwise multiplication and if $\H$ is a separable Hilbert space, then  $L^2(X,\mu)\ts\H$ is  a  non degenerate  standard $X$-module. Indeed, if $X$ has no isolated point, then the representation on $L^2(X,\mu)$ is   standard. \end{example}
  Let us first recall the definition of support for vectors and operators on an $X$-module.
  \begin{itemize}
\item  Let $X$ be a compact space, let  $(\H_X,\rho_X)$ be an $X$-module and let $\xi$ be  a vector of $\H_X$. Then the support of $\xi$, denoted by  {$\supp\,\xi$}, is the complementary in $X$ of  the open  subset of elements $x$  for which  there exists $f$ in $C(X)$  with $f(x)\neq 0$  and $f\cdot \xi=0.$
\item  Let $X$ and $Y$ be compact spaces, let  $(\H_X,\rho_X)$ and $(\H_Y,\rho_Y)$ be respectively a $X$-module and a $Y$-module and let $T:\H_X\to\H_Y$ be a bounded operator.
Then the support of $T$, denoted by {$\supp\, T$}, is the complementary in $X\times Y$ of  the open  subset of  pairs  $(x,y)$  for which  there exist  $f$ in $C(Y)$  and $g$ in $C(X)$ with $f(y)\neq 0$, $g(x)\neq 0$ and $f\cdot T\cdot g=0$;
\end{itemize}
Let us now define for a compact space $X$ the restriction of an $X$-module to a closed subspace $Y$.
Let $(\H_X,\rho_X)$ be a $X$-module.
Define then $\H^Y_X$ as the set of elements in  $\H_X$  with support in $Y$.  For an element $\xi$ in $\H_X$, then $\xi$ belongs to $\H^Y_X$ if and only if $f\cdot\xi=0$ for all $f$ in $C_0(X\setminus Y)$. In particular, $\H^Y_X$ is a closed linear subspace of $\H_X$ and hence is an Hilbert space. Let $f$ be a continuous function in $C(Y)$ and let $\tilde{f}$ in $C(X)$ be a continuous function such that $\tilde{f}_{/Y}=f$. It is straightforward to check that if $\xi$ lies in $\H^Y_X$, then $\rho_X(\tilde{f})\cdot \xi$ does not depend on the choice of the lift $\tilde{f}$ for $f$. Hence, if we set $\rho^Y_X(f)\cdot\xi=\rho_X(\tilde{f})\cdot \xi$, then $(\H^Y_X,\rho^Y_X)$ is a  $Y$-module that we shall call the restriction of
 $(\H_X,\rho_X)$ to $Y$. Moreover,   if $(\H_X,\rho_X)$ is non degenerate then $(\H^Y_X,\rho^Y_X)$ is also non degenerate.


\subsection{Coarse geometry and non degenerate standard modules}\label{subsec-coarse-geom-nds-mod}
\begin{definition}Let $X$ be a compact metric space, let $(\H_X,\rho_X)$ and $(\H'_X,\rho'_X)$ be two $X$-modules, 
 let $T:\H_X\to\H'_X$ be a  bounded operator and let $r$ be a positive number. The operator $T$ is said to have propagation (less than) $r$ if $$\supp\, T\subseteq \{(x,x')\in X\times X\text{ such that }d(x,x')\lq r\}.$$ \end{definition}

If $X$ is a compact metric space and  if $(\H_X,\rho_X)$ is an $X$-module,
the algebra of compact operator $\Kp(\H_X)$ is then filtered  by $(\Kp(\H_X)_r)_{r>0}$, where 
$\Kp(\H_X)_r$ is for  any positive number $r$ the set of compact operators on $\H_X$ with propagation less than $r$.  If $A$ is a $C^*$-algebra, let
$A\ts\Kp(\H_X)_r$ be the closure of the algebraic tensor product of $A$ and $\Kp(\H_X)_r$  in the $C^*$-tensor product $A\ts\Kp(\H_X)$, then  $A_{\H_X}\defi A\ts\Kp(\H_X)$ is filtered  by   $(A\ts\Kp(\H_X)_r)_{r>0}$. 

Recall from \cite{hry} that if $\H_X$ and $\H'_X$  are two non degenerate  standard $X$-modules, then for every positive number $\delta$, there exists an isometry $V:\H_X\to\H'_X$ with propagation less than $\delta$.  The isometry  $V$ gives rise
by conjugation for any positive number $\eps$ and $r$ with $\eps<1/100$  to  a homomorphism 
$$Ad_{V,A}: A_{\H_X}\to  A_{\H'_X}$$ such that   $Ad_{V,A}(A_{\H_X,r})\subseteq A_{\H'_X,r+2\delta}$ and 
hence to a group morphism $$Ad_{V,A,*}^{\eps,r}:K_*^{\eps,r}(A_{\H_X})\to   K_*^{\eps,r+2\delta}(A_{\H'_X}).$$

\begin{lemma}\label{lemma-adjoint}With notations above,   for any isometries  $V_1:\H_X\to\H'_X$ and $V_2:\H_X\to\H'_X$ with propagation less than $\delta$, then $$\Ad_{V_1,A,*}^{\eps,r}= \Ad_{V_2,A,*}^{\eps,r}.$$\end{lemma}
\begin{proof}
Let us set for $t$ in $[0,1]$ $$R_t=\begin{pmatrix} \cos \frac{\pi t}{2}& -\sin\frac{\pi t}{2}\\ \sin \frac{\pi t}{2}& \cos \frac{\pi t}{2}&\end{pmatrix}$$ and define then the isometry 
$$W_t:\H_X\lto \H_{X'}\oplus \H_{X'}; \xi \mapsto R_t\diag(V_1,V_2) R_{-t}\cdot (\xi\oplus 0).$$
Then $$\Ad_{W_t,A}:\A_{\H_X}\lto M_2(A_{\H_{X'}})$$ is a homomorphism that increase propagation by $2\delta$ and
$(\Ad_{ W_t,A})_{t\in[0,1]}$ is a homomorphism homotopy between  $\diag(\Ad_{V_0,A},0)$ and      $\diag(\Ad_{ V_1,A},0)$.    \end{proof}
As consequence, we get
\begin{corollary}\label{corollary-adjoint}
Let $X$ be a compact metric space, let $\H_X$ and $\H'_X$ be  two non degenerate  standard $X$-modules. For any positive number $r$, choose $V_r:\H_X\to\H'_X$ a partial isometry with propagation less than $r/2$. Then
\begin{enumerate}
\item $\Lambda_{\H_X,\H'_X,*}\defi(Ad_{V_r,A,*})_{0<\eps<1/4,r>0}:K_*(A_{\H_X})\to   K_*(A_{\H'_X})$ is a $(1,2)$-controlled morphism that not depend on the choice of the family $(V_r)_{r>0}$.
\item  $\Lambda_{\H_X,\H'_X,*}$ is $(1,4)$-invertible.
\end{enumerate}
\end{corollary}

Let $X$ and $Y$ be two compact metric spaces, let $\H_X$ be a standard non degenerate $X$-module,  let $\H_Y$ be a standard non degenerate $Y$-module and let $f:X\to Y$ be a $C$-Lipschitz  map. Consider  $\H_X\oplus \H_Y$ as a $Y$-module by
using the homomorphism $$C(Y)\to C(X)\oplus C(Y);\, h\mapsto (h\circ f) \oplus h.$$ Then $\H_X \oplus  \H_Y$ is a standard non degenerate $Y$-module. For $\delta$ a positive number, fix an isometry $V:\H_X\oplus \H_Y\to  \H_Y$  with propagation less than $\de$. Define then the partial isometry $$V_f:\H_X\to\H_Y;\, \xi\mapsto V(\xi\oplus 0).$$ The support of $V_f$ lies in $$\{(x,y)\in X\times Y\text{ such that }d(y,f(x))<\de\}.$$ For any $C^*$-algebra $A$, let $\Ad_{V_f,A}:A_{\H_X}\to A_{\H_Y}$ the element obtained by conjugation.
According to Lemma \ref{lemma-adjoint}, we see that  $$\Ad_{V_f,A,*}:K_*^{\eps,r}(A_{\H_X})\to   K_*^{\eps,Cr+2\delta}(A_{\H_Y})$$  does not depend  on  the partial isometry $V:\H_X\oplus \H_Y\to  \H_Y$. In particular, for   $\delta=Cr/2$ we obtain then a family  of morphisms
$$f^{\sharp,\eps,r}_{A,*}:K_*^{\eps,r}(A_{\H_X})\to   K_*^{\eps,2Cr\ }(A_{\H_Y})$$ compatible with the structure map
and hence \begin{equation}\label{equ-def-sharp}f_{A,*}^\sharp\defi (f^{\sharp,\eps,r}_{A,*})_{0<\eps<1/100}:\K_*(A_{\H_X})\to   \K_*(A_{\H_Y})\end{equation} is a $(1,2C)$-controlled morphism. 
\begin{remark}\label{remark-covering-isometries}\
\begin{enumerate}
\item In view of the proof of Lemma \ref{lemma-adjoint} and with the  above notations,     for any positive number $r$ and any partial isometry $V_r:\H_X\to\H_Y$ with support
in $$\{(x,y)\in X\times Y\text{ such that }d(y,f(x))<Cr/2\},$$ then we have
$Ad^{\eps,r}_{V_r,A,*}=f^{\sharp,\eps,r}_{A,*}$.
\item Let $X$ be a compact metric space and let $Y$  be a closed subset of $X$. Let $(\H_X,\rho_X)$ be a non degenerate standard module such that  $(\H^Y_X,\rho^Y_X)$ (the restriction of $(\H_X,\rho_X)$ to $Y$) is standard.
Let $\jmath_{X,Y}:Y\hookrightarrow X$ be the the inclusion. On the other hand,
 let $\jmath^\sharp_{X,Y}:\Kp(\H^Y_X)\hookrightarrow \Kp(\H_X)$ be the inclusion of 
$\Kp(\H^Y_X)$  as a corner in  $\Kp(\H_X)$ via the orthogonal decomposition $\H_X=\H_X^Y\oplus^\bot \H'$.
Then the induced morphism $\jmath^\sharp_{X,Y,A}: A_{\H^Y_X}\hookrightarrow  A_{\H_X}$ is for any $C^*$-algebra $A$ a morphism of filtered $C^*$-algebras. The morphism induced in quantitative $K$-theory by $\jmath^\sharp_{X,Y,A}$ coincide up to a rescaling of the propagation by $2$ with  the morphism constructed in equation (\ref{equ-def-sharp}) from the isometric inclusion  $\jmath_{X,Y}$.
\end{enumerate}
\end{remark}

 Let $X,\,Y$ and $Z$ be compact metric spaces, let $\H_X,\,\H_Y$ and $\H_Z$ be respectively standard non degenerate
standard $X$-module, $Y$-module and $Z$-module and let $f:X\to Y$ and $g:Y\to Z$ be respectively $C$-Lipschitz and $C'$-Lipschitz maps. Then we can easily check that \begin{equation}\label{eq-composition}(g\circ f)_{A,*}^\sharp\stackrel{(1,4CC')}{\sim} g_{A,*}^\sharp\circ f^\sharp_{A,*}\end{equation}
for any $C^*$-algebra $A$. Let $X$ and $Y$ be  two compact metric spaces and let $C$ be a positive number. 
\begin{definition}
A family of maps $(g_t)_{t\in[0,1]}$ with $g_t:X\to Y$ continuous for all $t$ in $[0,1]$ is a $C$-Lipschitz homotopy if
\begin{itemize}
\item $[0,1]\times X\to Y;\, (t,x)\mapsto g_t(x)$ is continous;
\item $g_t$ is $C$-Lipschitz for all $t$ in $[0,1]$.
\end{itemize}
\end{definition}
\begin{proposition}\label{prop-homotopy}
Let $X$ and $Y$ be  two compact metric spaces, let $C$ be a positive number and let  $(g_t)_{t\in[0,1]}$ be a $C$-Lipschitz homotopy. Then for any $C^*$-algebra $A$,  any non degenerate standard $X$-module  $\H_X$ and any non degenerate standard $Y$-module $\H_Y$,
we have $$g^\sharp_{0,A,*}=g^\sharp_{1,A,*}.$$
\end{proposition}
\begin{proof}
Let $\de$ be a positive number with $\de<r/2$ and pick $$t_0=1<t_1<\ldots<t_{n-1}<t_n=1$$ such that
$d(g_{t_i}(x),g_{t_{i-1}}(x))<\de/4$ for all $i=1,\ldots,n$ and all $x$ in $X$. For every $i=1,\ldots,n$, choose
$V_i:\H_X\to\H_Y$ an isometry with support  in $$\{(x,y)\in X\times Y\text{ such that } d(g_{t_i}(x),y)<\de/4\}.$$
Let us set for $t$ in $[0,1]$ $$R_t=\begin{pmatrix} \cos \frac{\pi t}{2}& -\sin\frac{\pi t}{2}\\ \sin \frac{\pi t}{2}& \cos \frac{\pi t}{2}&\end{pmatrix}$$ and define then the isometry 
$$W_t:\H_X\lto \H_Y\oplus \H_{Y}; \xi \mapsto R_{\frac{t-t_i}{t_{i+1}-t_i}}\diag(V_{i},V_{i+1}) R^*_{\frac{t-t_i}{t_{i+1}-t_i}}\cdot (\xi\oplus 0)$$ for $t_i\lq t\lq t_{i+1}$ and $1\lq i\lq n-1$. Then $W_t$  has for all $t$ in $[0,1]$ support in $$\{(x,y)\in X\times Y\text{ such that } d(g_{t}(x),y)<\de/2\}.$$

Hence   $W_t$ induces by conjugation a homomorphism $$\Ad_{W_t,A}:A_{\H_X}\lto M_2(A_{\H_Y})$$  such that 
 $$\Ad_{W_t,A}(\A_{\H_X,r})\subseteq M_2(A_{\H_Y,Cr+2\de}).$$ 
Moreover  $(\Ad_{ W_t,A})_{t\in[0,1]}$ is a homotopy of homomorphisms between  $\diag(\Ad_{V_0,A},0)$ and      $\diag(\Ad_{ V_1,A},0)$ 
which respectively implement  $g^{\sharp,\eps,r}_{0,A,*}$ and  $g^{\sharp,\eps,r}_{1,A,*}$.
\end{proof}
\subsection{Fundamental class}\label{sub-sec-fund-class}
Let us  now define the  fundamental class  in $K^{\eps,r}_0(C(X)_{\H_X})$  for a compact metric space $X$ and for any $\eps$ in $(0,1/100)$ and any positive number $r$. This class plays the role of the class in $K$-theory of a rank one projection.  Let us first consider   $\mu$  a borelian  measure on $X$ such that $\supp\mu=X$. Let  $r$ be       a positive number and  let us choose a continuous function $\phi:X\times X\to \R$ supported
in $$\{(x,y)\in X\times X\text{ such that } d(x,y)<r/2\}$$ and such that for any $x$ in $X$, then $$\int \phi(x,y)d\mu(y)=1.$$ Consider the operator
$E_\phi: C(X)\ts L^2(X,\mu)\lto C(X)\ts L^2(X,\mu)$ defined by 
$$E_\phi\cdot\xi(x,y)=\phi^{1/2}(x,y)\int_X \phi^{1/2}(x,z)\xi(x,z)d\mu(z).$$ Then  $E_\phi$   is a projection in $C(X)\ts \Kp(L^2(X,\mu))$ with propagation less than $r$. Indeed, $E_\phi(x)$ is supported in $B(x,r/2)\times  B(x,r/2) $ for every $x$ in $X$. If $\phi$ and $\phi'$ are two continuous functions 
 supported
in $$\{(x,y)\in X\times X;\,\text{ such that } d(x,y)<r/2\}$$ and such that $$\int \phi(x,y)d\mu(y)=\int \phi'(x,y)d\mu(y)=1$$ for any $x$ in $X$, then
$(E_{(1-t)\phi+t\phi'})_{t\in[0,1]}$ is a homotopy of projections with propagation less than $r$ between $E_{\phi}$ and $E_{\phi'}$. 

Let $\H_X$ be a non degenerate standard $X$-module, let $\H$ be a separable Hilbert space and let us fix $e_0$ a rank one projection on $\H$.  For any positive numbers $r$, let $V:L^2(X,\mu)\to\H_X$ be an isometry with propagation less than ${r}/{4}$. Let $E_\phi$ the projection constructed as above with 
$\phi:X\times X\to \R$ supported
in $$\{(x,y)\in X\times X;\,\text{ such that } d(x,y)<{r}/{4}\}.$$
According to Lemma \ref{lemma-adjoint}, for any positive number $\eps$  with $\eps<1/100$,  the element
\begin{equation}\label{equation-EX}[E_{X}]_{\eps,r}\defi Ad_{V,*}^{\eps,r/2}([E_\phi\ts e_0,0]_{\eps,r/2})\end{equation} in $K_0^{\eps,r}(C(X)_{\H_X})$ does not depend on the choices of $V$ and of $\phi$. It does not depend neither on $e_0$. 
 Moreover, for any positive number $\eps,\,\eps',\,r$ an $r'$ with $\eps\lq \eps'<1/100$ and $r<r'$, then we have $\iota_*^{\eps,\eps',r,r'}([E_{X}]_{\eps,r})=[E_{X}]_{\eps',r'}$.
 For any continuous map $f:X\to Y$ where $X$ and $Y$ are compact sets, any filtered $C^*$-algebra $A$ and any positive numbers $\eps$ and $r$ with 
 $\eps<1/4$, Let
 $$f^{*,\eps,r}_{A}:K_*^{\eps,r}(C(Y)\ts A)\to K_*^{\eps,r}(C(X)\ts A)$$ be the morphism induced by $C(Y)\to C(X):h\mapsto h\circ f$.
 \begin{remark}\label{rem-fund-class-inclusion}
 Let $X$ be a compact metric space, let $Y$ be closed subset of $A$ and let $(\H_X,\rho_X)$ be a non
  degenerate standard $X$-module with restriction to $\H_X^Y$ non degenerate. Then, with notations of point (ii) of
   Remark \ref{remark-covering-isometries}, we have
   $$ \jmath_{X,Y,C(Y),*}^{\sharp,\eps,r}[E_Y]_{\eps,r}=\jmath^*_{X,Y,\Kp(\H_X)}[E_X]_{\eps,r}.$$
   \end{remark}
   This identity can indeed be extended to $C$-Lipschitz maps.
 \begin{lemma}\label{functoraility-fundamental-class}
 Let $X$ and $Y$ be two compact metric spaces, let $\H_X$ and $\H_Y$ be respectively non degenerate $X$-module and $Y$-modules. Let $f:X\to Y$ be a $C$-Lipschitz map for some positive numbers $C$.
 Then for  any positive numbers $\eps$ and $r$ with $\eps$ in $(0,1/100)$, we have
 $$ f_{C(X),*}^{\sharp,\eps,r}[E_X]_{\eps,r}=f^*_{\Kp(\H_Y)}[E_Y]_{\eps,2Cr}.$$
 \end{lemma}
 \begin{proof}
 Let $\mu_X$ and $\mu_Y$ be respectively  borelian  measures on $X$ and $Y$ such that $\supp\mu_X=X$ and $\supp\mu_Y=Y$.
Let us choose $\phi_0:X\times X\to [0,1]$ and $\phi_1:Y\times Y\to [0,1]$ continuous functions  such that
 \begin{itemize}
 \item $\phi_0$ is supported in $\{(x,x')\in X\times X\text{ such that }d(x,x')<\frac{r}{2C}\}$;  
  \item $\phi_1$ is supported in $\{(y,y')\in Y\times Y\text{ such that }d(y,y')<r/2\}$;
 \item $\int_X\phi_0(x,x')d\mu_X(x')=1$  for any $x$ in $X$;
 \item $\int_Y\phi_1(y,y')d\mu_Y(y')=1$  for any $y$ in $Y$.
 \end{itemize} 
 Define for any $x$ in $X$
 $$\phi_0(x,\bullet):X\to [0,1];x'\mapsto \phi_0(x,x')$$ and
 for any $y$ in $Y$
 $$\phi_1(y,\bullet):Y\to [0,1];y'\mapsto \phi_1(y,y').$$
 For any $x$ in $X$ then $\phi^{1/2}_0(x,\bullet)$ is a unit vector in $L^2(X,\mu_X)$ and
 for any $y$ in $Y$ then $\phi^{1/2}_1(y,\bullet)$ is a unit vector in $L^2(X,\mu_Y)$.
 
 Let us consider  for any $t$ in $[0,1]$ the $C(X)$-linear operator
$$E_t: C(X)\ts \left(L^2(X,\mu_X)\oplus L^2(Y,\mu_Y)\right)\lto C(X)\ts \left(L^2(X,\mu_X)\oplus L^2(Y,\mu_Y)\right)$$  which is pointwise at $x$ in $X$
the rank one projection associated to the unit vector $$t^{1/2}\phi^{1/2}_0(x,\bullet)\oplus (1-t)^{1/2}\phi^{1/2}_1(f(x),\bullet)$$ in
 $L^2(X,\mu_X)\oplus L^2(Y,\mu_Y)$.
 If  we view $ L^2(X,\mu_X)\ts L^2(Y,\mu_Y)$ as a $Y$-module acted upon for any $h$ in $C(Y)$ by multiplication $(h\circ f)\oplus h$,  then $E_t(x)$ is supported in $\{y\in Y\text{ such that }d(f(x),y)<r/2\}^2$  for every $x$ in $X$. Hence $E_t$ 
   is for all $t$ in $[0,1]$  a projection of  $C(X)\ts \Kp(L^2(X,\mu_X)\ts L^2(Y,\mu_Y))$ with propagation less than $r$.
Moreover, 
\begin{itemize}
\item $E_0$ is the image of $E_{\phi_0}\in C(X)\ts \Kp(L^2(X,\mu_X))$ under conjugation by the (isometrical) inclusion
$L^2(X,\mu_X)\hookrightarrow L^2(X,\mu_X)\oplus L^2(Y,\mu_Y)$;
\item the pull-back of $E_1$ by $f$ (which belongs to $C(Y)\ts \left(\Kp(L^2(X,\mu_X)\oplus L^2(Y,\mu_Y))\right)$)   is the image of $E_{\phi_1}\in C(Y)\ts \Kp(L^2(Y,\mu_Y))$ under conjugation by the  inclusion
$L^2(Y,\mu_Y)\hookrightarrow L^2(X,\mu_X)\oplus L^2(Y,\mu_Y)$
\end{itemize}
Using the  definition  of $f_{C(X),*}^{\sharp,\eps,r}[E_X]_{\eps,r}$,  the result is now a consequence of Corollary \ref{corollary-adjoint}.
 \end{proof}
\subsection{Quantitative index map}\label{subsec-quantitative-index}

First, let us fix  some notations. 
Let $X$ be  a compact space, let $(\rho_X,\H_X)$ be an $X$-module and let $A$ be a $C^*$-algebra. Recall that we have  set  
$A_{\H_X}\defi A\ts \Kp(\H_X)$. If $A$ is unital, we can identify the unitarisation  $\widetilde{A_{\H_X}}$
of 
$A_{\H_X}$ with the algebra $A\ts \Kp(\H_X)+\C\cdot 1\ts Id_{\H_X}$ (viewed has adjointable operators on the  right Hilbert $A$-module $A\ts \H_X$). If $A$ is non unital, with unitarisation $\widetilde{A}$, then we identify $\widetilde{A_{\H_X}}$ with the subalgebra $$A\ts \Kp(\H_X)+\C\cdot 1\ts Id_{\H_X}$$ of
$$\tilde{A}\ts \Kp(\H_X)+\C\cdot 1\ts Id_{\H_X}.$$If $B$ is a $C^*$-algebra and $f:A\to B$ is a $*$-homomorphism, then $f_{\H_X}:A_{\H_X}\to B_{\H_X}$ stands for  the $*$-homomorphism induced by $f$. If $z$ in an element in $KK_*(A,B)$ then
 $\tau_{\H_X}(x)$ stands for  the element $\tau_{\Kp(\H_X)}(z)$ of $KK_*(A_{\H_X},B_{\H_X})=KK_*(A\ts\Kp(\H_X),B\ts\Kp(\H_X))$ and in the same way 
$$\T_{\H_X}(x)=\left(\tau_{\H_X}^{\eps,r}(x)\right)_{\eps\in \left(0,\frac{1}{100\al}\right),r>0}$$ stands for 
$$\T_{\Kp(\H_X)}(x)=\left(\tau_{\H_X}^{\eps,r}(x)\right)_{\eps\in \left(0,\frac{1}{100\al}\right),r>0},$$ with $(\al,k)=(\al_\T,k_\T)$.
\medskip

We are now in a position to define the quantitative index map.  For $X$ a compact space and $A$ a $C^*$-algebra,   we set  $K_*(X,A)\defi KK_*(C(X),A)$.
\begin{definition}
Let $X$ be a compact metric space, let $(\H_X,\rho_X)$ be  a non degenerate standard $X$-module and let $A$ be a $C^*$-algebra. For any positive number $\eps$ and $r$  with $\eps<1/100$, the $\eps$-$r$-index is the map
\begin{eqnarray*}
Ind_{X,A,*}^{\eps,r}: K_*(X,A)&\longrightarrow &K_*^{\eps,r}(A_{\H_X} )\\ 
x&\mapsto &\tau^{{\eps/\al,r/k_{\eps/\al}}}_{\H_X}(x)( [E_X]_{\eps/\al,r/k_{{\eps}/{\al}}})\end{eqnarray*}
with $(\al,k)=(\al_\T,k_\T)$.
\end{definition}
It is straightforward to check that 
$$\iota^{\eps,\eps',r,r'}_*\circ Ind_{X,A,*}^{\eps,r}=  Ind_{X,A,*}^{\eps',r'}$$ for any positive number $\eps,\,\eps',\,r$ an $r'$ with $\eps<\eps'<1/100$ and $r<r'$ and hence $$Ind_{X,A,*}\defi (Ind_{X,A,*}^{\eps,r})_{0<\eps<1/100,r>0}:K_*(X,A)\lto\K_*(A_{\H_X})$$ is a  projective morphism in the sense of Definition \ref{def-quant-morphism}. Furthermore,
the class in $K_0(C(X)_{\H_X})$ of the projection $E_\phi$ constructed above is up to Morita equivalence the class of the trivial projection $1$ in $K_0(C(X))$ and hence we get that 
$\iota_*^{\eps,r}\circ Ind_{X,A,*}^{\eps,r}$ is up to Morita equivalence the usual index map.

\begin{remark}\label{rem-HX}
Let $\H_X$ and $\H'_X$ be two non degenerate standard $X$ modules and let us denote respectively by
$Ind^{\H_X}_{X,A}$ and $Ind^{\H'_X}_{X,A}$ the quantitative index map relatively to $\H_X$ and $\H'_X$, then according to 
Corollary \ref{corollary-adjoint} and in view of remark \ref{remark-naturality} regarding to the functoriality of $\TT_\bullet$, we obtain that
$$Ind^{\H'_X}_{X,A,*}=\Lambda_{X,\H_X,\H_X',*}\circ  Ind^{\H_X}_{X,A,*}.$$
\end{remark}
By naturality of $\T_{\H_X}$,  we have the following functoriality for the controlled index:
\begin{proposition}\label{prop-naturality-index}
Let $X$ be a compact metric space, let $(\H_X,\rho_X)$ be  non degenerate standard $X$-module and let $f:A\to B$ be 
$C^*$-algebra homomorphism. Then
$$Ind_{X,B,*}\circ f_*=f_{\H_X,*}\circ Ind_{X,A,*},$$ where $f_{\H_X}:A_{\H_X}\to B_{\H_X}$ is induced by $f$.\end{proposition}
The quantitative index map is compatible with the rescaling map defined in Section \ref{subsec-rescaling}.
\begin{proposition}\label{prop-compatibility-index-rescaling}
Let $A$ be a $C^*$-algebra, let $X$ be a compact metric space and let $(\H_X,\rho_X)$ be a non degenerate standard $X$-module. Then
$$L^{\eps,r}_{A_{\H_X},*}\circ Ind^{\eps,r}_{X,A,*}= Ind^{\eps/2,l_{\eps} r}_{X,A,*}$$ for any $\eps$ in $(0,1/200)$ (with $l=l_{rs}$ as in Section \ref{subsec-rescaling}).
\end{proposition}
\begin{proof}
Let $E_X$ as in Equation \eqref{equation-EX}. Since $E_X$ is a projection, we have that $$L_{C(X)_{\H_X},*}^{\eps,r}([E_X,0]_{\eps,r})=[E_X,0]_{\eps/2,l_\eps r}$$ for any $\eps$ in $(0,1/200)$ and any $r>0$. The result is then a consequence of Proposition \ref{prop-compatibility-kasparov-rescaling}.
\end{proof}
Moreover, we have the following compatibility result for the quantitative index with respect to Kasparov product.
\begin{proposition}\label{prop-compatibility-index}Let $X$ be a compact metric space,  let $(\H_X,\rho_X)$ be  a non degenerate standard $X$-module, let $A$ and $B$ be 
$C^*$-algebras and let $z$ be an element in $KK_*(A,B)$. Then 
$$Ind_{X,B,*}\circ(\bullet\ts z)= \T_{\H_X}(z) \circ Ind_{X,A,*},$$ where $\bullet\ts z:K_*(X,A)\to K_*(X,B)$ is the right Kasparov product by $z$.\end{proposition}
\begin{proof}
For propagation small enough, this is a consequence of Theorem \ref{thm-product-tensor}. The general case follows then from the first point of Remark \ref{rem-quant-morphism-small-propagation}.
\end{proof}

As a consequence of Lemma \ref{functoraility-fundamental-class}, of point (iii) of Theorem \ref{thm-tensor} and of Remark  \ref{remark-naturality}, we see that
the quantitative index is also functorial in spaces.

\begin{lemma}\label{functoriality-index-in-space}
 Let $X$ and $Y$ be two compact metric spaces, let $\H_X$ and $\H_Y$ be respectively non degenerate $X$-module and $Y$-module. Let $f:X\to Y$ be a $C$-Lipschitz map for some positive number $C$ and denote by $f_*:K_*(X,A)\to K_*(Y,A)$ the induced morphism. Then we have
 $$Ind_{Y,A,*}\circ f_*=f^{\sharp}_{A,*}\circ Ind_{X,A,*}.$$
 \end{lemma}
 
 \begin{remark}\label{rem-funct-inclusion}
 In the special case of inclusion, let $X$ be a compact metric space, let $Y$ be a closed subset of $A$ and let $(\H_X,\rho_X)$ be a non
  degenerate standard $X$-module with standard restriction to $\H_X^Y$. Then with notations point ii) of
   Remark \ref{remark-covering-isometries} and according to Remark  \ref{rem-fund-class-inclusion}, we have
    $$Ind_{X,A,*}\circ \jmath_{X,Y,*}=\jmath^{\sharp}_{X,Y,A,*}\circ Ind_{Y,A,*}.$$
   \end{remark}

 We are now in a position to state the main theorem of this section: for a finite simplicial complexe, then the quantitative index is a $h$-isomorphism 
 in the sense of Definition \ref{def-h-isomorphism} with $h$ depending only on the dimension of the complexe. 
 \medskip
 
Let $X$ be a finite simplicial complex provided with the simplicial topology. We equip $X$ with the normalized spherical metric (see \cite[Section 7.2]{wybook}). 
\begin{itemize}
\item
Let $\De$ be a (geometrical) simplex of $X$ of dimension $n$ with vertices $x_0,\ldots,x_n$. Let $x$ be an element of $\De$  with barycentric coordinates $(t_0,\ldots,t_n)$ with respect to $x_0,\ldots,x_n$ and set
$$f_\De(x)=\frac{1}{(t_0^2+\ldots+t_n^2)^{1/2}}(t_0,\dots,t_n).$$Then $f_\De(x)$ belongs to the unit $n$-sphere. We define $d$ on $\De$ by
$$\cos \frac{\pi}{2}d(x,y)=\langle f_\De(x),f_\De(y)\rangle$$ for any $x$ and $y$ in $\De$.
\item If $x$ and $y$ belong to  the same connected component of $X$, we set
$d(x,y)=\inf\{d(x_0,x_1)+\ldots+d(x_{n-1},x_n)$, where the infimium is taken over sequence $x=x_0,x_1,\ldots,x_{n-1},x_n=y$ such that $x_i$ and $x_{i+1}$ lie for $i=0,\ldots,n-1$ in a common  simplex;
\item If $x$ and $y$ lie in  distinct connected components, we set $d(x,y)=+\infty$.
\end{itemize}
\begin{theorem}\label{thm-qindex-iso}
There exists a non increasing function $$(0,1/200)\to[1,+\infty);\,\eps\mapsto h_\eps$$ such that 
for any finite simplicial complexe $X$ of dimension less than $m$, any non degenerate standard  $X$-module $\H_X$ and any $C^*$-algebra $A$, then
$$Ind_{X,A,*}: K_*(X,A)\longrightarrow \K_*(A_{\H_X} )$$ is a  $h^{(m^2)}$-isomorphism (in sense of definition \ref{def-h-isomorphism}).
\end{theorem} This theorem will be proved in  Section \ref{sec-proof-thm-iso}.
\begin{corollary}
Let $X$  be a  finite simplicial complexe. Then for any non degenerate standard  $X$-module $\H_X$, any $C^*$-algebra $A$ and any fixed $\eps$ in $(0,1/200)$, the map
\begin{eqnarray*}
K_*(X,A)&\lto&\varprojlim_{r>0} {K_*(A_{\H_X})}\\
x&\mapsto&(Ind^{\eps,r}_{X,A})_{r>0}
\end{eqnarray*}
is an isomorphism.
\end{corollary}

%
%

We end this section by showing  that the quantitative assembly maps for finite metric spaces can indeed be recovered from the quantitative index.
 Let $\Si$ be a finite metric space. For any integer $d$, we equipped  $P_d(\Si)$ the Rips complex of degree $d$ of $\Si$ with the normalized spherical metric.  Let $\H_{\Si,d}$ be a non degenerate standard $P_d(\Si)$-module. Consider then the isometry
 \begin{equation}\label{equ-isometry-assembly}V_{\H_{\Si,d}}:\H_{\Si,d}\lto \ell^2(\Si,\H_{\Si,d})\end{equation} defined by 
 $$V_{\H_{\Si,d}}\cdot\xi(\si)=\la^{1/2}_{\si}\cdot \xi$$ for all $\xi$ in $\H_{\Si,d}$ and $\si$ in $\Si$. We equip
  $\ell^2(\Si,\H_{\Si,d})$ with  the $\Si$-module structure given by pointwise multiplication by $C(\Si)$.
  \begin{lemma} If $T$ is an element of  
  $A_{\H_{\Si,d}}$ with propagation  $r$,
  then  $$Ad_{V_{\H_{\Si,d},A}}\cdot T\defi (Id_A\ts V_{\H_{\Si,d}})\cdot T\cdot (Id_A\ts V^*_{\H_{\Si,d}})$$  has propagation $2d(r+1)$.
  \end{lemma}
  \begin{proof}
  A pair $(\si,\si')\in \Si^2$ is in the support of $Ad_{V_{\H_{\Si,d},A}}$ if and only if there exists $(x,x')$ in the support of $T$ such that $\si$ is in the support of $x$ and $\si'$ is in the support of $x'$.
 The result is then a consequence of \cite[Corollary 7.2.6]{wybook}.
 \end{proof}
   Hence, in view of Proposition \ref {prop-morita}, for every positive numbers $\eps$ and $r$ with $\eps<1/100$, then
 $Ad_{V_{\H_{\Si,d},A}}$ induces   a morphism 
 $$Ad^{\eps,r}_{V_{\H_{\Si,d}},A,*}:K_*^{\eps,r}(A_{\H_{\Si,d}}){\lto} K_*^{\eps,2d(r+1)}(A_{\ell^2(\Si,\H_d)})\stackrel{\cong}{\lto} K_*^{\eps,2d(r+1)}(A_\Si) $$
(recall that $A_\Si=A\ts\Kp(\ell^2(\Si))$) such that
\begin{equation}
\label{equ-V-compatible}
\iota_*^{\eps,2d(r+1),2d(r'+1)}\circ Ad^{\eps,r}_{V_{\H_{\Si,d},A},*}=Ad^{\eps,r'}_{V_{\H_{\Si,d},A},*}\circ \iota_*^{\eps,r,r'}.
\end{equation}
for every positive number $r'$ with $r'\gq r$

 \begin{proposition}\label{proposition-assembly-conjugate}
 Let $\Si$ be a finite metric space, let $A$ be a $C^*$-algebra and let $\eps$, $d$ and $r$ be positive numbers with $\eps<1/100$ and $r\gq k_{\T,\eps/\al_\T}-1$. Then for any 
    non degenerate standard $P_d(\Si)$-module $\H_{\Si,d}$, we have up to Morita equivalence
 $$\nu_{\Si,A,*}^{\eps,2d(r+1),d}=Ad^{\eps,r}_{V_{\H_{\Si,d}},A,*}\circ Ind^{\eps,r}_{P_d(\Si),A,*}.$$
\end{proposition}
\begin{proof}
Let us first assume that $\H_{\Si,d}=L^2(P_d(\Si),\mu)\ts\H$, where $\H$ is a separable Hilbert  space and $\mu$ is a borelian measure
with $\supp\mu=X$.  For sake of simplicity, let us forget the factor $\H$. Notice that since 
$r_{d,\eps}=2dk_{\T,\al_\T\eps}$, we deduce that   $2d(r+1)\gq r_{d,\eps}$.  Then according to Remark  \ref{remark-naturality}, to equation (\ref{equ-V-compatible}) and to the definitions of $\nu_{\Si,A,*}^{\eps,2d(r+1),d}$ and  $Ind^{\eps,r}_*(P_d(\Si),A)$, this amounts to 
show  with  notations of  Sections  \ref{sub-sec-geo-ass-map} and \ref{sub-sec-fund-class}  that for any positive number $r$, for any continuous map $$\phi:P_d(\Si)\times P_d(\Si)\to \R$$ supported
in $$\{(x,y)\in P_d(\Si)\times P_d(\Si)\,\text{ such that } d(x,y)<r/2\}$$ and such that $\int_{P_d(\Si)}\phi(x,y)d\mu(y)=1$ and for any continuous function $\psi:P_d(\Si)\to [0,1]$ such that $\int_{P_d(\Si)}\psi(x)d\mu(x)=1$, then
$$Ad^{\eps,r}_{V_{\H_{\Si,d}},C(P_d(\Si)),*}[E_\phi,0]_{\eps,r}=[e_\psi\ts P_{P_d(\Si)}]_{\eps,2d(r+1)},$$ where $e_\psi:L^2(P_d(\Si),\mu)\lto L^2(P_d(\Si),\mu)$ is the orthogonal projection corresponding to the unit vector $\psi^{1/2}\in L^2(P_d(\Si),\mu)$.
Let us view     $\ell^2(\Si,L^2(P_d(\Si),\mu))$   as functions 
$P_d(\Si)\times\Si\lto\C$.
Define then 
$$\varphi_0(x,\bullet,\bullet): P_d(\Si)\times\Si\lto\R;\,(y,\si)\mapsto \phi(x,y)\lambda_\si(y).$$
and for any $t$ in $[0,1]$
$$\varphi_t(x,\bullet,\bullet): P_d(\Si)\times\Si\lto\R;\,(y,\si)\mapsto  (1-t)\varphi_0(x,y,\si)+t\lambda_\si(x)\psi(y).$$
Then 
$Ad_{V_{\H_{\Si,d}},C(X)}(E_\phi)$ is  pointwise at $x\in P_d(\Si)$ the rank one projection  associated to the unit vector 
$\varphi_0^{1/2}(x,\bullet,\bullet)$ in  $\ell^2(\Si,L^2(P_d(\Si),\mu))$.
For every $t$ in $[0,1]$ and $x$ in $P_d(\Si)$, the projection associated to the unit vector  $\varphi^{1/2}_t(x,\bullet,\bullet)$ in  $\ell^2(\Si,L^2(P_d(\Si),\mu))$
 has propagation less than
$2d(r+1)$. Hence we get in this way a homotopy  of projection of propagation less than $2d(r+1)$ in 
 $C(P_d(\Si))\ts \Kp(L^2(P_d(\Si),\mu)\ts\ell^2(\Si))$  between $Ad_{V_{\H_{\Si,d}},C(X)}(E_\phi)$ and $e_\psi\ts P_{P_d(\Si)}$.
 
 In the general case, we use Remark \ref{rem-HX} and the argument of the proof of Lemma \ref{lemma-adjoint}  to reduce to the previous case.
\end{proof}

\section{Mayer-Vietoris six term exact sequences in $K$-homology and in quantitative $K$-theory}\label{section-mv-kh} 
\subsection{Mayer-Vietoris exact sequence in $K$-homology}\label{subsec-MV-K-hom}
Let $(X,d)$ be a compact metric space, let ${X}_1$ and ${X}_2$ be closed subsets of $X$ such that $X=\mathring{X}_1\cup \mathring{X}_2$   and let $A$ be a $C^*$-algebra. Then we have the following six-term exact sequence 
\cite[Theorem 21.5.1]{bl}:{\tiny
\begin{equation}\label{diag-MV-KH}\begin{CD}
K_0(X_1\cap X_2,A)  @>(\jmath_{X_1,X_2,*},\jmath_{X_1,X_2,*})>> K_0(X_1,A)\oplus K_0(X_2,A) @>\jmath_{X_1,*}-\jmath_{X_2,*}>>K_0(X,A)\\
    @A \partial_{X_1,X_2,*} AA @.     @VV\partial_{X_1,X_2,*} V\\
K_1(X,A) @<\jmath_{X_1,*}-\jmath_{X_2,*}<< K_1(X_1,A)\oplus K_1(X_2,A) @<(\jmath_{X_1,X_2,*},\jmath_{X_2,X_1,*})<< K_1(X_1\cap X_2,A)\end{CD}.
\end{equation}}
where
\begin{itemize}
\item $\jmath_{X_1},\,\jmath_{X_2},\,\jmath_{X_1,X_2}$ and  $\jmath_{X_2,X_1}$ are respectively induced by the restrictions to $X_1\subseteq X,\,X_2\subseteq X,\, X_1\cap X_2\subseteq X_1$ and     $X_1\cap X_2\subseteq X_2$;
\item the boundary maps in $K$-homology $\partial_{X_1,X_2,*}$ are   given by left multiplication by an element $[\partial_{X_1,X_2}]$ in $KK_1(C(X_1\cap X_2),C(X))$ canonically associated to the pair $(X_1,X_2)$ and that we shall describe from \cite[Section 21.2]{bl} next.
\end{itemize}
Let $C({X_1,X_2})$ be the algebra of continuous functions $$f:([0,1/2]\times X_1)\cup ([1/2,1]\times X_2)\to \C$$ (where 
$([0,1/2]\times X_1)\cup  ([1/2,1]\times X_1) \subseteq [0,1]\times X$ is equipped with the induced topolology)  
such that
\begin{itemize}
\item $f(0,x)=0$ for all $x$ in $X_1$;
\item $f(1,x)=0$ for all $x$ in $X_2$.
\end{itemize}
Let $SC(X_1\cap X_2)$ be the suspension algebra of $C(X_1\cap X_2)$, i.e the algebra of continuous functions
$f:[0,1]\times (X_1\cap X_2)\to\C$ such that $f(0,x)=f(1,x)=0$ for all $x$ in $X_1\cap X_2$.
By restriction to $$[0,1]\times (X_1\cap X_2)\subseteq ([0,1/2]  \times X_1)\cup   ([1/2,1]\times X_2),$$ we get a surjective 
$*$-homomorphism  $$\Psi_{X_1,X_2}:C(X_1,X_2)\to SC(X_1\cap X_2)$$ with contractible kernel (the direct sum of two cone algebras) and hence
$\Psi_{X_1,X_2}$ induces an invertible element $[\Psi_{X_1,X_2}]$ in $KK_*(C(X_1,X_2),SC(X_1\cap X_2))$.
The restriction to $$\{1/2\}\times X\subseteq  ([0,1/2]\times X_1)\cup([1/2,1]\times X_2)$$ gives rise to  a $*$-homomorphism
$$\Psi'_{X_1,X_2}:C(X_1,X_2)\to C(X).$$ Let $[\partial]$ be the element of $KK_1(\C,C(0,1))$ that corresponds to the boundary element of the extension $$0\lto C(0,1)\lto C(0,1]\stackrel{ev_1}{\lto}\C\lto 0,$$ ($ev_1$ being the evaluation at $1$).
Then we have  (see \cite[Section 21.2]{bl})\begin{equation}\label{boundary-mv}[\partial_{X_1,X_2}]=\tau_{C(X_1\cap X_2)}([\partial])\ts\Psi'_{X_1,X_2,*}[\Psi_{X_1,X_2,*}]^{-1}.\end{equation}
\subsection{Controlled Mayer-Vietoris pair}\label{subsec-MV-pairs}We refere to \cite[Section 2]{oy4} for the definition of Complete Intersection Approximation (CIA) property and of  controlled Mayer-Vietoris pairs.
 
Let us introduce some further notations. Let   $(X,d)$ be a compact metric space.
\begin{itemize}
\item For a  closed subset $Y$ of $X$  and $\delta$ a positive number, we set $$Y_{\delta}=\{x\in Y\text{ such that }d(x,Y)\lq \delta\}.$$
\item If $(\H_X,\rho_X)$ is a $X$-module and $Y$ is a closed subset of $X$, we set with notations of Section \ref{section-nds-mod} $\H_Y=\H_X^Y$;
\item If $(\H_X,\rho_X)$ is a $X$-module,  $r$ is a positive number, and $X_1$ and $X_2$ are two closed subsets of $X$, we set $$\Kp(\H_{X_1},\H_{X_2})_r=\Kp(\H_{X_1},\H_{X_2})\cap\Kp(\H_X)_r.$$
\item with above notations, if $A$ is a $C^*$-algebra, we set $A_{\H_{X_1},\H_{X_2},r}=A\ts\Kp( \H_{X_1},\H_{X_2})_r$ . \end{itemize}
Let $X_1$ and $X_2$ be two closed subsets of $X$ such $X=\mathring{X_1}\cup \mathring{X_2}$. Then we have $\H_X=\H_{X_1}+\H_{X_2}$. Let us fix  $\delta$  a positive number and let $X_1^{(\de)}$ and $X_2^{(\de)}$ be two closed subsets of $X$ such that $X_{1,\de}\subseteq X_1^{(\de)}$ and $X_{2,\de}\subseteq X_2^{(\de)}$.  
Then for any positive number $r$ with $r<\delta$  
$$\Kp(\H_X)_r=\Kp(\H_{X_1},\H_{{X}^{(\delta)}_1})_r+\Kp(\H_{X_2},\H_{{X}^{(\delta)}_2})_r,$$     and similarly for any $C^*$-algebra $A$,
$$A_{\H_X,r}=A_{\H_{X_1},\H_{{X}^{(\delta)}_1},r}+\A_{\H_{X_2},\H_{{X_2}^{(\delta)}},r}.$$ Moreover it is straightforward to check that if $5r<\delta$, then
$$\left(A_{\H_{X_1},\H_{{X}^{(\delta)}_1},r},\A_{\H_{X_2},\H_{{X}^{(\delta)}_2},r},A_{\H_{{X}^{(\delta)}_1}},A_{\H_{{X}^{(\delta)}_2}}\right)$$  is  a controlled Mayer-Vietoris pair of order $r$ with coercitivity $1$. According to \cite[Section 3]{oy4}, then  for some universal control pairs $(\al,k)=(\al_\mv,k_\mv)$ and $(\lambda,k)$, there exists a $(\al,k)$-controlled morphism of order $r$ and of odd degree
\begin{equation}\label{equation-CMV-boundary}\D^{(\delta)}_{X_1,X_2,A,*}=(\partial^{\eps,s,(\delta)}_{X_1,X_2,A,*})_{0<\eps<\frac{1}{100\al_\D},0<s<\frac{r}{k_\eps}}:\K_*(A_{\H_X})\lto\K_*(A_{\H_{{X}^{(\delta)}_1\cap {X}^{(\delta)}_2}})\end{equation} such that the following diagram is $(\lambda,h)$-exact at order $r$
{\tiny $$\begin{CD}
\K_0(A_{\H_{{X}^{(\delta)}_1\cap {X}^{(\delta)}_2}})  @>(\jmath^{\sharp,(\delta)}_{X_1,X_2,A,*},\jmath^{\sharp,(\delta)}_{X_2,X_1,A,*})>> \K_0(A_{\H_{{X}^{(\delta)}_1}})\oplus \K_0(A_{\H_{{X}^{(\delta)}_2}}) @>\jmath^{\sharp,(\delta)}_{X_1,A,*}-\jmath^{\sharp,(\delta)}_{X_2,A,*}>>\K_0(A_{\H_X})\\
    @AA\D^{(\delta)}_{X_1,X_2,A,*} A @.     @V\D^{(\delta)}_{X_1,X_2,A,*}VV\\
\K_1(A_{\H_X}) @<\jmath^{\sharp,(\delta)}_{X_1,A,*}-\jmath^{\sharp,(\delta)}_{X_2,A,*}<< \K_1(A_{\H_{{X}^{(\delta)}_1}})\oplus \K_1(A_{\H_{{X}^{(\delta)}_2}}) @<(\jmath^{\sharp,(\delta)}_{X_1,X_2,A,*},\jmath^{\sharp,(\delta)}_{X_2,X_1,A,*})<< \K_1(A_{\H_{{X}^{(\delta)}_1\cap {X}^{(\delta)}_2}})\end{CD}.
$$}
where
\begin{itemize}
\item $\jmath^{\sharp,(\delta)}_{X_1,A}:A_{\H_{{X}^{(\delta)}_1}}\hookrightarrow  A_{\H_X}$;
\item $\jmath^{\sharp,(\delta)}_{X_2,A}:A_{\H_{{X}^{(\delta)}_2}}\hookrightarrow  A_{\H_X}$;
\item $\jmath^{\sharp,(\delta)}_{X_1,X_2,A}:A_{\H_{{X}^{(\delta)}_1\cap {X}^{(\delta)}_2}}\hookrightarrow  A_{\H_{{X}^{(\delta)}_1}}$;
\item $\jmath^{\sharp,(\delta)}_{X_2,X_1,A}:A_{\H_{{X}^{(\delta)}_1\cap {X}^{(\delta)}_2}}\hookrightarrow  A_{\H_{{X}^{(\delta)}_2}}$;
\end{itemize}
are the obvious inclusions.
By construction, the Mayer-Vietoris controlled boundary map is compatible with the rescaling defined in Section \ref{subsec-rescaling}.
\begin{lemma}\label{lemma-compatibility-MV-boundary-rescaling}
With above notations and  with $(\al,k)=(\al_\mv,k_\mv)$  and $l=l_{rs}$ as in Section \ref{subsec-rescaling}, we have
\begin{eqnarray*}
\partial_ {X_1,X_2,A,*}^{\eps/2,l_\eps s,(\de)}\circ L_{A_{\H_X},*}^{\eps,s}&=&\iota_*^{-,\al\eps/2,k_{\eps/2}l_\eps s}\circ L_{A_{\H_{{X}^{(\delta)}_1\cap {X}^{(\delta)}_2}},*}^{\al\eps,k_\eps s}\circ \partial_ {X_1,X_2,A,*}^{\eps,s,(\de)}\\
&=& L_{A_{\H_{{X}^{(\delta)}_1\cap {X}^{(\delta)}_2}},*}^{\al\eps,\frac{k_{\eps/2} l_\eps s}{l_{\al\eps}}}\circ \partial_ {X_1,X_2,A,*}^{\eps,\frac{k_{\eps/2} l_\eps s}{k_{\eps}l_{\al\eps}},(\de)}\circ\iota_{*}^{\eps,s,-}
\end{eqnarray*} 
for any $\eps$ in $\left(0,\frac{1}{200\al}\right)$ and any $s$ in $(0,\frac{r}{l_\eps k_{\eps/2}})$.
\end{lemma}

The fact that the quantitative index intertwines with  $\D^{(\delta)}_{X_1,X_2,A,*}$ and $\partial_{X_1^{(\de)},X_2^{(\de)},*}$ 
 is the key point for the proof of Theorem   \ref{thm-qindex-iso} and will be proved in the next two subsections. For  that purpose, we need to introduce
a coarse Mayer-Vietoris pair for $$A_{\H_X,\rho_X}\defi A\ts \Kp(\H_X)+A\ts \rho_X(C(X))$$ where $A$ is   a $C^*$-algebra and $(\H_X,\rho_X)$ is a non degenerate standard $X$-module.
Let us consider the following $C^*$-subalgebras of $A_{\rho_X,\H_X}$: $$A^{(1)}_{\H_X,\rho_X}=A\ts \left(\Kp(\H_{{X}^{(\delta)}_1}))+\rho_X(C_0(\mathring{X_1})\right)$$ and    $$A^{(2)}_{\H_X,\rho_X}=A\ts \left(\Kp(\H_{{X}^{(\delta)}_2}))+\rho_X(C_0(\mathring{X_2})\right).$$
Notice that if $z$ is an element of $A_{\H_X,\rho_X}$ with propagation $r$ with $r<\de$ and if $f$ is in $A\ts\rho_X(C_0(\mathring{X_i}))$ for $i=1,2$, then $f\cdot z$ is in $A^{(i)}_{\H_X,\rho_X}$.

For any $C^*$-algebra $A$ and any $r$ in $(0,\de/5)$, then
 \begin{equation*}\left (A_{\H_{X_1},\H_{{X}^{(\delta)}_1},r}+
 A\ts \rho_X(C_0(\mathring{X_1})),\,  A_{\H_{X_2},\,\H_{{X}^{(\delta)}_2},r}+
 A\ts \rho_X(C_0(\mathring{X_2})) ,\,A^{(1)}_{\H_X,\rho_X},\,A^{(2)}_{\H_X,\rho_X}\right)\end{equation*} is  a controlled Mayer-Vietoris pair of order $r$ with coercitivity $1$ for $A_{\H_X,\rho_X}$ and hence, as before 
 there exists for $(\al,k)=(\al_\mv,k_\mv)$ a $(\al,k)$-controlled morphism of order $r$ and of odd degree
\begin{equation}\label{equation-CMV-boundary-rho}\D'^{(\delta)}_{X_1,X_2,A,*}=(\partial'^{\eps,s,(\delta)}_{X_1,X_2,A,*})_{0<\eps<\frac{1}{100\al},0<s<r}:\K_*(A_{\H_X,\rho_X})\lto\K_*(A_{\H_{{X}^{(\delta)}_1\cap {X}^{(\delta)}_2}})\end{equation} such that the following diagram is $(\lambda,h)$-exact at order $r$
{\tiny $$\begin{CD}\label{diagram-CMV-boundary-rho}\K_0(A^{(1,2)}_{\H_X,\rho_X})  @>({\jmath'}^\sharp_{X_1,X_2,A,*},{\jmath'}^\sharp_{X_2,X_1,A,*})>> \K_0(A^{(1)}_{\H_X,\rho_X})\oplus \K_0(A^{(2)}_{\H_X,\rho_X}) @>{\jmath'}^\sharp_{X_1,A,*}-{\jmath'}^\sharp_{X_2,A,*}>>\K_0(A_{\H_X,\rho_X})\\
    @AA\D'_{X_1,X_2,A,*} A @.     @V\D'_{X_1,X_2,A,*}VV\\
\K_1(A_{\H_X,\rho_X}) @<{\jmath'}^\sharp_{X_1,A,*}-{\jmath'}^\sharp_{X_2,A,*}<< \K_1(A^{(1)}_{\H_{X},\rho_X})\oplus \K_1(A^{(2)}_{\H_X,\rho_X}) @<({\jmath'}^\sharp_{X_1,X_2,A,*},{\jmath'}^\sharp_{X_2,X_1,A,*})<< \K_1(A^{(1,2)}_{\H_X,\rho_X})\end{CD}.
$$}
where
\begin{itemize}
\item $A^{(1,2)}_{\H_X,\rho_X}\defi A_{\H_{X^{(\delta)}_1\cap X^{(\delta)}_2}}+A\ts \rho_X(C_0(\mathring{X_1}\cap \mathring{X_2}))$;
\item ${\jmath'}^\sharp_{X_1,A}:A^{(1)}_{\H_{X},\rho_X}\hookrightarrow  A_{\H_X,\rho_X}$;
\item ${\jmath'}^\sharp_{X_2,A}:A^{(2)}_{\H_{X},\rho_X}\hookrightarrow  A_{\H_X,\rho_X}$;
\item ${\jmath'}^\sharp_{X_1,X_2,A}: A^{(1,2)}_{\H_X,\rho_X}    \hookrightarrow  A^{(1)}_{\H_{X},\rho_X}$;
\item ${\jmath'}^\sharp_{X_2,X_1,A}: A^{(1,2)}_{\H_X,\rho_X}    \hookrightarrow  A^{(2)}_{\H_{X},\rho_X}$;
\end{itemize}
are the obvious inclusions.

In what follows we will respectively identify 
the unitarisation of
$A^{(1)}_{\H_X,\rho_X},\,A^{(2)}_{\H_X,\rho_X}$ and $A^{(1,2)}_{\H_X,\rho_X}$  
with the subalgebras $A^{(1)}_{\H_X,\rho_X}+\C\cdot 1\ts Id_{\H_X},\,A^{(2)}_{\H_X,\rho_X}+\C\cdot 1\ts Id_{\H_X}$ and $A^{(1,2)}_{\H_X,\rho_X}+\C\cdot 1\ts Id_{\H_X}$ of $\widetilde{A_{\H_X}}$ if $A$ is unital  and of 
 ${{\tilde A_{\H_X}}}+\C\cdot 1\ts Id_{\H_X}$ viewed as adjointable operator in the right Hilbert $\tilde{A}$-module $\tilde{A}\ts \H_X$ if not.
\subsection{Compatiblity of   boundaries}
 Next proposition, together with the naturality of the quantitative index of Proposition \ref{prop-naturality-index} , shows that  the quantitative index map   intertwines the Mayer-Vietoris six terms   exact sequence in $K$-homology of diagram (\ref{diag-MV-KH}) of Section \ref{subsec-MV-K-hom}  with the above six term controlled Mayer-Vietoris exact sequence.
 \begin{proposition}\label{proposition-index-commutes}
 For any  compact metric space $X$, any  positive number $\delta$, any
 closed subsets $X_1$ and $X_2$ of $X$, any closed subsets $X_1^{(\de)}$ and $X_2^{(\de)}$ of $X$  and 
 any standard $X$-module $(\H_X,\rho_X)$ such that
 \begin{itemize}
 \item $X=\mathring{X_1}\cup \mathring{X_2}$;
 \item $X_{i,\de}\subseteq X_i^{(\de)}$ for $i=1,2$;
 \item the  restriction of $(\H_X,\rho_X)$ to 
 $$\H_{{X}^{(\delta)}_1\cap {X}^{(\delta)}_2}=\{\xi\in\H_X\text{ such that }\supp\xi\subseteq{X}^{(\delta)}_1\cap {X_2}^{(\delta)}\}$$ is standard,
  \end{itemize}then  the following diagram is commutative
   $$\begin{CD}
K_*(X,A)@>\partial_{X^{(\de)}_1,X^{(\de)}_2,*}>> K_{*+1}({X}^{(\delta)}_1\cap {X}^{(\delta)}_2,A)\\
         @V Ind_{X,A,*} VV
         @VV Ind_{X^{(\de)}_1\cap X^{(\de)}_2,A,*}V\\
\K_*(A_{\H_X})@>\D^{(\delta)}_{X_1,X_2,A,*}>>\K_{*+1}(A_{\H_{{X}^{(\delta)}_1\cap {X}^{(\delta)}_2}})\end{CD} 
$$where on the top row, $\partial_{X^{(\de)}_1,X^{(\de)}_2,*}$ is the  left multiplication by  the element
$[\partial_{X^{(\de)}_1,X^{(\de)}_2}]$ of  $KK_1(C({X}^{(\delta)}_1\cap {X}^{(\delta)}_2),C(X))$ described in Section \ref{subsec-MV-K-hom}.
\end{proposition}
\begin{remark}
The above diagram a priori makes sense only up to order $\de/5$.
But indeed $\D^{(\delta)}_{X_1,X_2,A,*}\circ Ind_{X,A,*}$ can be extended to any positive number  in the following way:
 let $\eps$ and $r$ be positive numbers with $\eps<1/100$, pick a positive 
 number $s$ with $k_{\eps/\al}s<\de/5$ (with $(\al,k)=(\al_\mv,k_\mv)$) and $s\lq r$. Then
 $$\partial^{(\delta),\eps,r}_{X_1,X_2,A,*}\circ Ind^{\eps/\al,r/k_{\eps/\al}}_{X,A,*}\defi\iota_*^{\eps,s,r}\circ \partial^{(\delta),\eps/\al,s/k_{\eps/\al}}_{X_1,X_2,A,*}\circ Ind^{\eps/\al,s/k_{\eps/\al}}_{X,A,*}$$ does not depends on the choice of $s$ and hence define a projective morphism of odd degree
 \begin{equation*}
 \begin{split}
 \D^{(\delta)}_{X_1\cap X_2,A,*}\circ Ind_{X,A,*} &=\\(\partial^{(\delta),\eps,r}_{X_1,X_2,A,*}\circ& Ind^{\eps/\al,r/k_{\eps/\al}}_{X,A,*})_{\eps\in(0,1/100),r>0}:K_*(X,A)\lto \K_{*+1}(A_{\H_{{X}^{(\delta)}_1\cap {X}^{(\delta)}_2}}).\end{split}\end{equation*}
 \end{remark}
%
Proposition \ref{proposition-index-commutes} will be proved in next subsection.
The end of this subsection is devoted to a carefull study of the projective
 morphism
$$ Ind_{X^{(\de)}_1\cap X^{(\de)}_2,A,*}\circ  \partial_{X^{(\de)}_1,X^{(\de)}_2,*}: K_*(X,A)\lto \K_{*+1}(A_{\H_{{X}^{(\delta)}_1\cap {X}^{(\delta)}_2}}).$$

Let $f:X\to[0,1]$ be a continuous function such that $f(x)=1$ if $x$ is in $X\setminus \mathring{X_1}$ and $f(x)=0$ if $x$ is in
$X\setminus \mathring{X_2}$. Let $(p_r)_{r>0}$ be a family of  projections   in $C({X}_1\cap {X}_2)_{\H_{X_1\cap {X}_2}}$ such that
\begin{itemize}
\item $p_r$ has propagation less than $r$ for every positive number $r$.
\item $\iota_*^{\eps,r,r'}([p_r,0]_{\eps,r}=[p_{r'},0]_{\eps,r'}$  in $K^{\eps,r'}_0(C({X}_1\cap {X}_2)_{\H_{X_1\cap {X}_2}})$ for every positive numbers $\eps,\,r$ and $r'$ with   $\eps<1/100$ and $r\lq r'$.
\end{itemize}
Consider then for every positive number $r$   
\begin{equation}\label{eq-Uf}
U_{f,r}:X \lto \Kp(\H_{X_1\cap {X}_2})+\C\ts Id_{\H_{X_1\cap {X}_2}}\end{equation} defined by
\begin{itemize}
\item $U_{f,r}(x)=Id_{\H_{X_1\cap X_2}}$ if $x\notin X_1\cap {X}_2$;
\item $U_{f,r}(x)=e^{2\imath\pi f(x)}p_r(x)+Id_{\H_{X_1\cap X_2}}-p_r(x)$ if $x$ is in $ X_1\cap {X}_2$.
\end{itemize}
Then  $U_{f,r}$ is a unitary in  $\widetilde{C(X)_{\H_{X_1\cap {X}_2}}}$ with propagation less than $r$ and 
$\iota_*^{\eps,r,r'}([U_{f,r}]_{\eps,r})=[U_{f,r'}]_{\eps,r'}$  in $K^{\eps,r'}_1(C(X)_{\H_{X_1\cap {X}_2}})$ for every positive numbers $\eps,\,r$ and $r'$ with   $\eps<1/100$ and $r\lq r'$.
For every positive number $\eps$ and $r$, with $\eps<1/100$, define then for $(\al,k)=(\al_\mv,k_\mv)$

\begin{eqnarray*}
\Lambda^{\eps,r}:K_1(X,A)&\lto&K^{\eps,r}_0(A_{\H_{X_1\cap {X}_2}})\\
x&\mapsto& \tau^{{\eps/\al,r/k_{\eps/\al}}}_{\H_{X_1\cap {X}_2}}(x)( [U_{f,r/k_{\eps/\al}}]_{\eps/\al,r/k_{\eps/\al}})
\end{eqnarray*}

and

\begin{eqnarray*}
\Lambda'^{\eps,r}:K_1(X,A)&\lto&K^{\eps,r}_0(A_{\H_{X_1\cap {X}_2}})\\
x&\mapsto& \tau^{{\eps/\al,r/k_{\eps/\al}}}_{\H_{X_1\cap {X}_2}}([\partial_{X_1, {X}_2}]\ts x)( [p_{r/k_{\eps/\al}},0]_{\eps/\al,r/k_{\eps/\al}})
\end{eqnarray*}
Then $$\Lambda=(\Lambda^{\eps,r})_{0<\eps,1/100,r>0}:K_1(X,A)\lto \K_0(A_{\H_{X_1\cap X_2}})$$
and
$$\Lambda'=(\Lambda'^{\eps,r})_{0<\eps,1/100,r>0}:K_1(X,A)\lto \K_0(A_{\H_{X_1\cap {X}_2}})$$ are
 projective morphisms in sense of Definition \ref{def-quant-morphism}.
\begin{lemma}\label{lemma-family-projections} 
With above notations, we have  $\Lambda=\Lambda'$.
\end{lemma}
\begin{proof}
Let 
$$\Psi_{X_1,X_2}:C(X_1,X_2)\to SC(X_1\cap X_2)$$
and 
$$\Psi'_{X_1,X_2}:C(X_1,X_2)\to C(X)$$ being as in equation (\ref{boundary-mv}). Recall that 
\begin{itemize}
\item the element $[\Psi_{X_1,X_2}]$ induced by $\Psi_{X_1,X_2}$ is invertible as an element in $KK_*(C(X_1,X_2),SC(X_1\cap X_2))$;
\item $[\partial_{X_1,X_2}]=\tau_{C(X_1\cap X_2)}([\partial])\ts\Psi'_{X_1,X_2,*}[\Psi_{X_1,X_2,*}]^{-1}$;
\end{itemize}
 (recall that   $[\partial \in KK_1(\C,C(0,1))$ corresponds to the boundary map of the Bott extension $0\lto C(0,1)\lto C(0,1]\stackrel{ev_1}{\lto}\C\lto 0$). 
%
 Then,  according to Theorems   \ref{thm-tensor} and \ref{thm-product-tensor}  we get that there exists a control pair depending only on $(\al_\T,k_\T)$ such that
 $$\T_{\H_{X_1\cap X_2}}( [\partial_{X_1,X_2}])\aeq  \Psi'_{X_1,X_2,\H_{X_1\cap X_2},*}\circ \T_{\H_{X_1\cap X_2}}([\Psi_{X_1,X_2,*}]^{-1})\circ \T_{\H_{X_1\cap X_2}}(\tau_{C(X_1\cap X_2}([\partial])),$$
 where $ \Psi'_{X_1,X_2,\H_{X_1\cap X_2}}:C(X_1,X_2)_{\H_{X_1\cap X_2}}\to C(X)_{\H_{X_1\cap X_2}}$  is the $*$-homomorphism induced by  $\Psi'_{X_1,X_2}$. For any $x$ in $K_*(X,A)$, we define
  $$\F(x)=\T_{\H_{X_1\cap X_2}}(x)\circ \Psi'_{X_1,X_2,\H_{X_1\cap X_2},*}\circ \T_{\H_{X_1\cap X_2}}([\Psi_{X_1,X_2,*}]^{-1})\circ \T_{{\H_{X_1\cap X_2}}}((\tau_{C(X_1\cap X_2}([\partial]))).$$
 Then $$\F(x)=(F^{\eps,r})_{0<\eps<\frac{1}{100\alpha_\F},r>0}:\K_*(C(X)_{\H_{X_1\cap X_2}})\lto\K(A_{\H_{X_1\cap X_2}})$$  is a $(\al_\F,k_\F)$-controlled morphism for some control pair $(\al_\F,k_\F)$ depending only on $(\al_\T,k_\T)$. 
 According to Theorems   \ref{thm-tensor} and \ref{thm-product-tensor}, there exists a  control pair $(\la,h)$, depending only on
$(\al_\F,k_\F)$ such that  \begin{equation}\label{equ-lemma-fund-class}\F(x)\aeq \T_{\H_{X_1\cap X_2}}(x\ts [\partial_{X_1,X_2}])\end{equation}
 
 for all $x$ in $K_*(X,A)$. Let $f:X\to[0,1]$ be a continuous function as above such that $f(x)=1$ if $x$ is in $X\setminus \mathring{X_1}$ and $f(x)=0$ if $x$ is in
$X\setminus \mathring{X_2}$ and defined then $F:X\times [0,1]\to [0,1]$  by 
 \begin{itemize}
 \item $F(x,t)=2f(x)t$ if $0\lq t\lq 1/2$;
 \item $F(x,t)=2(f(x)+t-f(x)t)-1$ if $1/2\lq t\lq 1$\end{itemize}
  Consider then for $(p_r)_{r>0}$ a family of projections in $C({X}_1\cap {X}_2)_{\H_{X_1\cap {X}_2}}$ as above
 \begin{eqnarray*}
 W_r:(X_1\times [0,1/2]) \cup (X_2\times [1/2,1]) &\lto& \K(\H_{X_1\cap X_2}))+\C Id_{\H_{X_1\cap X_2}}\\
 (x,t)&\mapsto&e^{2\imath\pi F(x,t)}p_r(x)+Id_{\H_{X_1\cap X_2}}-p_r(x). 
\end{eqnarray*}
Then  $W_r$ is a unitary in 
$\widetilde{C(X_1,{X}_2) _{\H_{X_1\cap {X_2}}}}$ with propagation less than $r$.
We clearly have $$\Psi'^{\eps,r}_{X_1,X_2,\H_{X_1\cap X_2},*}([W]_{\eps,r})=[U_{f,r}]_{\eps,r}.$$
By  using the homotopy of $\eps$-$r$-unitaries
\begin{eqnarray*}
\lbrack0,1\rbrack\times (X_1\cap X_2)\times  \lbrack 0,1\rbrack&\lto& {C(X_1\cap X_2)_{\H_{X_1\cap X_2}}}+\C\cdot1\ts Id_{\H_{X_1\cap X_2}}\\
(s,x,t)&\mapsto& e^{2\imath\pi (sF(x,t)+(1-s)t)}p_r(x)+ 1\ts Id_{\H_{X_1\cap X_2}}-p_r(x),
\end{eqnarray*}
we see that
 $\Psi^{\eps,r}_{X_1,X_2,\H_{X_1\cap X_2},*}([W]_{\eps,r})$ is equal to the class in $K_1^{\eps,r}( {SC(X_1\cap X_2)_{\H_{X_1\cap X_2}}})$ represented by the $\eps$-$r$-unitary defined 
 by \begin{eqnarray*}
 [0,1]&\lto &{C(X_1\cap X_2)_{\H_{X_1\cap X_2}}}+\C\cdot 1\ts Id_{\H_{X_1\cap X_2}} \\
 t&\mapsto&   e^{2\imath\pi t}p_r(x)+1\ts Id_{\H_{X_1\cap X_2}}-p_r(x).\end{eqnarray*} Hence,  according to \cite[Proposition 3.9]{oy2} and to point (vii) of Theorem   \ref{thm-tensor} and  up to possibly rescaling the control pair 
 $(\al_\T,k_\T)$  we get that 
 $$\Psi^{\al_\T \eps,k_{\T,\eps}r}_{X_1,X_2,\H_{X_1\cap X_2}),*}([W]_{\al_\T \eps,k_{\T,\eps}r})=\tau^{\eps,r}_{\H_{X_1\cap X_2}}(\tau_{C(X_1\cap X_2}([\partial])([p_r,0]_{\eps,r}).$$ 
 
 In view of  Theorem   \ref{thm-product-tensor}, there exists a  control pair $(\al,k)$, depending only on 
 $(\al_\T,k_\T)$ such that $\T_{\H_{X_1,X_2}}([\Psi_{X_1,X_2}]^{-1})$ is an $(\al,k)$-inverse for $\Psi_{X_1,X_2,\H_{X_1\cap X_2},*}$. We deduce from  equation (\ref{equ-lemma-fund-class}) that there exists $\al\gq \la$ such that
\begin{equation}\label{eq-small-control}\tau^{\eps,r}_{\H_{X_1\cap X_2}}(x)([U_{f,r}]_{\eps,r})=
 \tau^{\eps,r}_{\H_{X_1\cap X_2}}([\partial_{X_1,X_2}]\ts x)([p_r,0]_{\eps,r})\end{equation} for 
 $\eps$ in $(0,\frac{1}{100\la})$ and $r>0$.
 The result is then a consequence of the first point of Remark \ref{rem-quant-morphism-small-propagation}.
 \end{proof}
 Let $(X,d)$ be a compact metric space, let $(\H_X,\rho_X)$ be a standard $X$-modules and let $X_1$ and $X_2$ two closed subsets of $X$ such that 
 $X=\mathring{X_1}\cup \mathring{X_2}$.   Let $\delta$ and $r$ be positive numbers with $5r<\delta$ and let $p:X\to \Kp(\H_X)+Id_{\H_X}$ be a continuous map such that $p(x)$ is a projection with support in 
 $B(x,r/2)\times B(x,r/2)$ for every $x$ in $X$. For any continuous function $f:X\to[0,1]$ such that $f(x)=0$ if $x\notin X_2$ and $f(x)=1$ is $x\notin X_1$, then
 $$e^{2\imath\pi f(x)p(x)}=e^{2\imath\pi f(x)}p(x)+Id_{\H_{X}}-p(x)$$  is  for any $x$ in $X$
 a unitary with propagation $r$ in the unitarisation of 
$C(X)_{\H_{X^{(\delta)}_1\cap X^{(\delta)}_2}}$.
 Two  continuous functions from $X$ to $[0,1]$ vanishing on $X\setminus X_1$ and identically equal to $1$ on $X\setminus X_2$ give rise to homotopic unitaries with propagation $r$, i.e. connected by a unitary of the unitarisation of
$C([0,1]\times X)_{\H_{X^{(\delta)}_1\cap X^{(\delta)}_2}}$ with propagation $r$. Moreover two continuous   maps $p_0,p_1:X\to \Kp(\H_X)+Id_{\H_X}$   such that $p_0(x)$ and $p_1(x)$  are projection with support in 
 $B(x,r/2)\times B(x,r/2)$ and which are connected by a  homotopy $$[0,1]\times X\lto \Kp(\H_X)+Id_{\H_X};(t,x)\mapsto p_t(x)$$ 
 such that  $p_t(x)$  is  a projection with support in 
 $B(x,r/2)\times B(x,r/2)$ for any $t$ in $[0,1]$ and any  $x$ in $X$ gives rise to  homotopic unitaries. Apply this to the class $[E_X]_{\eps,r}$, we obtain a class $[U^{(\delta)}_{X_1,X_2}]_{\eps,r}$ in $K_1^{\eps,r}(C(X)_{\H_{X^{(\delta)}_1\cap X^{(\delta)}_2}})$ which satisfies $$\iota_*^{-,\eps',r'}([U^{(\delta)}_{X_1,X_2}]_{\eps,r})=[U^{(\delta)}_{X_1,X_2}]_{\eps',r'}$$ for every positive numbers $\eps,\,\eps',r$ and $r'$ with $\eps\lq\eps'<1/100$ and $r'\lq r<\de/5$.
 
 For any positive number $\eps$ and $r$  with $\eps<1/100$, pick any  positive number $s$ with $s\lq r$ and $s<k_{\eps/\al}\de/5$ for $(\al,k)=(\al_\T,k_\T)$.
 Then
\begin{eqnarray*}
\partial Ind^{{(\delta}),\eps,r}_{X_1,X_2,A,*}: K_1(X,A)&\longrightarrow &K_0^{\eps,r}(A_{\H_{X^{(\delta)}_1\cap {X}^{(\delta)}_2}} )\\ 
x&\mapsto& \iota_*^{\eps,s,r}\tau^{{\eps/\al,s/k_{\eps/\al}}}_{\H_{X^{(\delta)}_1\cap {X}^{(\delta)}_2}}(x)( [U^{(\delta)}_{X_1,X_2}]_{\eps/\al,s/k_{\eps/\al}})
\end{eqnarray*}
does not depend on the choice of $s$ and
  $$\partial Ind^{(\delta)}_{X_1,X_2,A,*}\defi (\partial Ind^{(\delta),\eps,r}_{X_1,X_2,A,*})_{0<\eps<1/100,0<r}:K_1(X,A)\lto\K_0(A_{\H_{X^{(\delta)}_1\cap {X}^{(\delta)}_2}})$$ is a projective  morphism.
%
\begin{lemma}\label{lemma-bondary-index}
We above notations, we have $$\partial Ind^{{(\delta})}_{X_1,X_2,A,*}=Ind_{X_1^{(\delta)}\cap X_2^{(\delta)},A,*}\circ\partial_{X_1^{(\delta)}\cap X_2^{(\delta)},A,*}.$$
\end{lemma}
\begin{proof} Let us set $Y=X^{(\delta)}_1\cap {X}^{(\delta)}_2$.
Let $(\H_X,\rho_X)$ be a non degenerate standard $X$-module with standard restriction to $Y$. Let $\mu$ be a measure on $X$, let $\mu'$ be the restriction of $\mu$ tu $Y$ and assume that  $\supp \mu=X$ and $\supp \mu'=Y$. Fix a Hilbert space $\H$ and let  choose isometries 
 $V_1:\H_{Y}\to L^2(Y,\mu')\ts \H$ and $V_2:\H_X\to L^2(X,\mu)\ts \H$
 with propagation
 less than $r/2$.
By considering  the  square
$$\begin{CD}
L^2(Y,\mu')\ts\H@>>>L^2(X,\mu)\ts \H'\\
         @V {V_1} VV
         @VV  {V_2} V\\
\H_{Y}@>>>\H_X\end{CD},$$where the horizontal maps are inclusions and  according to  the first point of Remark \ref{remark-covering-isometries}, we can indeed assume without loss of generality
that $\H_X=L^2(X,\mu)\ts \H$ for $\H$ a separable Hilbert space and thus 
$\H_{Y}=L^2(Y,\mu')\ts\H$.
As in the definition of $[E_X]_{\eps,r}$ and  $[E_Y]_{\eps,r}$, let us choose $\phi_{0,r}:X\times X\to [0,1]$ and $\phi_{1,r}:Y\times Y\to [0,1]$ continuous functions  such that
 \begin{itemize}
 \item $\phi_{0,r}$ is supported in $\{(x,x')\in X\times X\text{ such that }d(x,x')<r/2\}$;  
  \item $\phi_{1,r}$ is supported in $\{(y,y')\in Y\times Y\text{ such that }d(y,y')<r/2\}$;
 \item $\int_X\phi_{0,r}(x,x')d\mu(x')=1$  for any $x$ in $X$;
 \item $\int_Y\phi_{1,r}(y,y')d\mu'(y')=1$  for any $y$ in $Y$.
 \end{itemize}

 Then  $y\mapsto\phi_1(y,\bullet)$ can be viewed as a continuous function from $Y$ valued in $L^1(Y,\mu')\subseteq L^1(Y,\mu)$. The
 idea is to   compare $[U^{(\delta)}_{X_1,X_2}]_{\eps,r}$ with the class of the unitary of Equation \eqref{eq-Uf} arising  from  $E_{\phi_{0,r}}$.
 Define then for any $t$ in $[0,1]$
 $$\phi_{t,r}:Y\ \to L^1(X,\mu): x\mapsto (1-t)\phi_0(y,\bullet)+t\phi_1(y,\bullet).$$
 For all $y$ in $Y$ and  for all $t$, the support of $\phi_{t,r}(y)$  is included in 
 $$\{x\in  X\text{ such that }d(y,x)<r\}.$$
  Consider then  the operator
$$E_{t,r}: C(Y)\ts L^2(X,\mu)\lto C(Y)\ts L^2(X,\mu)$$ defined by 
$$E_{t,r}\cdot\xi(y,x)=\phi_{t,r}^{1/2}(y,x)\int_{X} \phi_{t,r}^{1/2}(y,x')\xi(y,x')d\mu_X(x')),$$ with $\xi$ in  $C(Y)\ts L^2(X,\mu)$ viewed as a function in $Y\times X$. Then $E_{t,r}(y)$ 
   is for all $t$ in $[0,1]$  a projection of  $C(Y)_{L^2(X,\mu)}=C(Y)\ts \Kp(L^2(X,\mu))$ with propagation less than $r$.
 If $f:X\to[0,1]$ is a continuous function such that $f(x)=0$ if $x\notin X_2$ and $f(x)=1$ is $x\notin X_1$, define for $r<\de/5$
 $$U_{t,r}(x):X\lto \Kp(L^2(Y,\mu'))+\C\cdot Id_{L^2(Y,\mu')}$$ by 
 \begin{itemize}
\item $U_{t,r}(x)=Id_{L^2(Y,\mu')}$ if $x\notin Y$;
\item $U_{t,r}(x)=e^{2\imath\pi f(x)}E_{t,r}(x)+Id_{L^2(Y,\mu')}-E_{t,r}(x)$ if $x$ is in $ Y$.
\end{itemize}
We obtain in this way a homotopy of unitaries  with propagation $r$ in the unitarisation of 
$C(X)_{L^2(Y,\mu')}$. Since $[ U_{0,r}(x)]_{\eps,r}=[U^{(\delta)}_{X_1,X_2}]_{\eps,r}$, the result is now a consequence of
Lemma \ref{lemma-family-projections} applied to the family $(E_{1,r})_{r>0}$.
 \end{proof}

\subsection{Proof of Proposition \ref{proposition-index-commutes}} 
  This proof will require some preliminary work.
  Let considerer a semi-split extension
  $$0\to J\to A\to C(X)\to 0$$ where $J$ is a closed two sided  ideal in a unital $C^*$-algebra $A$.
   Let $(E_r)_{r>0}$ be a family of projections in $C(X)\ts \Kp(\H_X)$ such that for every positive number $r$ and for every
 $x$ in $X$, then $E_r(x)$ has support in $B(x,r/2)\times B(x,r/2)$ and $[E_X]_{\eps,r}=[E_r,0]_{\eps,r}$ for every positive numbers $\eps$ and $r$ with $\eps<1/100$.  
  Let $f:X\to[0,1]$ be a continuous function such that $\supp f\subseteq \mathring{X_2}$ and $\supp (1-f)\subseteq \mathring{X_1}$ 
%
%
 
Let $r_0$ be a positive number such that $|f(x)-f(x')|<\frac{\eps}{16\pi}$ if $d(x,x')<r_0$. We can assume without loss of generality that $E_r(x)$ has support in $B(x,r_0/2)\times B(x,r_0/2)$   for every  positive number $r$ and for every
 $x$ in $X$.  Let then $z_r$ in $A_{\H_X}$ be a lift for $E_r$ with propagation $r$
 with respect to  the   semi-split filtered extension  $$0\to J_{\H_X}\to  A_{\H_X}\to C(X)_{\H_X}\to 0$$ such that $\|z_r\|\lq 1$.
 Let $g:[0,1]\to[0,1]$ be defined by
 \begin{itemize}
 \item $g(t)=0$ if $0\lq t\lq 1/4$;
 \item $g(t)=2(t-1/4)$ if $1/4\lq t\lq 3/4$;
 \item $g(t)=1$ if $3/4\lq t\lq 1$.
 \end{itemize}
Let us fix a positive number $\de$ and consider, with the notations of Section \ref{subsec-MV-pairs} the following elements
\begin{itemize}
\item $u:[0,1]\to A_{\H_X,\rho_X}$,
\item $v:[0,1]\to C(X)^{(1)}_{\H_X,\rho_X}+\C\cdot 1\ts \H_X$,
\item   $w:[0,1]\to C(X)^{(2)}_{\H_X,\rho_X}+\C\cdot 1\ts \H_X$
\end{itemize} defined  for $\theta$ in $[0,1]$ by 
 \begin{itemize}
 \item $u(\theta)=(1\ts\rho_X(\exp(2\imath\pi \theta g\circ f)))\cdot\exp(2\imath\pi (1-\theta) z_r)$ 
  \item  $v(\theta)=(1\ts  \rho_X(\exp(2\imath\pi\theta( g\circ f-1))))\cdot E_r+1\ts Id_{\H_X}-E_r$;
\item  $w(\theta)=E_r+
(1\ts\rho_X(\exp(2\imath\pi \theta g\circ f)))\cdot (1\ts Id_{\H_X}-E_r)$.
\end{itemize}
Since 
\begin{equation}
\label{eq-Er-almost-commutes0}\|(1\ts\rho_X(\exp(2\imath\pi\theta g\circ f)-1))\cdot E_r-
  ((\exp(2\imath\pi\theta g\circ f)-1)\ts Id_{\H_X})\cdot E_r\|<\eps/4\end{equation}
  and 
$$\|E_r\cdot(1\ts\rho_X(\exp(2\imath\pi\theta g\circ f)-1))-
  E_r\cdot((\exp(2\imath\pi\theta g\circ f)-1)\ts Id_{\H_X})\|<\eps/4$$

we deduce that
\begin{equation}
\label{eq-Er-almost-commutes}
\|(1\ts\rho_X(\exp(2\imath\pi\theta g\circ f)-1))\cdot E_r-
E_r\cdot(1\ts\rho_X(\exp(2\imath\pi\theta g\circ f)-1))\|
<\eps/2\end{equation} Hence, $v$ and $w$ are $\eps$-$r$-unitaries
and moreover,  $v-1\ts Id_{\H_X}$ lies in   
$C([0,1]\times X)^{(1)}_{\H_X,\rho_X}$  that   $w-1\ts Id_{\H_X}$ lies in   
$C([0,1]\times X)^{(2)}_{\H_X,\rho_X}$.

For next proposition, we need to recall from  \cite[Section 2]{oy4} the notion  of $\eps$-$r$-$N$-invertible  for positive number $\eps$, $r$ and $N$ with $\eps<1/100$: an element $x$ of a unital  $C^*$-algebra $B$ filtered by $(B_r)_{r>0}$ is  $\eps$-$r$-$N$-invertible if 
\begin{itemize}
\item $x$ is in $B_r$;
\item $\|x\|<N$;
\item there exists an element $y$ in $B_r$ with $\|y\|<N$ and such  that $\|xy-1\|<\eps$ and $\|yx-1\|<\eps$.
\end{itemize}
Such  an element $y$ is called an  $\eps$-$r$-$N$-inverse for $x$.

\begin{proposition}\label{prop-decompositionP1P2}

For some universal control pair $(\al,k)$ and positive number $N$ and with notations as above,  there exist 
$P_1$ and $P_2$   respectively $\al\eps$-$k_\eps r$-$N$-invertible elements in
 $M_2(A[0,1]_{\H_X,\rho_X}^{(1)})+\C\cdot I_2\ts Id_{\H_X}$ and $M_2(A[0,1]_{\H_X,\rho_X}^{(2)})+\C\cdot I_2\ts Id_{\H_X}$ 
 (with $A[0,1]=C([0,1],A)$) connected to $I_2$ as $\al\eps$-$k_\eps r$-$N$-invertible for any positive numbers $\eps$ and $r$ with
$\eps<\frac{1}{100\al}$ and $r<\frac{\delta}{5k_\eps}$  and such that
 \begin{enumerate}
 \item $\|P_1P_2-\diag(u,u^*)\|<\al \eps$;
  \item $\|\dot{P_1}-\diag(v,v^*)\|<\al \eps$, where $\dot{P_1}$ is the image of $P_1$  in  $$M_2(C(X)^{(1)}_{\H_X,\rho_X}[0,1])+\C\cdot I_2\ts Id_{\H_X}$$ relatively to  the extension 
 $$0\to A^{(1)}_{\H_X,\rho_X}[0,1]\to B^{(1)}_{\H_X,\rho_X}[0,1]{\to} C(X)^{(1)}_{\H_X,\rho_X}[0,1]\to 0;$$ 
 \item $\|\dot{P_2}-\diag(w,w^*)\|<\al \eps$, where $\dot{P_2}$ is the image of $P_2$  in  $$M_2(C(X)^{(2)}_{\H_X,\rho_X}[0,1])+\C\cdot I_2\ts Id_{\H_X}$$ relatively to the extension 
 $$0\to A^{(2)}_{\H_X,\rho_X}[0,1]\to B^{(2)}_{\H_X,\rho_X}[0,1]{\to} C(X)^{(2)}_{\H_X,\rho_X}[0,1]\to 0.$$
 \item $P_1(1)$  is  an  $\al\eps$-$k_\eps r$-$N$-invertible of $M_2(A_{\H_{X^{(\delta)}_1\cap X^{(\delta)}_2}})+\C\cdot I_2\ts Id_{\H_X}$ connected
 to   $I_2\ts Id_{\H_X}$ as an $\al\eps$-$k_\eps r$-$N$-invertible of  $M_2(A_{\H_{X^{(\delta)}_1\cap X^{(\delta)}_2}})+\C\cdot I_2\ts Id_{\H_X}$.\end{enumerate}
 Moreover,
 \begin{itemize}
 \item $P_1(0)$ and $P_2(0)$ lie respectively in $M_2(A_{\H_{X^{(\delta)}_1}})+\C\cdot 1\ts Id_{\H_X}$ and $M_2(A_{\H_{X^{(\delta)}_2}})+\C\cdot 1\ts Id_{\H_X}$  \item $P_1(1)-I_2\ts Id_{\H_X}$ and $P_1^{-1}(1)-I_2\ts Id_{\H_X}$ have coefficient in 
 $A_{\H_{X^{(\delta)}_1\cap X^{(\delta)}_2}}$.  \end{itemize}
\end{proposition}
\begin{proof}
 Let us fix a continuous function  $\phi:[0,1]\to[0,1]$  such that $\phi(1)=1$ and $\phi(t)=0$ if $t$ is in $[0,3/4]$ and let us  then fix
  $(\phi_n)_{n\in\N}$ a sequence of real polynomial functions  converging uniformly to $\phi$ on $[0,1]$ and such that $|\phi_n(t)|\lq2,\, \phi_n(0)=0$, $\phi_n(1)=1$  and $d^o\phi_n\gq n$ for every integer $n$ and every $t$ in $[0,1]$.
   For $\eps$ a positive number let $l_\eps$ be the smallest integer such that
   \begin{itemize}
   \item for all integer $k$ with $k>l_\eps$ and all $t$ in $[0,1]$ then $|\phi_k(t)-\phi(t)|<\eps$;
   \item $\displaystyle \sum_{k> l_\eps} \frac{(2\pi)^k}{k!}<\eps/3$.
   \end{itemize}

If $a$ and $b$ are elements in a  $C^*$-algebra and $\theta$ in $[0,1]$, let us set
   $$\psi^{(1)}_\eps(\theta,a,b)=\phi(a)\sum_{k=1}^{l_\eps} \frac{(2\imath\pi (1-\theta) b)^k}{k!}+(1-\phi(a))\cdot(\exp(-2\imath\pi \theta g(a))-1)\phi_{l_\eps}(b)$$
   and 
   $$\psi^{(2)}_\eps(\theta,a,b)=\exp(2\imath\pi\theta g(a))\cdot\left(\sum_{k=0}^{l_\eps} \frac{(2\imath\pi (1-\theta) b)^k}{k!}\right).$$   Notice that the degree in $b$ of $\psi^{(1)}_\eps(\theta,a,b)$ is less than $$h_\eps=\sup  \{d^o\phi_k;\,k=0,\ldots,l_\eps\}.$$
 
%
%
 We define  for any $\theta$ in $[0,1]$
 \begin{eqnarray*}u'(\theta)&=&\psi^{(2)}_\eps(\theta,1\ts \rho_X(f),z_r);\\
 y(\theta)&=&1\ts Id_{\H_X}+\psi^{(1)}_\eps(\theta,1\ts \rho_X(1-f),z_r)+\psi^{(1)}_\eps(\theta,1\ts \rho_X(f),1\ts Id_{\H_X}-z_r)^*;\\ 
  u_{1}(\theta)&=&1\ts Id_{\H_X}+\psi^{(1)}_\eps(\theta,1\ts \rho_X(1-f),z_r)+(1\ts \rho_X(1-g\circ f))\cdot (u'(\theta)-y(\theta))\\
  u_{2}(\theta)&=&\psi^{(1)}_\eps(\theta,1\ts \rho_X(f),1-z_r)^*+(1\ts \rho_X(g\circ f))\cdot(u'(\theta)-y(\theta))
 \end{eqnarray*}
 Then  for every $\theta$ in $[0,1]$ and $r<\frac{\de}{5h_\eps }$,
 \begin{itemize}
 \item $u'(\theta)$ and $y(\theta)$ are  elements of $A_{\H_X,\rho_X}$ with  propagation  less than $h_\eps r$;
 \item $u_{1}(\theta)-1\ts Id_{\H_X}$ is an element in  $A^{(1)}_{\H_X,\rho_X}$  with propagation less 
  than $h_\eps r$;
  \item $u_{2}(\theta)$ is an element in  $A^{(2)}_{\H_X,\rho_X}$  with propagation less 
  than $h_\eps r$;  \item $u'(\theta)=u_{1}(\theta)+u_{2}(\theta)$;
   \item  $\|u_{1}(\theta)\|\lq 15$ and  $\|u_{2}(\theta)\|\lq 15$.\end{itemize}  
   
   Moreover,$$\|u(\theta)-u'(\theta))\|<\eps/3$$
 and hence $u'(\theta)$ is an $\eps$-$h_\eps r$-unitary in $A_{\H_X,\rho_X}$.

 Let us set then 
$X(x)=\begin{pmatrix}1&x\\0&1\end{pmatrix}$ and $Y(y)=\begin{pmatrix}1&0\\y&1\end{pmatrix}$ and
define then for $\theta$ in $[0,1]$
$$P_{1}(\theta)=X(u_2(\theta))X(u_1(\theta))Y(-u_1^*(\theta))X(u_1(\theta)) \begin{pmatrix}0&-1\\1&0\end{pmatrix}X(-u_2(\theta))$$

and 

$$P_{2}(\theta)=X(u_2(\theta))\begin{pmatrix}0&1\\-1&0\end{pmatrix}X(-u_1(\theta))Y(-u_2^*(\theta))X(u_1(\theta))X(u_2(\theta))\begin{pmatrix}0&-1\\1&0\end{pmatrix}.$$
Then for some constant $N$ and some control pair $(\al,k)$ independent of $X\,,X_1\,X_2,\,A$ and $B$, for all $\theta$ in $[0,1]$ and for all positive number $r$ with $r<\frac{\delta}{5k_\eps}$,
\begin{itemize}
\item $P_{1}(\theta)$ and $P_2(\theta)$ are invertible elements in $M_2(A_{\H_X,\rho_X})$;
\item $P_{1}(\theta)-I_2$ and  $P_{1}(\theta)^{-1}-I_2$ have coefficients in  $A^{(1)}_{\H_X,\rho_X}$   and propagation less than $k_\eps r$;
 \item  $P_{2}(\theta)-I_2$ and  $P_{2}(\theta)^{-1}-I_2$ have coefficients in  $A^{(2)}_{\H_X,\rho_X}$   and propagation less than $k_\eps r$;
 \item $P_{1}(\theta),\,P_{2}(\theta)\,,P_{1}(\theta)^{-1}$ and $P_{2}(\theta)^{-1}$ are bounded in norm by $N$;
 \item $\|P_{1}(\theta)P_{2}(\theta)-\diag(u(\theta),u^*(\theta))\|<\alpha\eps$;
 \item $P_1(0)$ and $P_2(0)$ lie respectively in $M_2(A_{\H_{X^{(\delta)}_1}})+\C\cdot 1\ts \H_X$ and $M_2(A_{\H_{X^{(\delta)}_2}})+\C\cdot 1\ts \H_X$;
   \item $P_1(1)-I_2\ts\H_X$ and $P_1^{-1}(1)-I_2\ts\H_X$ have coefficients in 
 $A_{\H_{X^{(\delta)}_1\cap X^{(\delta)}_2}}$;
 \item  $P_1(1)$  is   connected
 to   $I_2\ts Id_{\H_X}$ as an $\al\eps$-$k_\eps r$-$N$-invertible of  $M_2(A_{\H_{X^{(\delta)}_1\cap X^{(\delta)}_2}})+\C\cdot I_2\ts Id_{\H_X}$. \end{itemize}
  Hence  if we set for $i=1,2$
  \begin{eqnarray*}
  P_i:[0,1]&\lto& M_2(A_{\H_X,\rho_X}) \\         
 \theta&\mapsto& P_{i}(\theta),
 \end{eqnarray*}
%
%
then $P_i$  is an $\eps$-$k_\eps r$-$N$-invertible element in
 $$M_2(A^{(i)}_{\H_X,\rho_X}[0,1])+\C\cdot I_2\ts Id_{\H_X}$$   connected to $I_2\ts Id_{\H_X}$ as an $\eps$-$k_\eps r$-$N$-invertible.

\smallbreak

Let us  consider the filtered semi-split extension $$0\to A^{(1)}_{\H_X,\rho_X}[0,1]\to B^{(1)}_{\H_X,\rho_X}[0,1]{\to} C(X)^{(1)}_{\H_X,\rho_X}[0,1]\to 0$$
and let us give an approximation of the image of $P_1$  in  the $C^*$-algebra $M_2({C(X)^{(1)}_{\H_X,\rho_X}[0,1]})+\C\cdot I_2\ts \H_X$.
In this computation, we can replace in the definition of  $P_1(\th)$.
\begin{itemize}
\item $u'(\th)$ by $$(1\ts \rho_X(\exp(2\imath\pi\th g\circ f )))\cdot \exp(2\imath\pi(1-\th)E_r);$$
\item $\psi^{(1)}_\eps(\theta,1\ts \rho_X(1-f),z_r)$ by
\begin{equation*}
\begin{split}
(1\ts\rho_X(\phi\circ(1-f))(e^{2\imath\pi(1-\th)}-1)&\\+
1\ts\rho_X((1-&\phi\circ (1-f))(\exp(-2\imath\pi\th g\circ(1-f))-1)))\cdot E_r\\
=&(1\ts\rho_X(\exp(-2\imath\pi\th g\circ(1-f))-1))\cdot E_r\\
=&(1\ts\rho_X(\exp(2\imath\pi\th (g\circ f -1))-1))\cdot E_r\\
\end{split}
\end{equation*}
\item $\psi^{(1)}_\eps(\theta,1\ts \rho_X(f),1\ts Id_{\H_X}-z_r)$ by
$$(1\ts\rho_X(\exp(-2\imath\pi\th g\circ f)-1)\cdot (1\ts Id_{\H_X}-E_r);$$
\item $y(\th)$ by
\begin{eqnarray*}
(1\ts\rho_X(\exp(2\imath\pi\th (g\circ f-1))))\cdot E_r+(1\ts\rho_X(\exp(2\imath\pi\th g\circ f))\cdot (1\ts Id_{\H_X}-E_r)
&=&u'(\th);
\end{eqnarray*}
\item 
$u_1(\th)$ by $$1\ts Id_{\H_X}-E_r+(1\ts\rho_X(\exp(2\imath\pi\th (g\circ f-1)))\cdot E_r;$$
\item  $u_2(\th)$ by $(1\ts\rho_X(\exp(2\imath\pi\th g\circ f)))\cdot(1\ts Id_{\H_X}-E_r)$.
\end{itemize}
In these computations, we have used the identities $(1-g)\phi=0$ and $g(t)+g(1-t)=1$ for all $t$ in $[0,1]$ and in the approximation of $y(\th)$ and of $u_2(\th)$, the fact that according to   \eqref{eq-Er-almost-commutes}, then $E_r$ and $\rho_X(\exp(2\imath\pi\th g\circ f-1)$ almost commute. Using again this almost commutation, we see that the image of $P_1$  in  $M_2({C(X)^{(1)}_{\H_X,\rho_X}[0,1]})+\C\cdot I_2\ts \H_X$ is closed to the unitary
\begin{eqnarray}\label{equ-unitary-1}
[0,1]&\to& M_2(C(X)^{(1)}_{\H_X,\rho_X})+\C\cdot I_2\ts \H_X\\
\nonumber\theta&\mapsto&
\diag(v_\theta,v^*_\theta)\end{eqnarray}
\smallbreak

In the same way,  if we consider the filtered semi-split extension $$0\to A^{(2)}_{\H_X,\rho_X}[0,1]\to B^{(2)}_{\H_X,\rho_X}[0,1]{\to} C(X)^{(2)}_{\H_X,\rho_X}[0,1]\to 0,$$
then the image of $P_2$  in  $M_2(C^{(2)}(X)_{\H_X,\rho_X}[0,1])+\C\cdot I_2\ts Id_{\H_X}$ is closed to the unitary
\begin{eqnarray}\label{equ-unitary-2}
[0,1]&\to&  M_2(C(X)^{(2)}_{\H_X,\rho_X})+\C\cdot I_2\ts \H_X\\
\nonumber\theta&\mapsto&\diag(w_\theta,w^*_\theta)\end{eqnarray}
with 
 for all $\theta$ in $[0,1]$.

 \end{proof}
 
 As a corollary we get with notations of Proposition \ref{prop-decompositionP1P2}
 \begin{corollary}\label{cor-decompW1W2}
 For some  control pair $(\lambda,h)$ depending only on the control pair $(\al,k)$ and on the positive number $N$ of Proposition \ref{prop-decompositionP1P2}, then for any positive numbers $\eps$ and $r$ with
$\eps<1/100\lambda$ and $r<\frac{\delta}{5h_\eps}$, there exist 
$W_1$ and $W_2$   respectively $\lambda\eps$-$h_\eps r$-unitary elements in
 $M_2(A[0,1]_{\H_X,\rho_X}^{(1)})+\C\cdot I_2\ts Id_{H_X}$ and $M_2(A[0,1]_{\H_X,\rho_X}^{(2)})+\C\cdot I_2\ts Id_{H_X}$ 
 connected to $I_2\ts Id_{\H_X}$ as $\lambda\eps$-$h_\eps r$-unitaries such that
 \begin{enumerate}
 \item $\|W_1W_2-\diag(u,u^*)\|<\lambda \eps$;
  \item $\|\dot{W_1}-\diag(v,v^*)\|<\lambda \eps$, where $\dot{W_1}$ is the image of $W_1$  in  $$M_2(C(X)^{(1)}_{\H_X,\rho_X}[0,1])+\C\cdot I_2\ts Id_{\H_X}$$ relatively to  the extension 
 $$0\to A^{(1)}_{\H_X,\rho_X}[0,1]\to B^{(1)}_{\H_X,\rho_X}[0,1]{\to} C(X)^{(1)}_{\H_X,\rho_X}[0,1]\to 0;$$ 
 \item $\|\dot{W_2}-\diag(w,w^*)\|<\lambda \eps$, where $\dot{W_2}$ is the image of $W_2$  in  $$M_2(C(X)^{(2)}_{\H_X,\rho_X}[0,1])+\C\cdot I_2\ts Id_{\H_X}$$ relatively to the extension 
 $$0\to A^{(2)}_{\H_X,\rho_X}[0,1]\to B^{(2)}_{\H_X,\rho_X}[0,1]{\to} C(X)^{(2)}_{\H_X,\rho_X}[0,1]\to 0.$$
 \item $W_1(0)$ and $W_2(0)$ lie  respectively in $M_2(A_{\H_{X^{(\delta)}_1}})+\C\cdot 1\ts Id_{\H_X}$ and $M_2(A_{\H_{X^{(\delta)}_2}})+\C\cdot 1\ts Id_{\H_X}$  \item $W_1(1)$  is  an  $\lambda\eps$-$h_\eps r$-unitary of  $M_2(A_{\H_{X^{(\delta)}_1\cap X^{(\delta)}_2}})+\C\cdot I_2\ts Id_{\H_X}$ connected
 to   $I_2\ts Id_{\H_X}$ as a $\lambda\eps$-$l_\eps r$-unitary which belongs to   $M_2(A_{\H_{X^{(\delta)}_1\cap X^{(\delta)}_2}})+\C\cdot I_2\ts Id_{\H_X}$.\end{enumerate} 
 \end{corollary}
 \begin{proof}
 In view of \cite[Lemma 2.4]{oy4}, there exists a control pair $(\lambda,h)$ and a positive number $N'$  depending only on $N$ and on the controlled pair $(\al,k)$ of  Proposition \ref{prop-decompositionP1P2} for which  the following is satisfied :
 
 \smallbreak
 
 For any positive number $\eps$ with $\eps<\frac{1}{100\lambda}$, there exists   a polynomial function $Q_\eps$ of degree less than $h_\eps$  with $Q_\eps(1)=1$ such that $Q_\eps(P_1^*P_1)$ and $Q_\eps(P_2^*P_2)$ respectively admit positive $\la\eps$-$h_\eps r$-$N'$-inverse in  $M_2(A^{(1)}_{\H_X,\rho_X}[0,1])+\C\cdot 1\ts Id_{\H_X}$ and $M_2(A^{(2)}_{\H_X,\rho_X}[0,1])+\C\cdot 1\ts Id_{\H_X}$ and such that if we set
$W_1=P_1Q_\eps(P_1^*P_1)$ and $W_2=Q_\eps(P_2P_2^*)P_2$,
then 
\begin{itemize}
\item $W_1$ and $W_2$ are $\lambda\eps$-$h_\eps r$-unitaries in $M_2(A_{\H_X,\rho_X}[0,1])+\C\cdot 1\ts Id_{\H_X}$;
\item $W_1-I_2\ts Id_{\H_X}$ has coefficents  in  $A^{(1)}_{\H_X,\rho_X}[0,1]$;
\item $W_2-I_2\ts Id_{\H_X}$ has coefficents  in  $A^{(2)}_{\H_X,\rho_X}[0,1]$;
\item $\|P_1-|P_1|W_1\|<\lambda\eps$ and $\|P_2-W_2|P^*_2|\|<\lambda\eps$;
\item $W_1$ and $W_2$ are respectively homotopic to $I_2\ts Id_{\H_X}$ as $\lambda\eps$-$h_\eps r$-unitaries in
 $M_2(A^{(1)}_{\H_X,\rho_X}[0,1] )  +\C\cdot I_2\ts Id_{\H_X}$  and $M_2(A^{(2)}_{\H_X,\rho_X}[0,1] )  +\C\cdot I_2\ts Id_{\H_X}$;
 \item $W_1(0)$ and $W_2(0)$ lie respectively in $M_2(A_{\H_{X^{(\delta)}_1}})+\C\cdot 1\ts \H_X$ and $M_2(A_{\H_{X^{(\delta)}_2}})+\C\cdot 1\ts \H_X$. \end{itemize}
 Proceeding as in the proof of  \cite[Proposition 2.11]{oy4} and possibly rescaling  $\la$, we see that $\|W_1W_2-\diag(u,u^*)\|<\lambda\eps$.
 Moreover,  since $\dot{P_1}$ is closed to a unitary, then $|\dot{P_1}|$ is closed to $I_2\ts \H_{X}$ in 
$M_2(C(X)^{(1)}_{\H_X,\rho_X}[0,1])+\C\cdot I_2\ts Id_{\H_X}$ and therefore,  up to rescaling $\lambda$, we get that 
$\|\dot{W_1}-\diag(v,v^*)\|<\lambda \eps$ and similarly that $\|\dot{W_2}-\diag(w,w^*)\|<\lambda \eps$.

Since $Q_\eps(1)=1$ and since  $P_1(1)-I_2\ts Id_{\H_X}$ has coefficients in $A_{\H_{X^{(\delta)}_1}\cap \H_{X^{(\delta)}_2}}$ , then  the same holds for  $W_1(1)-I_2\ts Id_{\H_X}$ and 
$W_1(1)$  is connected
 to   $I_2\ts \H_X$ as a  $\lambda\eps$-$l_\eps r$-unitary of  $M_2(A_{\H_{X^{(\delta)}_1\cap X^{(\delta)}_2}})+\C\cdot I_2\ts \H_X$.
 \end{proof}

{\bf End of proof of Proposition  \ref{proposition-index-commutes}.}
With notations of Lemma \ref{lemma-bondary-index}, this amounts to proving that
 $$\D^{(\delta)}_{X_1,X_2,A,*}\circ Ind_{X,A,*}=\partial Ind^{{(\delta})}_{X_1,X_2,A,*}.$$ Let us give the proof for  $K_1(X,A)$, the even case being deduced from controlled Bott periodicity.
 Then any element of $K_1(X,A)$  corresponds to a  semi-split extension
  $$0\to A\to B\to C(X)\to 0,$$ where $A$ is closed two-sided  ideal in $B$. We can assume without loss of generality that $B$ is unital.         
 
\medskip

Recall that
\begin{itemize}

\item $f:X\to[0,1]$ is a continuous function such that   $\supp f\subseteq \mathring{X_2}$ and $\supp 1-f \subseteq \mathring{X_1}$;
\item $r_0$ is a positive number with $r_0\lq r$ such that $|f(x)-f(x')|<\frac{\eps}{16\pi}$ if $d(x,x')<r_0$;
\item the function $g:[0,1]\to[0,1]$ is defined by $g(t)=0$ if $0\lq t\lq 1/4$, $g(t)=2(t-1/4)$ if $1/4\lq t\lq 3/4$ and  $g(t)=1$ if $3/4\lq t\lq 1$;
\item $E_r(x)$ is a projection in $C(X)\ts \Kp(\H_X)$ with  support in $B(x,r')\times B(x,r')$ with $r'=\frac{1}{2}\min\{r_0,r\}$ and such that  $[E_X]_{\eps,r}=[E_r,0]_{\eps,r}$ for every positive numbers $\eps$ and $r$ with $\eps<1/100$;
\item $u(\theta)=(1\otimes \rho_X(\exp({2\imath\pi \theta g\circ f)}))\cdot\exp(2\imath\pi (1-\theta) z_r)$ for all $\theta$ in $[0,1]$ with $z_r$ a lift with propagation $r$ of
$E_r$ in $A_{\H_X}$. 
\end{itemize}
In order to have a more readable  end of proof of the proposition, we divide it  into three steps.

\medskip

{\bf Step 1: }
Let $(\la,h)$ be a control pair and let $W_1$ and $W_2$   be respectively $\lambda\eps$-$h_\eps r$-unitary elements in
 $M_2(A[0,1]_{\H_X,\rho_X}^{(1)})+\C\cdot I_2\ts Id_{H_X}$ and $M_2(A[0,1]_{\H_X,\rho_X}^{(2)})+\C\cdot I_2\ts Id_{H_X}$  as in Corollary \ref{cor-decompW1W2}. According to points (ii) and (iii) of Corollary \ref{cor-decompW1W2},  up to enlarging $\la$ and since for $i=1,2$ the extension
$$0\to A^{(i)}_{\H_X,\rho_X}[0,1]\to B^{(i)}_{\H_X,\rho_X}[0,1]{\to} C(X)^{(i)}_{\H_X,\rho_X}[0,1]\to 0$$ is semi-split filtered,
 there exists an element $y_1$ in  
$M_2(A^{(1)}_{\H_X,\rho_X}[0,1])+\C\cdot 1\ts Id_{\H_X}$ with propagation less than $2h_\eps$ such
that 
$$\|W^*_1\diag(I_2\ts Id_{\H_X},0) W_1-y_1-\diag(I_2\ts Id_{\H_X},0)\|<\la\eps$$ 
 and an element $y_2$ in  
$M_2(A^{(2)}_{\H_X,\rho_X}[0,1])+\C\cdot I_2\ts Id_{\H_X}$ with propagation less than $2h_\eps$ such
that 
$$\|W_2\diag(I_2\ts Id_{\H_X},0) W_2^*-y_2-\diag(I_2\ts Id_{\H_X},0)\|<\la\eps.$$ 
Since $W_2$ is close to $W^*_1\cdot\diag(u,u^*)$, we can assume up to rescaling $\la$ that 
$$\|y_1-y_2\|<\la\eps.$$
Using the CIA property and  up to rescaling the control pair $(\la,h)$, we deduce that there exists  $q$  a $\la\eps$-$h_\eps r$-projection in $M_2(A_{\H_X,\rho_X}[0,1])$
such that
\begin{itemize}
\item $q-\diag(1\ts Id_{\H_X},0)$ has coefficients in 
 $$A^{(1,2)}_{\H_X,\rho_X}[0,1]=A[0,1]\ts\Kp(\H_{X^{(\delta)}_1}\cap \H_{X^{(\delta)}_2})+A[0,1]\ts\rho_X(C_0(\mathring{X_1}\cap \mathring{X_2}));$$
 \item $\|y_1+\diag(1\ts Id_{\H_X},0)-q\|<\la\eps$ and $\|y_2+\diag(1\ts Id_{\H_X},0)-q\|<\la\eps$.
 \end{itemize}

\medskip

{\bf Step 2: }
Let
 $$\D_{A_{\H_X},B_{\H_X},*}=( \partial^{\eps,r}_{A_{\H_X},B_{\H_X},*})_{0<\eps<\frac{1}{100\al_\D},r>0}:\K_*(C(X)_{\H_X}) \lto\K_{*+1}(A_{\H_X})$$ be the controlled boundary map  associated to the semi-split filtered extension
 $$0\lto A_{\H_X}\lto B_{\H_X}\lto C(X)_{\H_X}\lto 0$$ and 
 \begin{equation*}\begin{split}\D_{A_{\H_{X^{(\delta)}_1\cap X^{(\delta)}_2}},B_{\H_{X^{(\delta)}_1\cap X^{(\delta)}_2}},*}=( \partial^{\eps,r}_{A_{\H_{X^{(\delta)}_1\cap X^{(\delta)}_2}},B_{\H_{X^{(\delta)}_1\cap X^{(\delta)}_2}},*})_{0<\eps<\frac{1}{100\al_\D},r>0}&:\\
 \K_*(C(X)_{\H_{X^{(\delta)}_1\cap X^{(\delta)}_2}}) &\lto\K_{*+1}(A_{\H_{X^{(\delta)}_1\cap X^{(\delta)}_2}})\end{split}\end{equation*} be  the controlled boundary associated to the semi-split filtered extension
 $$0\lto A_{\H_{X^{(\delta)}_1\cap X^{(\delta)}_2}}\lto B_{\H_{X^{(\delta)}_1\cap X^{(\delta)}_2}}\lto C(X)_{\H_{X^{(\delta)}_1\cap X^{(\delta)}_2}}\lto 0.$$ 

 Since $W_1(0)$ and $W_2(0)$ lie indeed respectively in $M_2(A_{\H_{X^{(\delta)}_1}})+\C\cdot 1\ts Id_{\H_X}$ and $M_2(A_{\H_{X^{(\delta)}_2}})+\C\cdot 1\ts Id_{\H_X}$ and  since $W_1(0)W_2(0)$ is close  to $\diag(\exp(2\imath\pi z_r),\exp(-2\imath\pi z_r))$, we see that, up to rescaling $(\la,h)$ and possibly $(\al_\D,k_\D)$, the class $[q(0),1]_{\la\eps,h_\eps r}$ coincides with the image of 
 \begin{equation*}\iota_*^{-,\la\eps,h_\eps r}\circ 
 \partial^{(\delta),\al_\D\eps,k_{\D,\eps}r}_{X_1,X_2,*}\circ \partial^{\eps,r}_{A_{\H_X},B_{\H_X},*}([E_X]_{\eps,r})\end{equation*} under the morphism \begin{equation}\label{equ-end-proof}\K_*(A_{\H_{X^{(\delta)}_1\cap X^{(\delta)}_2}})\to \K_*(A^{(1,2)}_{\H_X,\rho_X})\end{equation} induced by the inclusion $$\Kp(\H_{\H_{X^{(\delta)}_1\cap X^{(\delta)}_2}})\hookrightarrow\Kp(\H_{\H_{X^{(\delta)}_1\cap X^{(\delta)}_2}})+\rho_X(C_0(\mathring{X_1}\cap \mathring{X_2})).$$
 On the other hand, 
 \begin{itemize}
 \item according to points (ii)  and (v) of Corollary \ref{cor-decompW1W2};
 \item since $$v(1)=(1\ts  \rho_X(\exp(2\imath\pi   \circ f)))\cdot E_r+1\ts Id_{\H_X}-E_r$$ is according of Equation \eqref{eq-Er-almost-commutes0}\ close and  hence  homotopic to   $$(\exp(2\imath\pi   \circ f)\ts  Id_{H_X})\cdot E_r+1\ts Id_{\H_X}-E_r$$ as an $\eps$-$r$-unitary in $C(X)_{X^{(\delta)}_1\cap X^{(\delta)}_2}$,
 \end{itemize}
  we see that up to rescaling $(\la,h)$, then
 $[q(1),1]_{\la\eps,h_\eps r}$ coincides with the image of $$\iota^{-,\la\eps,h_\eps r}_*\circ\partial^{\eps,r}_{A_{\H_{X^{(\delta)}_1\cap X^{(\delta)}_2}},B_{\H_{X^{(\delta)}_1\cap X^{(\delta)}_2}},*}([U^{(\delta)}_{X_1,X_2}]_{\eps,r})$$   under the morphism of equation (\ref{equ-end-proof}).

 \medskip

{\bf Step 3: }
 As we have seen in the previous step,  $$\iota^{-,\la\eps,h_\eps r}_*\circ\partial^{\eps,r}_{A_{\H_{X^{(\delta)}_1\cap X^{(\delta)}_2}},B_{\H_{X^{(\delta)}_1\cap X^{(\delta)}_2}},*}([U^{(\delta)}_{X_1,X_2}]_{\eps,r})$$   and 
 $$\iota_*^{-,\la\eps,h_\eps r}\circ 
 \partial^{(\delta),\al_\D\eps,k_{\D,\eps}r}_{X_1,X_2,*}\circ \partial^{\eps,r}_{A_{\H_X},B_{\H_X},*}([E_X]_{\eps,r})$$
have same images  under the morphism of equation (\ref{equ-end-proof}).
 Since the extension
 $$0\lto A_{\H_{X^{(\delta)}_1\cap X^{(\delta)}_2}}\lto A^{(1,2)}_{\H_{X^{(\delta)}_1\cap X^{(\delta)}_2}}\lto A\ts \rho_X(C_0(V_1\cap V_2)) \lto 0$$ is filtered and split, then according to  \cite[Corollary 4.9]{oy2}, there exists a control pair $(\la',h')$ such that   the  morphism of equation (\ref{equ-end-proof})  is  $(\lambda',h')$-injective as defined in Definition \ref{def-exact-seq}. Hence we get the result for $K_1(X,A)$ at least for small $\eps$. For $\eps$ in $(0,1/100)$, the result is a consequence of the first point of Remark \ref{rem-quant-morphism-small-propagation}.

%

 \medskip

 Let us prove now the result in the even case.
  Let us consider the following diagram
{ {\verytiny $$\begin{CD}K_0(X,A)@>\ts\tau_{A}([\partial])>> K_1(X,SA)@>\ts[\partial_{X^{(\de)}_1,X^{(\de)}_2}]>> K_{0}({X}^{(\delta)}_1\cap {X}^{(\delta)}_2,SA)@>\ts\tau_{A}([\partial]^{-1})>>K_{1}({X}^{(\delta)}_1\cap {X}^{(\delta)}_2,A)\\
         @VV Ind_{X,A,*} V @V Ind_{X,SA,*} VV         @V Ind_{X^{(\de)}_1,X^{(\de)}_2,SA,*}VV @V Ind_{X^{(\de)}_1,X^{(\de)}_2,A,*}VV \\
\K_0(A(\H_X))@>\T_{H_X}(\tau_{A}([\partial]))>>\K_{1}(SA(\H_{{X}^{(\delta)}_1\cap {X}^{(\delta)}_2} ))@>\D^{(\delta)}_{X_1,X_2,SA,*}>>\K_{0}(SA(\H_{{X}^{(\delta)}_1\cap {X}^{(\delta)}_2} ))@> \T_{\H_X}(\tau_{A}([\partial]^{-1}))>>\K_{1}(A(\H_{{X}^{(\delta)}_1\cap {X}^{(\delta)}_2} ))
\end{CD},
$$}}
where $\ts\bullet$ stands for the left Kasparov product.
According to Proposition \ref{prop-compatibility-index} and to the even case, the three squares  are  commutative. The composition on the top row
is  
$$\ts[\partial_{X^{(\de)}_1,X^{(\de)}_2}]: K_0(X,A)\lto K_{1}({X}^{(\delta)}_1\cap {X}^{(\delta)}_2,A).$$
The result for $\eps$ small enough is then a consequence of the definition of 
$$\D^{(\delta)}_{X_1,X_2,A}:\K_{0}(SA(\H_{{X}^{(\delta)}_1\cap {X}^{(\delta)}_2} ))\lto \K_{1}( A(\H_{{X}^{(\delta)}_1\cap {X}^{(\delta)}_1})$$ (see \cite[Section 3.2]{oy4}).
For $\eps$ in $(0,1/100)$, the result is a consequence of the first point of Remark \ref{rem-quant-morphism-small-propagation}.
 \qed

\subsection{Proof of Theorem \ref{thm-qindex-iso}}\label{sec-proof-thm-iso}
Let $X$ be a compact metric space, let $A$ be a $C^*$-algebra, let $\H_X$ be a non degenerate standard $X$-module and   let $(0,1/200)\to [1,+\infty);\, \eps\mapsto h_\eps$ be a non increasing function. Recall from definition \ref{def-h-isomorphism} that $$Ind_{X,A,*}:K_*(X,A)\lto \K(A_{\H_X})$$ is a $h$-isomorphism if
for any $\eps$ in $(0,1/200)$.
\begin{itemize}
\item $Ind^{\eps,r}_{X,A,*}$ is one-to-one for any $r$ in $(0,1/h_\eps)$;
\item for any $r$ in $(0,1/h^2_\eps)$ and any $y$ in $K^{\eps,r}_*(A_{\H_X})$, there exists
$x$ in $K_*(X,A)$ such that $Ind^{\eps,h_\eps r}_{X,A,*}(x)=\iota_*^{\eps,r,h_\eps r}(y)$.
\end{itemize}
The proof of  Theorem  \ref{thm-qindex-iso} will be by induction. We first prove a Mayer-Vietoris type stability result for $h$-isomorphism of 
$Ind_{X,A,*}$.
\begin{lemma}\label{lemma-5-lemma}
There exists a
non increasing function 
 $$(0,1/200)\to [1,+\infty);\, \eps\mapsto k_\eps$$ such that the following is satisfied. Assume we are given
 \begin{itemize}
  \item $X$ a compact metric space and $X_1$ and $X_2$ two closed subsets  
 such  that $X=\mathring{X_1}\cup \mathring{X_2}$;
 \item  non increasing functions
 $(0,1/200)\to [1,+\infty);\, \eps\mapsto h_\eps$ and $(0,1/200)\to [1,+\infty);\, \eps\mapsto h'_\eps$;
\item  $\de>\frac{5}{\inf\{h_\eps h'_\eps;\,\eps\in (0,1/200)\}}$ and for
 $i=1,2$ a closed subset $X_i^{(\de)}$ with 
 $\{x\in X\text{ such that }d(x,X_i)\lq \de\}\subseteq X_i^{(\de)}$
 \end{itemize}
 for which
 $$Ind_{X_1^{(\de)}\cap X_2^{(\de)},A,*}:K_*(X_1^{(\de)}\cap X_2^{(\de)},A)\to \K(A_{\H_{1,2}})$$ is 
 for 	any 
  non degenerate standard $X_1^{(\de)}\cap X_2^{(\de)}$-module $\H_{1,2}$ a $h$-isomorphism 
 and
 $$Ind_{X_i^{(\de)},A,*}:K_*(X_i^{(\de)},A)\to \K(A_{\H_{i}})$$ is  for any non degenerate standard $X_i^{(\de)}$-module $\H_i$ a $h'$-isomorphism for $i=1,2$.

 \smallbreak
 
 Then,$$Ind_{X,A,*}:K_*(X,A)\to \K(A_{\H_X})$$ is a $khh'$-isomorphism for any non degenerate standard $X$-module 
$\H_X$.
\end{lemma} 
\begin{proof}
The lemma is proved by using a $5$-lemma argument. Let $(\al,k)$ be a control pair such that the controlled Mayer-Vietoris sequence of Section \ref{subsec-MV-pairs} is $(\al,k)$-exact. We can assume without loss of generality that
$\al=2^{n_0}$ for some integer $n_0$. We can also assume that the restrictions of $\H_X$ to $X_1^{(\de)}$, $X_2^{(\de)}$ and $X_1^{(\de)}\cap X_2^{(\de)}$ are also standard and then use the notations of Section \ref{subsec-MV-pairs}. Let us first prove the injectivity part. Let us prove indeed that
$Ind^{\eps,r}_{X,A,*}$ is one-to-one for any $\eps$ in $(0,\frac{1}{200\al})$ and any $r$ in $(0,\frac{1}{k_\eps h^2_{\al \eps} h'_{\al \eps}})$. The general case will follow using a rescaling argument. Let $x$ be an element of $K_*(X,A)$ such that $Ind_{X,A,*}^{\eps,r}(x)=0$ in
$K^{\eps,r}_*(A_{\H_X})$. In view of Proposition \ref{proposition-index-commutes}, we have
$$Ind^{\al\eps,k_\eps r}_{X_1^{(\de)}\cap X_2^{(\de)},A,*}\circ \partial_{X_1^{(\de)},X_2^{(\de)},*}(x)=0$$
with $\al\eps$ in $(0,1/200)$ and 
$k_\eps r<\frac{1}{h^2_{\al \eps} h'_{\al \eps}}<\frac{1}{h_{\al\eps}}$.
Since by assumption, $Ind^{\al\eps,k_\eps r}_{X_1^{(\de)}\cap X_2^{(\de)},A,*}$ is one-to-one, we deduce that $\partial_{X_1^{(\de)},X_2^{(\de)},*}(x)=0$ in 
$K_*(X_1^{(\de)}\cap X_2^{(\de)},A)$.
Using the Mayer-Vietoris exact sequence in $K$-homology for the pair
$(X_1^{(\de)}, X_2^{(\de)})$, we deduce with notations of Section \ref{subsec-MV-K-hom} that there exist 
$x'_1$ in $K_*(X_1^{(\de)},A)$ and $x'_2$ in $K_*(X_2^{(\de)},A)$ such that
$$\jmath_{X_1^{(\de)},*}(x'_1)-\jmath_{X_2^{(\de)},*}(x'_2)=x.$$ Let us set then 
$y'_1=Ind^{\eps, r}_{X_1^{(\de)},A,*}(x'_1)$ and 
$y'_2=Ind^{\eps, r}_{X_2^{(\de)},A,*}(x'_2)$.
By Lemma \ref{functoriality-index-in-space}  and with notations of Section \ref{subsec-MV-pairs}, we deduce that 
$$\jmath^{\sharp,(\delta),\eps,r}_{X_1,A,*}(y'_1)-\jmath^{\sharp,(\delta),\eps,r}_{X_2,A,*}(y'_2)=0$$ and hence, by controlled exactness of the quantitative Mayer-Vietoris sequence, we deduce that there exists $y''$ in $K_*^{\al\eps,k_\eps r}(A_{\H_{X_1^{(\de)}\cap X_2^{(\de)}}})$ such that
$$\jmath^{\sharp,(\delta),\al\eps,k_\eps r}_{X_1,X_2,A,*}(y'')=\iota_*^{-,\al\eps,k_\eps r}(y'_1)$$ and
$$\jmath^{\sharp,(\delta),\al\eps,k_\eps r}_{X_2,X_1,A,*}(y'')=\iota_*^{-,\al\eps,k_\eps r}(y'_2).$$
Since $\al\eps\in (0,1/200)$ and $k_\eps r<1/ h^2_{\al\eps}$, then by assumption
on $Ind_{X_1^{(\de)}\cap X_2^{(\de)},A,*}$, there exists 
$x''$ in $K_*(X_1^{(\de)}\cap X_2^{(\de)},A)$ such that
$$Ind^{\al\eps,h_{\al\eps}k_\eps r}_{X_1^{(\de)}\cap X_2^{(\de)},A,*}(x'')=\iota_*^{-,\al\eps,h_{\al\eps}k_\eps r}(y'').$$
Since  $k_\eps h_{\al\eps}r<1/ h'_{\al\eps}$, then 
$Ind^{\al\eps,h_{\al\eps}k_\eps r}_{X_i^{(\de)},A,*}$ is one-to-one for $i=1,2$ and hence, we deduce from Lemma \ref{functoriality-index-in-space} that
$x'_1=\jmath_{X_1^{(\de)},X_2^{(\de)},*}(x'')$ and $x'_2=\jmath_{X_2^{(\de)},X_1^{(\de)},*}(x'')$.  We deduce from this that $x=0$ and hence $Ind_{X,A,*}^{\eps,r}$ is one-to-one if
$\eps$ is in $(0,\frac{1}{200\al})$ and  $r$ is in 
$(0,\frac{1}{k_\eps h^2_{\al \eps} h'_{\al \eps}})$.
 Since we have chosen $\al=2^{n_0}$, we see using 
Proposition \ref{prop-compatibility-index-rescaling} that
$Ind_{X,A,*}^{\eps,r}$ is one-to-one if
$\eps$ is in $(0,\frac{1}{200})$ and  $r$ is in $(0,\frac{1}{k'_\eps h_{\eps} h'_{\eps}})$
with $k'_\eps=l_{\eps}\cdots l_{\eps/2^{n_0-1}}k_{\eps/2^{n_0}}$ and 
$l=l_{rs}$ (with notations of Section \ref{subsec-rescaling}). 

\smallbreak

Let us prove now the surjectivity statement. We can assume without loss of generality that 
$(\al,k)$ is greater than the control pair of \cite[Corollary 3.6]{oy4}.
Let us set $l'_\eps=l_{\eps}\cdots l_{\eps/2^{n_0-1}}$  with $l_\eps$ as above and $k'_\eps=l'_\eps k_\eps k_{\eps/\al}$.  Let us prove first that for any $\eps$ in $\left(0,\frac{1}{200\al}\right)$, any $r$ in 
$\left(0,\frac{1}{k'_\eps h_{\al\eps}{h'}^2_{\al\eps}}\right)$  and any $y$ in 
$K^{\eps,r}_*(A_{\H_X})$, there exists $x_y$ in  $K_*(X,A)$ such that 
$$Ind_{X,A,*}^{\al\eps, k'_\eps h_{\al\eps}h'_{\al\eps}r}(x_y)=
\iota_*^{-,\al\eps,k'_\eps h_{\al\eps}h'_{\al\eps}r}(y).$$
With notations as above, let us set
$$y'=\iota_*^{-,\al\eps,k_\eps r}\circ \partial^{(\delta),\eps,r}_{X_1,X_2,A,*}(y).$$
Since $\al\eps$ is in $(0,1/200)$ and
 $k_\eps r<\frac{1}{l'_\eps k_{\eps/\al} h_{\al\eps}{h'}^2_{\al\eps}}\lq \frac{1}{h^2_{\al\eps}}$, there exists
$x'$ in $K_{*+1}(X_1^{(\de)}\cap X_2^{(\de)},A)$ such that 
$$Ind_{X,A,*}^{\al\eps,h_{\al\eps}k_\eps r}(x')=\iota_*^{-,\al\eps,h_{\al\eps}k_\eps r}(y').$$
Let us set then 
$$x'_1=\jmath_{X_1^{(\de)},X_2^{(\de)},*}(x')\in K_{*+1}(X_1^{(\de)},A)$$ and 
$$x'_2=\jmath_{X_2^{(\de)},X_1^{(\de)},*}(x')\in K_{*+1}(X_2^{(\de)},A).$$
According to Lemma \ref{functoriality-index-in-space}, we have
$$Ind_{X_i^{(\de)},A,*}^{\al\eps,h_{\al\eps}k_\eps r}(x'_i)=0$$ for $i=1,2$ and hence, since
$\al\eps\in(0,1/200)$ and $$h_{\al\eps}k_\eps r< \frac{1}{h'_{\al\eps}},$$  we deduce by injectivity of
$Ind_{X_i^{(\de)},A,*}^{\al\eps,h_{\al\eps}k_\eps r}$
that $x'_i=0$ for $i=1,2$.
Hence, by exactness of the Mayer-Vietoris sequence in $K$-homology, there exists an element
$x$ in $K_*(X,A)$ such that 
$$\partial_{X_1^{(\de)},X_2^{(\de)},*}(x)=x'.$$
Using Proposition \ref{proposition-index-commutes}, we see that
$$\iota_*^{-,\al\eps,k_\eps h_{\al\eps}r}\circ \partial^{(\delta),\eps,r}_{X_1,X_2}(Ind_{X,A,*}^{\eps,r}(x)-y)=0$$ in
$K_{*+1}^{\al\eps,k_\eps h_{\al\eps}r}(A_{\H_{X_1^{(\de)}\cap X_2^{(\de)}}}).$ Let us set
$$y_\al=L_{A_{\H_{X}},*}^{\eps/2^{n_0-1}, l_{\eps/2^{n_0-2}}\cdots l_\eps r}\circ\cdots \circ L_{A_{\H_{X}},*}^{{\eps/2}, l_{\eps}r}\circ L_{A_{\H_{X}},*}^{\eps, r}(y).$$ Then $y_\al$ lies in $K^{\eps/\al,l'_\eps r}_*(A_{\H_X})$ and in view of
Lemma \ref{lemma-compatibility-MV-boundary-rescaling} and of Proposition \ref{prop-compatibility-index-rescaling}, we see that
$$\iota_*^{-,\eps,k_{\eps/\al} h_{\al\eps}l'_\eps r}\circ \partial^{(\delta),\eps/\al,l'_\eps r}_{X_1,X_2}(Ind_{X,A,*}^{\eps/\al,l'_\eps r}(x)-y_\al)=0$$ in
$K_{*+1}^{\eps,k_{\eps/\al} h_{\al\eps}l'_\eps r}(A_{\H_{X_1^{(\de)}\cap X_2^{(\de)}}}).$

Then according to \cite[Corollary 3.6]{oy4}, there exists $y_i$ in 
$K_{*}^{\al\eps,k'_\eps h_{\al\eps}r}(A_{\H_{X_i^{(\de)}}})$ such that
\begin{eqnarray*}
\jmath^{\sharp,(\delta),\al\eps,k'_\eps h_{\al\eps}r}_{X_1,A,*}(y_1)-
\jmath^{\sharp,(\delta),\al\eps,k'_\eps  h_{\al\eps}r}_{X_2,A,*}(y_2)&=&
Ind_{X,A,*}^{\al\eps,k'_\eps h_{\al\eps}r}(x)-\iota_*^{-,\al\eps, k'_\eps h_{\al\eps} r}(y_\al)\\
&=&
Ind_{X,A,*}^{\al\eps,k'_\eps  h_{\al\eps}r}(x)-\iota_*^{-,\al\eps,k'_\eps h_{\al\eps} r}(y),
\end{eqnarray*}
where the second equality holds in view of the first point of Lemma \ref{lemma-rescaling} 
(recall that $k'_\eps=k_\eps k_{\eps/\al}l'_\eps$).
Since $\al\eps$ is in $(0,1/200)$ and $k'_\eps  h_{\al\eps}r<\frac{1}{h'^2_{\al\eps}}$, we see by surjectivity assumption for $Ind_{X_i^{(\de)},A,*}$ for $i=1,2$ that   there exists $x_i$ in $K_*(X_i^{(\de)},A)$ such that
$$Ind_{X_i^{(\de)},A,*}^{\al\eps,k'_\eps h_{\al\eps}h'_{\al\eps}r}(x_i)
=\iota_*^{-,\al\eps,k'_\eps  h_{\al\eps}h'_{\al\eps}r}(y_i).$$
Eventually we set $$x_y=x-\jmath_{X_1^{(\de)},*}(x_1)+\jmath_{X_2^{(\de)},X_2^{(\de)},*}(x_2)$$ in
$K_*(X,A)$. Using Lemma \ref{functoriality-index-in-space}, we obtain that
$$\iota_*^{-,\al\eps,k'_\eps  h_{\al\eps}h'_{\al\eps}r}(y)=
Ind_{X,A,*}^{\al\eps,k'_\eps h_{\al\eps}h'_{\al\eps}r}(x_y).$$ 
\smallbreak

Now, for $\eps$ in $(0,1/200)$, $r$ in $\left(0,\frac{1}{l'_\eps k'_{\eps/\al}h_\eps h'^2_{\eps}}\right)$ and $y$ in $K^{\eps,r}_*(A_{\H_X})$, we can  apply the preceding result to
$$L_{A_{\H_X},*}^{\eps/2^{n_0-1}, l_{\eps/2^{n_0-2}}\cdots l_{\eps/2}l_\eps r}\circ\cdots\circ
 L_{A_{\H_X},*}^{\eps/2,l_\eps r} \circ L_{A_{\H_X},*}^{\eps,r}(y).$$ We deduce from
 Proposition \ref{prop-compatibility-index-rescaling} that there exist $x$ in 
 $K_*(X,A)$ such that 
 $$Ind_{X,A,*}^{\eps,l'_\eps k'_{\eps/\al} h_{\eps}h'_{\eps}r}(x)=y.$$

\end{proof}
Let us first prove Theorem \ref{thm-qindex-iso} for  a finite simplicial complexe of dimension $0$.
\begin{lemma}\label{lemma-induction}
 Let $Z_0$ be  a finite metric set such that for some $\de>0$,  then $d(x,x')\gq \de$ for every distinct  points $x$ and $x'$ in $Z_0$ and  let $A$ be a $C^*$-algebra.
 Then
 $Ind^{\eps,r}_{Z_0,A,*}$ is an isomorphism for any $\eps$ in $(0,1/200)$ and any $r$ in $(0,\de)$.
 \end{lemma}
 \begin{proof}
 We have  an  isomorphism
\begin{equation}\label{eq-induction}K_*(Z_0,A)\stackrel{\cong}{\lto} K_*(A)^{Z_0}\end{equation} induced by the family of evaluation at points in $Z_0$.
Any non degenerate standard $Z_0$-module  $(\rho_{Z_0},\H_{Z_0})$ is unitary equivalent to  $(\rho,\ell^2(Z_0)\ts\H)$ where $\H$ is a separable Hilbert space and $\rho$ is the pointwise multiplication. Hence we can assume without loss of generality that 
 $(\H_{Z_0},\rho_{Z_0})=(\ell^2(Z_0)\ts\H,\rho)$ and then we have a  Morita equivalence 
 $K_*^{\eps,r}(A_{\H_{Z_0}}))\cong K_*^{\eps,r}(A\ts \Kp(\ell^2(Z_0)))$.
 Any operator in   $A\ts \Kp(\ell^2(Z_0))$  with propagation less than $\de$ is actually diagonal. Hence if $r$ is a positive number with $r<\de$, then any $\eps$-$r$ projection of  $A\ts \Kp(\ell^2(Z_0))$ is given by a family  $(q_z)_{z\in Z_0}$ of self-adjoint elements in $A$ such that $\|q^2_z-q_z\|<\eps$. Since $\eps$ lies in $(0,1/200)$, then the spectra of $q_z$ lies in $(-\infty,1/5)\cup(4/5,+\infty)$. Let us fix a continuous function $\kappa_0:\R\to [0,1]$ sucht that
 \begin{itemize}
 \item  $\kappa_0(t)=0$ if $t\lq1/5$ \item $\kappa_0(t)=1$ if $\ t \gq 4/5$. 
 \end{itemize}
 If $P_n(A)$ stands for the set of projections in $M_n(A)$, it is then straightforward to see that
 $$P^{\eps,r}_n(A\ts \Kp(\ell^2(Z_0)))\lto P_n(A)^{Z_0};\,\diag(q_z)_{z\in Z_0}\mapsto (\kappa(q_z))_{z\in Z_0}$$ induces an isomorphism
 $$K^{\eps,r}_0(A\ts \Kp(\ell^2(Z_0)))\stackrel{\cong}{\lto} K_0(A)^{Z_0}.$$ Similarly, in the odd case, the map 
 $$U^{\eps,r}(A\ts \Kp(\ell^2(Z_0)))\lto GL(A)^{Z_0};\,\diag(u_z)_{z\in Z_0}\mapsto (u_z)_{z\in Z_0}$$ induces an isomorphism
 $$K^{\eps,r}_1(A\ts \Kp(\ell^2(Z_0)))\stackrel{\cong}{\lto} K_1(A)^{Z_0}.$$ Under these identifications and that of equation (\ref{eq-induction}), then $$Ind_{Z_0,A,*}:K_*(Z_0,A)\lto \K_*(A\ts \Kp(\H_{Z_0}))$$ is the identity map for any positive number $\eps$ and $r$ with $\eps< 1/200$ and $r<\de$.
\end{proof}

As a first application of Lemma \ref{lemma-5-lemma},
we have the following estimates for boundaries of simplices.
Let $\partial\De_n$ be the boundary of the standard simplex  of dimension $n$ equipped with the normalized spherical metric. 
\begin{lemma}\label{lemma-index-iso-boundary-simplex}
There exists a non increasing function $$(0,1/200)\to [1,+\infty);\, \eps\mapsto h_\eps,$$ such that for any integer $n$,  for any $C^*$-algebra $A$  and for any non degenerate standard $\partial\De_n$-module $\H_{\partial\De_n}$, then
$$Ind_{\partial\De_n,A,*}:K_*(\partial\De_n,A)\lto K_*(A_{\H_{\partial\De_n}})$$ is a $h^n$-isomorphism.
\end{lemma}
\begin{proof}
Let us view $\partial\De_n$ as embedded in $\S_n$ and let us consider for $\al$ in $(0,1)$
$$\partial\De^+_{n,\al}=\left\{\cos \th\cdot v+\sin\th \cdot e_{n+1};\, \th \in\left[\frac{\al\pi}{2},\pi/2\right], \, v\in \S_{n-1}\subseteq \S_{n}\right\}\bigcap \partial\De_n$$
and
$$\partial\De^-_{n,\al}=\left\{\cos \th\cdot v+\sin\th \cdot e_{n+1};\, \th \in\left[0,\frac{\al\pi}{2}\right], \, v\in \S_{n-1}\subseteq \S_{n}\right\}\bigcap \partial\De_n.$$
Then for any $\beta$ in $(0,\al)$, we have
$$\{v\in \partial\De_n\text{ such that } d(v,\partial\De^+_{n,\al})<\beta\}\subseteq  \partial\De^+_{n,\al-\beta}$$ and for 
any $\beta$ in $(0,1-\al)$, we have
$$\{v\in \partial\De_n\text{ such that } d(v,\partial\De^-_{n,\al})<\beta\}\subseteq  \partial\De^-_{n,\al+\beta}.$$
For $\de$ in $(0,1/6)$, we set $X_1= \partial\De^+_{n,1/3}$, $X^{\de}_1= \partial\De^+_{n,1/3-\de}$, $X_2= \partial\De^-_{n,2/3}$ and 
 $X^{\de}_2= \partial\De^-_{n,2/3+\de}$.

We have thus that
\begin{itemize}
\item $X^{\de}_1$ can be contracted to $\{e_{n+1}\}$ by  a $1$-Lipschitz homotopy;
\item $X^{\de}_2$ is homotopic to $\De_{n-1}$ by using a $3/2$-Lipschitz homotopy which in turn can be contracted to its center by a $1$-Lipschitz homotopy;
\item $X^{\de}_1\cap X^{\de}_2$ is homotopic to \begin{equation}\label{eq-rescal-bound-simplex}\left\{\frac{1}{\sqrt{2}}(v+e_{n+1});v\in \S_{n}\right\}\bigcap \partial\De_n\end{equation} by a $3$-Lipschitz homotopy. The metric subspace of equation \eqref{eq-rescal-bound-simplex} is a rescaling of 
 $\partial\De_{n-1}$ with a factor $\frac{1}{\sqrt{2}}$.
 Using equation \eqref{eq-composition} of Section \ref{subsec-coarse-geom-nds-mod},   Proposition \ref{prop-homotopy} and Lemma \ref{functoriality-index-in-space}  and in view of Lemma \ref{lemma-induction}, we see that
  $Ind_{X_i^{(\de)},A,*}$ is a $C$-isomorphism for some constant $C$ that does not depend on $n$ and that we can choose such that $C>30$.
 Pick then $\delta$ in $(5/C,1/6)$, the result is  then by induction a consequence of Lemma \ref{lemma-5-lemma}.

 \end{itemize}
\end{proof}
We are now in a position to achieve our goal.

\smallbreak

{\bf Proof of Theorem \ref{thm-qindex-iso} :} 
Let $X$ be a connected finite simplicial complexe of dimension $n$. Let us prove the result by induction
 for $k=0,\ldots,n$ and   any simplicial subcomplexe   of $X$ of dimension $k$. According to Lemma \ref{eq-induction}, the result clearly holds for subcomplexes of dimension $0$.
Let us fix $k$ in $1,\ldots,n$ and let $Z$ be  simplicial subcomplexe   of $X$ of dimension $k$.  Let $Z'_0$ be the set of centers of $k$-simplexes in $Z$. For any $\si$ in $Z_0'$ let
 $\De_\si$ be the $k$-dimension simplex that contains $\si$ and define  for any $\la$ in $(0,1)$ the set

$$\De_{\si,\la}=\{x\in\De_\si;\,\exists x'\in\De_\si\,\slash\, d(x,x')=(1-\la) d(x',\si)\text{ and }
d(x,\si)=\la d(x',\si)\}.$$
In other word, an element $x$ of  $\De_\si$ belongs to $\De_{\si,\la}$ if there exists $x'$ in $\De_\si$ such that $x$ is in the geodesic arc between $\si$ and $x'$ at distance $\la d(x',\si)$ from $\si$ and $(1-\la) d(x',\si)$ from $x'$.
Let us set then $$Z_1=\bigsqcup_{\si\in Z'_0}\De_{\si,2/3},$$

 $$Z_2=\overline{Z\setminus \bigsqcup_{\si\in Z'_0}\De_{\si,1/3}},$$

and for $\de$ in $\left(0,\frac{2}{3\pi^2(k+1)}\right)$,

$$Z_1^{(\de)}=\bigsqcup_{\si\in Z'_0}\De_{\si,2/3+\de\pi\sqrt{\frac{k+1}{4}}},$$

and 

 $$Z_2^{(\de)}=\overline{Z\setminus \bigsqcup_{\si\in Z'_0}\De_{\si,1/3-\de\pi\sqrt{\frac{k+1}{4}}}}.$$
 A lenghty but straightforward computation shows that  for $\left(0,\frac{2}{3\pi^2(k+1)}\right)$ and $i=1,2$, we have
 $$\{x\in Z \text{ such that }d(x,Z_i)<\de\}\subseteq Z_i^{(\de)}.$$
Using equation \eqref{eq-composition} of Section \ref{subsec-coarse-geom-nds-mod},   Proposition \ref{prop-homotopy} and Lemma \ref{functoriality-index-in-space}, we see that   
\begin{itemize}
\item $\De_{\si,2/3+\de\pi\sqrt{\frac{k+1}{4}}}$ being homotopic to $\{\si\}$ by a $C$-Lipschitz homotopy with $C$ which  depends neither on $\si$, $k$ nor  $\de$, we deduce that $Z_1^{(\de)}$ is homotopic to $Z'_0$ by a $C$-Lipschitz homotopy.
 Noticing that the distance between two points in $Z_0'$ in greater than $\frac{4}{\pi\sqrt{k+1}}$,   we get according to Lemma \ref{lemma-induction} that $Ind_{Z_1^{(\de)},A,*}$ is a $C'\sqrt{k+1}$-isomorphism for $C'>1$  that depends neither on $k$ nor $\de$.
\item since $$Z_1^{(\de)}\cap Z_2^{(\de)}=\bigsqcup_{\si\in Z_0'}\De_{\si,2/3+\de\pi\sqrt{\frac{n+1}{4}}}\bigcap
\overline{\De_\si\setminus \De_{\si,1/3-\de\pi\sqrt{\frac{n+1}{4}}}}$$
with
$$\De_{\si,2/3+\de\pi\sqrt{\frac{n+1}{4}}}\bigcap
\overline{\De_\si\setminus \De_{\si,1/3-\de\pi\sqrt{\frac{n+1}{4}}}}$$ homotopic to 
the boundary of $\De_\si$ by a $C$-Lipschitz homotopy with $C$ which  depends neither on $k$ nor $\de$, then according to 
Lemma \ref{lemma-index-iso-boundary-simplex}, there exists
a non increasing function  
 $$(0,1/200)\to [1,+\infty);\, \eps\mapsto h_\eps$$ 
 that we can choose satisfying $h^n>10\pi^2(n+1)$ for all positive integer $n$ and 
 such that
$Ind_{Z_1^{(\de)}\cap Z_2^{(\de)},A,*}$ is a $h^k$-isomorphism.
\end{itemize}
Since $Z_2^{(\de)}$ is homotopic to the $k-1$-skeletton of $Z$ by a    $C$-Lipschitz homotopy with $C$ independent  of $k$, and in view once again of equation \eqref{eq-composition} of Section \ref{subsec-coarse-geom-nds-mod},   Proposition \ref{prop-homotopy} and Lemma \ref{functoriality-index-in-space}, we can then apply
Lemma \ref{lemma-5-lemma} to $Z=Z_1\cup Z_2$, $h$, $h'=C'\sqrt{k+1}$ and $Z_1^{\de}$ and 
$Z_2^{\de}$
for $\de$ in $\left(\frac{1}{2\pi^2(n+1)},\frac{2}{3\pi^2(n+1)}\right)$. The result follows then  by induction.

\section{Quantitative assembly map estimates and coarse decomposability}\label{section-QAM-coarse-decomposability}In this section, we give the proof of Theorem \ref{theorem-main}.

\subsection{Preliminary}
\begin{definition}
Let  $X$ and $Y$ be two discrete proper metric spaces such that $Y\subseteq X$ and let $s$ be a positive number. The metric space $Y$ is called a $s$-net for $X$ if for any element $x$ in $X$ there exists an element $y$ in $Y$ such that $d(x,y)<s$.
\end{definition}
\begin{lemma}\label{lemma-coarse-eq}
Let $X$ be a proper metric space, let $s$ be a positive number, let $Y$ be a $s$-net for $X$ and let  $\H$ be a separable Hilbert space.
Then for any  positive number $d$ and for  any  rank $1$ projection $e$ on $\H$,
there exists a continuous map
$$\phi_d:P_d(X)\lto P_{d+2s}(Y)$$ and a unitary map $$V_d:\ell^2(X,\H)\to \ell^2(Y,\H)$$ such that
\begin{enumerate}
\item for any $z$ in $P_d(X)$, then $\supp\phi_d(z)$ has support contained in $$\{y\in Y\text{ such that }d(y,\supp z)<s\}$$ (where $\supp z$ is the support of $z$ as a probability measure on $X$);
\item $V_d$  has support contained  in $$\{(x,y)\in X\times Y\text{ such that }d(x,y)<s\};$$
\item for any compact subset $Z$ of $P_d(X)$  and for any positive number $r$  with  $r>2d+4s$, then $Ad_{V_d,C(Z)}\cdot (P_Z\ts e)$ and $\phi_{d,\ell^2(X,\H)}^*(P_{\phi_d(Z)}\ts e)$ are homotopic as $\eps$-$(r+2s)$-projections in
$C(Z)_{\ell^2(Y,\H)}$, where 
\begin{itemize}
\item $\phi_{d,\ell^2(X,\H)}^*:C(\phi_d(Z))_{\ell^2(X,\H)}\lto C(Z)_{\ell^2(X,\H)}$ is the map induced  by $\phi_d$; \item $P_Z$ and $P_{\phi_d(Z)}$  are respectively the projections  in $C(Z)_Y=C(Z)\ts\Kp(\ell^2(Y))$ and $C(\phi_d(Z))_X=C(\phi_d(Z))\ts\Kp(\ell^2(X))$\ constructed in Section  \ref{sub-sec-geo-ass-map}.\end{itemize}\end{enumerate}
\end{lemma}
\begin{proof} Let $(\la_{X,x})_{x\in X}$ and  $(\la_{Y,y})_{y\in Y}$ be respectively the coordinate systems for the Rips complexes of $X$  and $Y$.
For each $x$ in $X$, pick $\tilde{x}$ in $Y$ such that $d(x,\tilde{x})<s$ and set $X_y=\{x\in X,\, \tilde{x}=y\}$ for any $y$ in $Y$. Define then
$\phi_d:P_d(X)\lto P_{d+2s}(Y)$   by
$$\la_{Y,y}(\phi_d(z))=\sum_{x\in X_y}\la_{X,x}(z)$$ for any $z$ in $P_d(X)$ and any $y$ in $Y$. Then $\phi_d(z)$ is clearly a probability measure supported in $$\{y\in Y\text{ such that }d(y,\supp z)<s\}.$$ Let $\eta$ be a vector in $\H$ with $\|\eta\|=1$ such that $e$ is the projection on $\C\cdot\eta$ and let us choose for any $y$ in $Y$ a unitary $$W_y:\ell^2(X_y)\ts \H\lto \H$$ such that 
$$W_y\cdot \left(\frac{1}{\sqrt{\sharp X_y}}\chi_{X_y}\ts \eta\right)=\eta,$$ where $\chi_{X_y}$ is the characteristic function of ${X_y}$ viewed as an element in $\ell^2(X_y)$. Define then
$$V_d=\bigoplus_{y\in Y} W_y:\ell^2(X)\ts\H\cong\bigoplus_{y\in Y} \ell^2(X_y)\ts \H\lto \ell^{2}(Y)\ts\H.$$
The support of $V_d$ is obviously in $\{(x,y)\in X\times Y\text{ such that }d(x,y)<s\}$.
Now for any $z$ in $Z$, then $\phi_{d,\Kp(\ell^2(X)\ts\H)}^*\cdot (P_{\phi_d(Z)}\ts e)(z)$ is the projection associated to 
$$\left(\left(\sum_{x\in X_y} \la_{X,x}(z)\right)^{1/2}\ts \eta\right)_{y\in Y}.$$
Define for all $x$ in $X$
$$\la'_x=\frac{1}{\sharp X_{\tilde{x}}}\sum_{x'\in X_{\tilde{x}}}\la_{X,x'}.$$
Consider  then for any $t$ in $[0,1]$ and any $x$ in $Z$ the projection $q_t(z)$ on $$(((1-t)\la_{X,x}(z)+t\la'_x(z))^{1/2}\ts\eta)_{x\in X}.$$ If $r>2d+4s$, then $(q_t)_{t\in[0,1]}$ is a projection in $C(Z)\ts\Kp(\ell^2(X)\ts \H)$ with propagation less than $r$ such that 
$q_0=P_Z\ts e$ and  $$Ad_{V_d,C(Z)}\cdot q_1=\phi_{d,\Kp(\ell^2(X)\ts\H)}^*\cdot (P_{\phi_d(Z)}\ts e).$$
\end{proof}
\begin{remark}\label{remark-coarse-inclusion}
Let $A$ be a $C^*$-algebra and let $d,\,s,\,X$  and $Y$ be as in the lemma and recall that $A_X=A\ts \Kp(\ell^2(X))$ and $A_Y=A\ts \Kp(\ell^2(Y))$.
\begin{enumerate}
\item According to Lemma \ref{lemma-adjoint}, then the compositions
$$\Kp(\ell^2(X)\ts \H)\stackrel{Ad_{V_d}}{\lto}\Kp(\ell^2(Y)\ts \H)\hookrightarrow\Kp(\ell^2(X)\ts \H)$$  and

$$\Kp(\ell^2(Y)\ts \H)\hookrightarrow\Kp(\ell^2(X)\ts \H)\stackrel{Ad_{V_d}}{\lto}\Kp(\ell^2(Y)$$  respectively  induce for $r>2d+4s$  and up to Morita equivalence
$$\iota_*^{\eps,r,r+2s}:K_*^{\eps,r}(A_X){\lto}K_*^{\eps,r+2s}(A_X)$$  and

$$\iota_*^{\eps,r,r+2s}:K_*^{\eps,r}(A_Y)\lto K_*^{\eps,r+2s}(A_Y).$$
\item The compositions
$$P_d(X)\stackrel{\phi_d}{\lto} P_{d+2s}(Y)\hookrightarrow  P_{d+2s}(X)$$ and
$$P_d(Y)\hookrightarrow P_{d}(X) \stackrel{\phi_d}{\lto} P_{d+2s}(Y)$$ are respectively homotopic to the inclusions
$P_d(X)\hookrightarrow P_{d+2s}(X)$ and $P_d(Y)\hookrightarrow P_{d+2s}(Y)$ and hence induce respectively the morphisms
$$q^X_{d,d+2s,*}:K_*(P_d(X),A)\lto K_*(P_{d+2s}(X),A)$$ and $$q^Y_{d,d+2s,*}:K_*(P_d(Y),A)\lto K_*(P_{d+2s}(Y),A).$$
\item $$q_{d+2s,d'+2s,*}^Y\circ \phi_{d,*}=\phi_{d',*}\circ q^X_{d,d';*}$$ for any positive number $d'$ with $d'\gq d$;
\item $$\iota_*^{\eps,r+2s,r'+2s}\circ Ad^{\eps,r}_{V_d,*}=Ad^{\eps,r'}_{V_{d'},*}\circ \iota_*^{\eps,r,r'}$$ for any positive number $r'$ with $r'\gq r$, $r\gq  2d+4s$ and  $r'\gq  2d'+4s$.\end{enumerate}
\end{remark}
\begin{lemma}\label{lemma-s-net}
Let $\X$ and $\Y$ be families of proper discrete metric spaces. Assume that there exists a positive number $s$ such that \begin{itemize}
\item for any $X$ in $\X$, there exists a set $Y$ in $\Y$  contained in  $X$ as a $s$-net;
\item for any $Y$ in $\Y$, there exists $X$ in $\X$ that contains $Y$ as a $s$-net.
\end{itemize}
Then the following assertions are equivalent:
\begin{enumerate}
\item the family $\X$ satisfies the QAM-estimates;
\item the family $\Y$ satisfies the QAM-estimates.
\end{enumerate}
\end{lemma}
 \begin{proof}
 Let $X$ be an element of $\X$ and let $Y$ be a $s$-net for $X$ that belongs to $\Y$, let $$\jmath^\sharp_{X,Y}:\Kp(\ell^2(Y))\to  \Kp(\ell^2(X))$$ be the inclusion and  for any $C^*$-algebra $A$, let  $$\jmath_{P_d(X),P_d(Y),*}:K_*(P_d(Y),A)\lto K_*(P_d(X),A)$$  be the morphism induced by the inclusion $P_d(Y)\hookrightarrow P_d(X)$. Then  according to  Lemma \ref{lemma-subset}, we have
\begin{equation}\label{equ-incl}\nu^{\eps,r,d}_{A,X,*}\circ \jmath_{P_d(X),P_d(Y),*}=\jmath^{\sharp,\eps,r}_{X,Y,A,*}\circ \nu^{\eps,r,d}_{A,Y,*},\end{equation} where
$$\jmath^{\sharp}_{X,Y,A,*}:\K_*(A_Y)\lto \K_*(A_X)$$ is the morphism induced by $\jmath^\sharp_{X,Y}$. 
According to Remark 
 \ref{remark-naturality} and Lemma \ref{lemma-coarse-eq}, we also have (up to Morita equivalence)
 \begin{equation}\label{equ-phi}
Ad^{\eps,r}_{V_d,A,*}\circ \nu^{\eps,r,d}_{A,X,*} =\nu^{\eps,r+2s,d+2s}_{A,Y,*}\circ \phi_{d,*}\end{equation}  
for every number $r$ such that $r\gq r_{d+2s,\eps}$,
 where $$\phi_{d,*}:K_*(P_d(X),A)\lto K_*(P_{d+2s}(X),A)$$ is the morphism induces by $\phi_d$ as in Lemma \ref{lemma-coarse-eq}.
 In view of Remark \ref{remark-coarse-inclusion} and using Equations (\ref{equ-incl}) and (\ref{equ-phi}), then for any positive numbers $\eps,\,d,\,d',\,r$ and $r'$ with $\eps<\frac{1}{200}$, then
 \begin{itemize}
 \item $QI_{X,A,*}(d,d',r,\eps)$ implies $QI_{Y,A,*}(d,d'+2s,r,\eps)$ if  $r\gq r_{d,\eps}$ and $d'\gq d$;
 \item $QI_{Y,A,*}(d+2s,d',r+2s,\eps)$ implies $QI_{X,A,*}(d,d'+2s,r,\eps)$ if  $r\gq r_{d+2s,\eps}$ and $d'\gq d+2s$; 
 \item $QS_{X,A,*}(d,r,r',\eps)$ implies $QS_{Y,A,*}(d+2s,r,r'+2s,\eps)$ if  $r'\gq r$ and $r'> r_{d+2s,\eps}-2s$; 
  \item $QS_{Y,A,*}(d,r+2s,r',\eps)$ implies $QS_{X,A,*}(d,r,r',\eps)$ if $r'\gq r+2s$ and $r'>r_{d,\eps}$.\end{itemize}
 \end{proof}
\subsection{Notations}\label{subsection-notations}For convenience of the readers, we  collect in this section some   notations that we shall use throughout the proof of Theorem \ref{theorem-main}.
 \begin{itemize}
 \item If $Z$ is compact   metric space, then for any  subset $Y$  of $Z$ and any positive number $s$, we   set $Y_s=\{z\in Z\text{ such that }d(Y,z)\lq s\}$;
\item If $X_1$ and $X_2$ are two subsets of a finite metric space $X$ and $A$ is a $C^*$-algebra, then $A_{X_1,X_2,r}=\left(A\ts\Kp(\ell^2(X_1),\ell^2(X_2))\right)\cap A_{X,r}$ (recall that $A_X=A\ts \Kp(\ell^2(X))$);
  \item if $(\H,\rho)$ is a $Z$-module,
 \begin{itemize}\item  we set $A_\H=A\ts \Kp(\H)$. Then $A_\H$ is filtered by $$(A_{\H,r})_{r>0}=(A\ts \Kp(\H)_r)_{r>0};$$ 
 \item if $\H'$ is a closed linear subspace of $\H$, we view $A_{\H'}$ as a corner sub-$C^*$-algebra of $A_{\H}$ and $Id_{\H'}$ as the projection onto $\H'$;
 \item 
 If $\H'$ and $\H''$ are two closed linear subsets of $\H$, we view $$A_{\H',\H''}\defi A\ts\Kp(\H',\H'')$$ as a closed subset of $A_\H$
 and we set $A_{\H',\H'',r}=A_{\H',\H''}\cap A_{\H,r}$;\end{itemize}
 \item let $X$ be a finite metric space, let $d$ be a positive number and let $(\H_d,\rho_d)$ be a $P_d(X)$-module. 
 \begin{itemize}
 \item $V_{\H_d}:\H_d\to \ell^2(X,\H_d)$ is the isometry defined by $$(V_{\H_d}\cdot\xi)(x)=\rho(\la^{1/2}_x)\cdot \xi$$ for any $\xi$ in $\H_d$, where $(\la_x)_{x\in X}$ is the coordinates system in $P_d(X)$.
 \end{itemize}
 \end{itemize}
 Let $X$ be a finite metric space and let $X^{(1)}$ and $X^{(2)}$ be two  subsets of $X$ such that
 $X^{(1)}\cup X^{(2)}$. Let $\de_0$ be a non negative number and let $\de_1$ and $r$ be positive numbers such that $\de_1> 5r$. With notations as above  and for any $C^*$-algebra $A$, 
 \begin{equation}\label{equ-MVpair1}\left(A_{X^{(1)}_{\de_0},X^{(1)}_{\de_0+\de_1},r},A_{X^{(2)}_{\de_0},X^{(2)}_{\de_0+\de_1},r},A_{X^{(1)}_{\de_0+\de_1}},A_{X^{(2)}_{\de_0+\de_1}}\right)\end{equation} is a $r$-controlled  Mayer-Vietoris pair for $A_X$.
We denote respectively by $$\jmath^{(1),\sharp}_{X,A}:A_{X^{(1)}_{\de_0+\de_1}}\lto A_{X},$$  $$\jmath^{(2),\sharp}_{X,A}:A_{X^{(2)}_{\de_0+\de_1}}\lto A_{X},$$  $$\jmath^{(1,2),\sharp}_{X,A}:A_{X^{(1)}_{\de_0+\de_1}\cap X^{(2)}_{\de_0+\de_1}}\lto A_{X^{(1)}_{\de_0+\de_1}}$$
and   $$\jmath^{(2,1),\sharp}_{X,A}:A_{X^{(1)}_{\de_0+\de_1}\cap X^{(2)}_{\de_0+\de_1}}\lto A_{X^{(2)}_{\de_0+\de_1}}$$  the morphisms of filtered $C^*$-algebras  induced  by the inclusion $\ell^2(X^{(1)}_{\de_0+\de_1})\hookrightarrow \ell^2(X)$, $\ell^2(X^{(2)}_{\de_0+\de_1})\hookrightarrow \ell^2(X)$,  $\ell^2(X^{(1)}_{\de_0+\de_1}\cap X^{(2)}_{\de_0+\de_1})\hookrightarrow \ell^2(X^{(1)}_{\de_0+\de_1})$ and $\ell^2(X^{(1)}_{\de_0+\de_1}\cap X^{(2)}_{\de_0+\de_1})\hookrightarrow \ell^2(X^{(2)}_{\de_0+\de_1})$.
Let \begin{equation*}\D_{X^{(1)}_{\de_0+\de_1},X^{(2)}_{\de_0+\de_1},A,*}=(\partial_{X^{(1)}_{\de_0+\de_1},X^{(2)}_{\de_0+\de_1},A,*}^{\eps,s})_{0<\eps<\frac{1}{4\al},0<s<\frac{r}{k_{\eps}}}:
\K_*(A_{X})\lto\K_*(A_{X^{(1)}_{\de_0+\de_1}\cap X^{(2)}_{\de_0+\de_1}})\end{equation*} be the associated controlled Mayer-Vietoris boundary map (of  order $r$) with $(\al,k)=(\al_\mv,k_\mv)$.

\smallskip

 Let   $d$ be a  positive number  and let $(\H_d,\rho_d)$ be a non degenerate $P_d(X)$-module.
  For $i=1,2$, and $\de>0$, let  
 $\H_{d,\de}^{(i)}$ be the   set of elements of $\H_{d}$ with support in $P_{d}(X^{(i)}_{\de})$.
Assume now that  $\de_0>2d$ and $\de_1> (2d+1)(5r+1)$.
It is straightforward to check that
 \begin{equation}\label{eq-MV-ripps}P_d(X)=\widering{{P_d(X_{\de_0}^{(1)})}}\cup \widering{P_d(X_{\de_0}^{(2)})}.\end{equation}
Let us  describe the $r$-controlled Mayer-Vietoris for $A_{\H_d}$  associated to the decomposition $P_d(X)={{P_d(X_{\de_0}^{(1)})}}\cup {P_d(X_{\de_0}^{(2)})}$.
 Since $P_d(Y)_s\subseteq P_d(Y_{2d(s+1)})$ for any subset $Y$ of $X$ and any positive number $s$, 
 \begin{equation}\label{equ-MVpair2}\left(A_{\H_{d,\de_0}^{(1)},\H_{d,\de_0+\de_1}^{(1)},r},A_{\H_{d,\de_0}^{(2)},\H_{d,\de_0+\de_1}^{(1)},r}, A_{\H_{d,\de_0+\de_1}^{(1)}}, A_{\H_{d,\de_0+\de_1}^{(2)}}\right)\end{equation}  is a $r$-controlled Mayer-Vietoris pair for $A_{\H_d}$.
 We denote respectively by 
 $$\jmath^{(1),\sharp}_{P_d(X),A}:A_{\H_{d,\de_0+\de_1}^{(1)}}\lto A_{\H_d},$$
  $$\jmath^{(2),\sharp}_{P_d(X),A}:A_{\H_{d,\de_0+\de_1}^{(2)}}\lto A_{\H_d},$$  
  $$\jmath^{(1,2),\sharp}_{P_d(X),A}:A_{\H_{d,\de_0+\de_1}^{(1)}\cap\H_{d,\de_0+\de_1}^{(2)}} \lto A_{\H_{d,\de_0+\de_1}^{(1)}}$$ and 
  $$\jmath^{(2,1),\sharp}_{P_d(X),A}:A_{\H_{d,\de_0+\de_1}^{(1)}\cap\H_{d,\de_0+\de_1}^{(2)}}\lto A_{\H_{d,\de_0+\de_1}^{(2)}}$$ 
  the morphisms of filtered $C^*$-algebras induced by the inclusions $\H_{d,\de_0+\de_1}^{(1)}\hookrightarrow \H_d$, $\H_{d,\de_0+\de_1}^{(2)}\hookrightarrow \H_d$, $\H_{d,\de_0+\de_1}^{(1)}\cap\H_{d,\de_0+\de_1}^{(2)} \hookrightarrow\H_{d,\de_0+\de_1}^{(1)}$ and $\H_{d,\de_0+\de_1}^{(1)}\cap\H_{d,\de_0+\de_1}^{(2)} \hookrightarrow\H_{d,\de_0+\de_1}^{(2)}$.
  
  On the same way, for $K$-homology, 
  we denote respectively by 
 $$\jmath^{(1)}_{P_d(X),*}:K_*(P_d(X^{(1)}_{\de_0+\de_1}),A)\lto K_*(P_d(X),A),$$
 $$\jmath^{(2)}_{P_d(X),*}:K_*(P_d(X^{(2)}_{\de_0+\de_1}),A)\lto K_*(P_d(X),A),$$
  $$\jmath^{(1,2)}_{P_d(X),*}:K_*(P_d(X^{(1)}_{\de_0+\de_1}\cap X^{(2)}_{\de_0+\de_1}),A)\lto K_*(P_d(X^{(1)}_{\de_0+\de_1}),A)$$ and 
   $$\jmath^{(2,1)}_{P_d(X),*}:K_*(P_d(X^{(1)}_{\de_0+\de_1}\cap X^{(2)}_{\de_0+\de_1}),A)\lto K_*(P_d(X^{(2)}_{\de_0+\de_1}),A),$$
  the morphisms  induced by the inclusions 
   $P_d(X^{(1)}_{\de_0+\de_1})\hookrightarrow  P_d(X)$, 
   $P_d(X^{(2)}_{\de_0+\de_1})\hookrightarrow  P_d(X)$,
   $P_d(X^{(1)}_{\de_0+\de_1}\cap X^{(2)}_{\de_0+\de_1})\hookrightarrow P_d(X^{(1)}_{\de_0+\de_1})$ and
    $P_d(X^{(1)}_{\de_0+\de_1}\cap X^{(2)}_{\de_0+\de_1})\hookrightarrow P_d(X^{(2)}_{\de_0+\de_1})$.
   
    Let \begin{equation*}\D_{X,d,\de_0,\de_1,A,*}=(\partial^{\eps,r}_{X,d,\de_0,\de_1,A,*})_{0<\eps<\frac{1}{100\al},0<s<\frac{r}{k_{\eps}}}:
\K_*(A_{\H_d})\lto\K_*(A_{\H_{d,\de_0+\de_1}^{(1)}\cap\H_{d,\de_0+\de_1}^{(2)} })\end{equation*} be the associated controlled Mayer-Vietoris boundary map at order $r$ with $(\al,k)=(\al_\mv,k_\mv)$.

\medskip
Assume now that there exists a positive number $R$ such for $i=1,2$ then $$X^{(i)}=_R\bigsqcup_{j=1}^{n_i}X^{(i)}_j$$ and let us fix a positive number $\de$ such that $2\de<R$.
Let $s$ and $s'$ be positive numbers with $2s<R-2\de$ and $s\lq s'$.
 Let us define  
for $i=1,2$ $${\iota'}^{(i),\eps,s,s'}_{X,A,*}:\bigoplus_{j=1}^{n_i}K_*^{\eps,s}(A_{X^{(i)}_{j,\de}})\lto\bigoplus_{j=1}^{n_i}K_*^{\eps,s'}(A_{X^{(i)}_{j,\de}})$$ 

and 

$${\iota'}^{(1,2),\eps,s,s'}_{X,A,*}:\bigoplus_{j=1,\ldots,n_1,l=1,\ldots,n_2}K_*^{\eps,s}(A_{X^{(1)}_{j,\de}\cap X^{(2)}_{l,\de}  })\lto\bigoplus_{j=1,\ldots,n_1,l=1,\ldots,n_2}K_*^{\eps,s'}(A_{X^{(1)}_{j,\de}\cap X^{(2)}_{l,\de}  })$$ as the morphisms  obtained by increasing the control from $s$ to $s'$.
As usual, we should write ${\iota'}^{(i),-,\eps,s'}_{X,A,*}$ (resp. ${\iota'}^{(i),\eps,s,-}$) when the source (resp. the range) is obvious and similarly for ${\iota'}^{(1,2),\eps,s,s'}_{X,A,*}$.
Let  
  $${\jmath'}^{(1,2),\sharp,\eps,s'}_{X,A,*}:\bigoplus_{j=1,\ldots,n_1,l=1,\ldots,n_2}K_*^{\eps,s'}(A_{X^{(1)}_{j,\de}\cap X^{(2)}_{l,\de}})\lto\bigoplus_{j=1}^{n_1}K_*^{\eps,s'}(A_{X^{(1)}_{\de,j}})$$ be the morphism induced by the inclusion
  $$\bigoplus_{j=1,\ldots,n_1,l=1,\ldots,n_2}A_{X^{(1)}_{j,\de}\cap X^{(2)}_{l,\de}}\hookrightarrow \bigoplus_{j=1}^{n_1}A_{X^{(1)}_{j,\de}}$$ and let  $${\jmath'}^{(2,1),\sharp,\eps,s'}_{X,A,*}:\bigoplus_{j=1,\ldots,n_1,l=1,\ldots,n_2}K_*^{\eps,s'}(A_{X^{(1)}_{j,\de}\cap X^{(2)}_{l,\de}})\lto\bigoplus_{l=1}^{n_2}K_*^{\eps,s'}(A_{X^{(2)}_{l,\de}})$$ be the morphism induced by the inclusion
  $$\bigoplus_{j=1,\ldots,n_1,l=1,\ldots,n_2}A_{X^{(1)}_{j,\de}\cap X^{(2)}_{l,\de}}\hookrightarrow \bigoplus_{l=1}^{n_2}A_{X^{(2)}_{l,\de}}$$
 Let us set also for any positive number $d$
\begin{eqnarray*}
&&P'_{d}\left(X^{(i)}_\de\right)=\bigsqcup_{j=1}^{n_i} P_{d}\left(X^{(i)}_{j,\de}\right)\subseteq P_{d}\left(X^{(i)}_\de\right)\text{ for }i=1,2;\\
&&P'_{d}\left(X^{(1)}_\de\cap X^{(2)}_\de\right)=P'_{d}\left(X^{(1)}_\de\right)\cap P'_{d}\left(X^{(2)}_\de\right)=\displaystyle \bigsqcup_{j=1,\ldots, n_1;\,l=1,\ldots n_2} P_{d}\left(X^{(1)}_{j,\de}\cap X^{(2)}_{l,\de}\right);\\
&&P_{d,\de}(X)=P'_{d}\left(X^{(1)}_{\de}\right)\cup P'_{d}\left(X^{(2)}_{\de}\right).\end{eqnarray*}

For $i=1,2$ and $\eps$, $s$ positive numbers with $\eps<1/100$ and $s\gq r_{\eps,d}$,  let us define $${\nu'}^{\eps,s,d}_{X_\de^{(i)},A,*}:K_*(P'_d(X_\de^{(i)}),A)\lto \bigoplus_{j=1}^{n_i}K_*^{\eps,s}(A_{X^{(i)}_{j,\de}})$$ as the composition
\begin{equation*}K_*(P'_d(X_\de^{(i)}),A)\stackrel{\cong}{\lto} \bigoplus_{j=1}^{n_i}K_*(P_d(X_{j,\de}^{(i)}),A)
\stackrel{\underrightarrow{{\oplus_{j=1}^{n_i}\nu^{\eps,s,d}_{X_{j,\de}^{(i)},A,*}}}}{}\bigoplus_{j=1}^{n_i}K_*^{\eps,s}(A_{X^{(i)}_{j,\de}})\end{equation*}where the first  map is induced by the isomorphism $$C(P'_d(X_\de^{(i)}))\stackrel{\cong}{\lto}\bigoplus_{j=1}^{n_i}C(P_d(X_{j,\de^{(i)}})).$$
Define similarly 
 $${\nu'}^{\eps,s,d}_{X_\de^{(1)}\cap X_\de^{(2)},A,*}:K_*(P'_d(X_\de^{(1)}\cap X_\de^{(2)}),A)\lto\bigoplus_{j=1,\ldots,n_1;l=1,\ldots,n_2}K_*^{\eps,s}(A_{X^{(1)}_{j,\de}\cap X^{(2)}_{l,\de}}).$$

\smallskip

Assume that $d<R-2\de$, in which case we have $P^\de_{d}(X)=P_d(X)$,   $P'_{d}(X^{(i)}_\de)=P_d(X^{(i)}_\de)$ for $i=1,2$ and
$P'_d(X^{(1)}_\de\cap X^{(2)}_\de)=P_d(X^{(1)}_\de\cap X^{(2)}_\de)$.
Let $d'$ be a positive number with $d\lq d'$.
For any $C^*$ algebra $A$, let us define 
$${q'}_{d,d',\de,*}:K_*(P_{d}(X),A)\lto K_*(P_{d',\de}(X),A),$$
$${q'}^{(1)}_{d,d',\de,*}:K_*(P_{d}(X^{(1)}_\de),A)\lto K_*(P'_{d'}(X^{(1)}_\de),A),$$
$${q'}^{(2)}_{d,d',\de,*}:K_*(P_{d}(X^{(2)}_\de),A)\lto K_*(P'_{d'}(X^{(2)}_\de),A)$$ and 
$${q'}^{(1,2)}_{d,d',\de,*}:K_*(P_{d}(X^{(1)}_\de\cap X^{(2)}_\de),A)\lto K_*(P'_{d'}(X^{(1)}_\de\cap X^{(2)}_\de),A)$$
as the morphism induced by the inclusion 
$P_{d}(X)\hookrightarrow P_{d',\de}(X)$.

\smallskip

 On the other hand, let $(\H_{d',\de},\rho_{d',\de})$ be a non degenerate $P_{d',\de}(X)$-module, let  $\H_d$ be the set of elements in $\H_{d',\de}$ with support in $P_d(X)$ and let $\rho_d$ be the restriction of
 $\rho_{d',\de}$ to  $\H_d$. 
 For $i=1,2$, let  
 ${\H}_{d',\de}^{(i)}$ be the   set of elements of $\H_{d',\de}$ with support in $P'_{d'}(X^{(i)}_{\de})$.

 Let then
$${q'}_{d,d',\de}^\sharp:\Kp(\H_d)\lto\Kp(\H_{d',\de})$$ be the natural inclusion and let us denote respectively by 
$${q'}_{d,d',\de}^{(1),\sharp}:\Kp({\H}_{d,\de}^{(1)})\lto\Kp({\H}_{d',\de}^{(1)})$$
$${q'}_{d,d',\de}^{(2),\sharp}:\Kp({\H}_{d,\de}^{(2)})\lto\Kp({\H}_{d',\de}^{(2)})$$ and
$${q'}_{d,d',\de}^{(1,2),\sharp}:\Kp({\H}_{d,\de}^{(1)}\cap {\H}_{d,\de}^{(2)} )\lto\Kp({\H}_{d',\de}^{(2)}\cap {\H}_{d',\de}^{(2)} )$$ the restriction of    ${q'}_{d,d',\de}^\sharp$ to 
${\H}_{d,\de}^{(1)}$, ${\H}_{d,\de}^{(2)}$ and ${\H}_{d,\de}^{(1)}\cap{\H}_{d,\de}^{(2)}$.

   Since the inclusion 
 $P_d(X)\hookrightarrow P_{d'}(X)$ is contractive, we see that ${q'}_{d,d',\de}^\sharp$ is a morphism of filtered $C^*$-algebras. 
Notice that if   moreover $(\H_d,\rho_d)$ and $(\H_{d',\de},\rho_{d',\de})$ are  standard modules  and  up to a propagation rescaling, then ${q'}_{d,d',\de}^\sharp$ induces in quantitative $K$-theory the morphism corresponding to the inclusion $P_d(X)\hookrightarrow P_{d',\de}(X)$ constructed in Section \ref{subsec-coarse-geom-nds-mod} (see Remark \ref{remark-covering-isometries}).

As for $P_d(X)$,
  we denote respectively by 
 $$\jmath^{(1),\sharp}_{P_{d',\de}(X),A}:A_{{\H}_{d',\de}^{(1)}}\lto A_{{\H_{d',\de}}},$$ $$\jmath^{(2),\sharp}_{P_{d',\de}(X),A}:A_{{\H}_{d',\de}^{(2)}}\lto A_{{\H_{d',\de}}},$$  $$\jmath^{(1,2),\sharp}_{P_{d',\de}(X),A}:A_{{\H}_{d',\de}^{(1)}\cap {\H}_{d',\de}^{(2)}}\lto  A_{{\H}_{d',\de}^{(1)}}$$ and
  $$\jmath^{(2,1),\sharp}_{P_{d',\de}(X),A}:A_{{\H}_{d',\de}^{(1)}\cap {\H}_{d',\de}^{(2)}}\lto A_{{\H}_{d',\de}^{(2)}}$$ the morphisms of filtered $C^*$-algebras induced by the inclusions ${\H}_{d',\de}^{(1)}\hookrightarrow {\H_{d',\de}}$, ${\H}_{d',\de}^{(2)}\hookrightarrow {\H_{d',\de}}$, ${\H}_{d',\de}^{(1)}\cap {\H'}_{d',\de}^{(2)} \hookrightarrow{\H}_{d',\de}^{(1)}$ and  ${\H}_{d',\de}^{(1)}\cap {\H}_{d',\de}^{(2)} \hookrightarrow{\H}_{d',\de}^{(2)}$.

 In the same way, for $K$-homology, 
  we denote for $i=1,2$  by 
 $$\jmath^{(i)}_{P_{d',\de}(X),*}:K_*(P'_{d'}(X^{(i)}_{\de}),A)\lto K_*(P_{d',\de}(X),A),$$
  the morphisms  induced by the inclusion
   $P'_{d'}(X^{(i)}_{\de})\hookrightarrow P_{d',\de}(X).$

 Eventually let $$V_{\H_{d',\de}}:\H_{d',\de}\lto \ell^2(X,\H_{d',\de})$$  be  the isometry defined by $$(V_{\H_{d',\de}}\cdot\xi)(x)=\rho(\la^{1/2}_x)\cdot \xi$$ for any $\xi$ in $\H_{d',\de}$  and any $x$ in $X$. Let
 $$V_{{\H}_{d',\de}^{(1)}}:{\H}_{d',\de}^{(1)}\lto \ell^2(X^{(1)}_{\de},{\H}_{d',\de}^{(1)}),$$ 
  $$V_{{\H}_{d',\de}^{(2)}}:{\H}_{d',\de}^{(1)}\lto \ell^2(X^{(2)}_{\de},{\H}_{d',\de}^{(2)})$$ and
   $$V_{{\H}_{d',\de}^{(1)}\cap{\H}_{d',\de}^{(2)}}:{\H}_{d',\de}^{(1)}\cap {\H}_{d',\de}^{(2)}\lto \ell^2(X^{(1)}_{\de}\cap X^{(2)}_{\de},{\H}_{d',\de}^{(1)}\cap {\H}_{d',\de}^{(2)})$$
   be respectively the restriction of 
  $V_{\H_{d',\de}}$ to    
 ${\H}_{d',\de}^{(1)}$, ${\H}_{d',\de}^{(2)}$ and
  ${\H}_{d',\de}^{(1)}\cap{\H}_{d',\de}^{(2)}$. 
 
\subsection{The $QI$-statement}\label{subsection-QI-statement}
In this subsection, we set $(\al,k)=(\al_\mv,k_\mv)$.
 \begin{lemma}\label{lemma1-preliminary-QI}There exists a control pair  $(\la,h)$ with $(\la,h)\gq(\al,k)$ such that the following holds :
 
 \medskip

Let $X$ be a finite metric space and let  $X^{(1)}$ and $X^{(2)}$ be two subsets of $X$ such that
  $X=X^{(1)}\cup X^{(2)}$, then 
\begin{itemize}
\item for any $C^*$-algebra $A$;
\item for any positive numbers $\eps$, $s$, $d$, $\de_0$, $\de_1$ and $r$ with  $\de_0>2d$, $\de_1> (2d+1)(5r+1)$,
 $\eps<\frac{1}{100\la}$,  $s\gq r_{d,\al\eps}$ and  $s<\frac{r}{k_\eps}$,
\end{itemize}
we have

$$\nu_{X^{(1)}_{\de_0+\de_1}\cap X^{(2)}_{d,\de_0+\de_1},A,*}^{\al\eps,k_{\eps}s,d}\circ \partial_{P_d(X^{(1)}_{\de_0+\de_1}),P_d(X^{(2)}_{\de_0+\de_1},*)}=\partial^{\eps,s}_{X^{(1)}_{\de_0+\de_1},X^{(2)}_{\de_0+\de_1},A,*}\circ \nu^{\eps,s,d}_{X,A,*}.$$ 
\end{lemma}
\begin{proof}
Let us choose a non degenerate standard $P_d(X)$-module $(\H_d,\rho_d)$ such that the restriction to 
$P_d(X^{(1)}_{\de_0+\de_1}\cap X^{(2)}_{\de_0+\de_1})$ is  also standard. Then $Ad_{V_{\H_d},A}$ clearly maps the controlled $r$-controlled Mayer-Vietoris pair of  equation (\ref{equ-MVpair2})  to  the one of  equation (\ref{equ-MVpair1})  and hence intertwines  up to Morita equivalence the associated controlled boundary maps. The lemma is then a consequence of   Propositions \ref{proposition-assembly-conjugate} and  \ref{proposition-index-commutes}.
\end{proof}

  \begin{lemma}\label{lemma2-preliminary-QI}
There exists a control pair $(\la,h)$ with  $(\la,h)\gq (\al,k)$ such that
\begin{itemize}
\item for any finite metric space $X$;
\item for any $C^*$-algebra $A$;
\item for any subsets $X^{(1)}$ and $X^{(2)}$ of $X$  and any $R>0$ such that
\begin{itemize}
\item  $X=X^{(1)}\cup X^{(2)}$;
\item $X^{(i)}=_R\bigsqcup_{k=1}^{n_i}X^{(i)}_k$ for $i=1,2$.
\end{itemize}
\item for any positive numbers $d$, $d'$, $\de_0$, $\de_1$ and $s$ with $2(\de_0+\de_1)<R$, $d<2s$, $d\lq d'$, $\de_0>2d$ and $\de_1> (2d+1)(5s+1)$, $s<R/2-\de_0-\de_1$,
\end{itemize}

 the following is satisfied:

\medskip

\begin{itemize}
\item let $(\H_{d',\de_0+\de_1},\rho_{d',\de_0+\de_1})$  be a non degenerate  $P_{d',\de_0+\de_1}(X)$-module and let  $(\H_d,\rho_d)$ be its restriction to
the set of elements in $\H_{d',\de_0+\de_1}$ with support in $P_d(X)$;
\item let $\eps_0$, $\eps$ and $s_0$ be positive numbers such that $\eps_0\lq \eps<\frac{1}{100\la}$, $h_{\eps_0}s_0< s$ and $2d(h_\eps s_0+1)<s$;
\item let $y$ be an element in $K_1^{\eps_0,s_0}(A_{\H_d})$ such that 
\begin{itemize} \item $\iota^{-,\eps,s}_*\circ Ad^{\eps_0,s_0}_{V_{\H_d},A,*}(y)=0$ in 
$K_1^{\eps,s}(A_{X})$ (up to Morita equivalence);
\item ${q'}^{(1,2),\sharp,\al\eps_0,k_{\eps_0}s_0}_{d,d',\de_0+\de_1,A,*}\circ \partial^{\eps_0,s_0}_{X,d,\de_0,\de_1,A,*}(y)=0$ in $K_0^{\al\eps_0,k_{\eps_0}s_0}(A_{{\H}_{d',\de_0+\de_1}^{(1 )}\cap {\H}_{d',\de_0+\de_1}^{(2)}}  )$,
\end{itemize}
\end{itemize}
Then there exists
\begin{itemize}
\item $y^{(1)}$ in $K_1^{\la\eps_0,h_{\eps_0}s_0}(A_{{\H}_{d',\de_0+\de_1}^{(1 )}}  )$;
\item $y^{(2)}$ in $K_1^{\la\eps_0,h_{\eps_0}s_0}(A_{{\H}_{d',\de_0+\de_1}^{(2 )}}  )$;
\item $z$ in  $\displaystyle\bigoplus_{j=1,\ldots,n_1;\,l=1,\ldots,n_2}K_1^{\la\eps,2(d'+1)(h_{\eps}s+1)}(A_{X^{(1)}_{j,\de_0+\de_1}\cap X^{(2)}_{l,\de_0+\de_1}})$
\end{itemize}
such that
\begin{enumerate}
\item ${\jmath}_{P_{d',\de_0+\de_1}(X),A,*}^{(1),\sharp,\la\eps_0,h_{\eps_0} s_0}(y^{(1)})-{\jmath}_{P_{d',\de_0+\de_1}(X),A,*}^{(2),\sharp,\la\eps_0,h_{\eps_0} s_0}(y^{(2)})=\iota^{-,\la\eps_0,h_{\eps_0} s_0}_*\circ {q'}_{d,d',\de_0+\de_1,A,*}^{\sharp,\eps_0,s_0}(y)$ in $K_1^{\la\eps_0,h_{\eps_0} s_0}(A_{\H_{d',\de_0+\de_1}})$;
\item ${\iota'}_{X,A,*}^{(1),-,\la\eps,2(d'+1)(h_{\eps}s+1)}\circ Ad^{\la\eps_0,h_{\eps_0} s_0}_{V_{{\H}_{d',\de_0+\de_1}^{(1)}},A,*}(y^{(1)})={\jmath'}^{\sharp,(1,2),\la\eps,2(d'+1)(h_{\eps}s+1)}_{X,A,*}(z)$ in $K^{\la\eps',h_{\eps'}r}_1(A_{X^{(1)}_{\de_0+\de_1}})$;
\item ${\iota'}_{X,A,*}^{(2),-,\la\eps,2(d'+1)(h_{\eps}s+1)}\circ Ad^{\la\eps_0,h_{\eps_0} s_0}_{V_{{\H}_{d',\de_0+\de_1}^{(2)}},A,*}(y^{(2)})={\jmath'}^{\sharp,(2,1),\la\eps,2(d'+1)(h_{\eps}s+1)}_{X,,A,*}(z)$ in $K^{\la\eps,2(d'+1)(h_{\eps}s+1)}_1(A_{X^{(2)}_{\de_0+\de_1}})$
where the equalities in points (ii) and (iii) are indeed up to Morita equivalence.
\end{enumerate}\end{lemma}
\begin{proof}
Let $u$ be a $\eps_0$-$s_0$-unitary in some $\M_n(\widetilde{A_{\H_d}})$ such that
$y=[u]_{\eps_0,s_0}$ in $K_1^{\eps_0,s_0}(A_{\H_d})$. In view of \cite[Lemma 1.8]{oy4}, we can assume without loss of generality that
$u-I_n\ts Id_{\H_d}$ is in $M_n(A_{\H_d})$.
As in the construction of the controlled Mayer-Vietoris boundary  \cite[Section 3.2]{oy4} and possibly taking a larger $n$, let us pick  two  $\al\eps_0$-$k_{\eps_0}s_0$-unitaries  $v_1$ and $v_2$ in $\widetilde{M_{2n}(A_{\H_d})}$ and  an $\al\eps_0$-$k_{\eps_0}s_0$-projection $p$ in $\widetilde{M_{2n}(A_{\H_d})}$ such that
\begin{itemize}
\item $v_i-I_{2n}\ts Id_{\H_d}$ lies in $M_{2n}(A_{\H_{d,\de_0+\de_1}^{(i)}})$ for $i=1,2$;
\item $p-\diag(I_n\ts Id_{\H_d},0)$ lies in $M_{2n}(A_{\H_{d,\de_0+\de_1}^{(1)}\cap\H_{d,\de_0+\de_1}^{(2)}     })$
\item $\|\diag(u,u^*)-v_1v_2\|<\al \eps_0$;
\item $\|p-v_1^*\diag(I_n\ts Id_{\H_d},0)v_1\|<\al \eps_0$;
\item $\|p-v_2\diag(I_n\ts Id_{\H_d},0)v_2^*\|<\al \eps_0$;
\item  $[p,n]_{\al \eps_0,k_{\eps_0}s_0}= \partial^{\eps_0,s_0}_{X,d,\de_0,\de_1,A,*}(y)$ in
 $K_0^{\al \eps_0,k_{\eps_0}s_0}(A_{\H_d^{X^{(1)},\de_0+\de_1}\cap \H_d^{X^{(2)},\de_0+\de_1}     })$.
 \end{itemize}
 Let us set   \begin{eqnarray*}
 u'&=&u-I_n\ts Id_{\H_d}+I_n\ts Id_{\H_{d',\de_0+\de_1}}\\
 v_1'&=&v_1- I_{2n}\ts Id_{\H_d}+I_{2n}\ts Id_{\H_{d',\de_0+\de_1}}\\
 v_2'&=&v_2-I_{2n}\ts Id_{\H_d}+I_{2n}\ts Id_{\H_{d',\de_0+\de_1}}\\ 
 p'&=&p-\diag(I_n\ts Id_{\H_d},0)+\diag(I_n\ts Id_{\H_{d',\de_0+\de_1}},0),\end{eqnarray*}(recall that $Id_{\H_d}$ is viewed as the projection on $\H_d\subseteq \H_{d',\de_0+\de_1}$).
 Apply then \cite[Lemma 3.5]{oy4} to
 $\widetilde{A_{\H_{d',\de_0+\de_1}}}$, $\widetilde{A_{{\H'}_{d',\de_0+\de_1}^{(1)}}}$,   $\widetilde{A_{{\H}_{d',\de_0+\de_1}^{(2)}}}$, $u'$, ${u'}^*$,  $v'_1$,  $v_2'$ and $p'$, we see that for some control pair $(\la,h)$ depending only on $(\al,k)=(\al_\mv,k_\mv)$ and such that     $(\la,h)\gq (\al,k)$ and possibly taking a larger $n$, there exist two $\la\eps_0$-$h_{\eps_0} s_0$-unitaries $w_1$ and $w_2$ in  
  $\widetilde{M_{n}(A_{\H_{d',\de_0+\de_1}})}$ such that
  \begin{itemize}
  \item $w_i-I_{2n}\ts Id_{\H_{d',\de_0+\de_1}}$ lies in ${A_{{\H_{d',\de_0+\de_1}}^{(i)}}}$ for $i=1,2$; 
  \item $\|\diag(u',I_n\ts Id_{\H_{d',\de_0+\de_1}})-w_1w_2\|<\la\eps_0$.
  \end{itemize}
  Let us set then
 \begin{itemize}
 \item $y^{(1)}=[w_1]_{\la\eps_0,h_{\eps_0} s_0}$ in $K_1^{\la\eps_0,h_{\eps_0} s_0}(A_{{\H_{d',\de_0+\de_1}}^{(1)}})$;
 \item  $y^{(2)}=[w_2^*]_{\la\eps_0,h_{\eps_0} s_0}$ in $K_1^{\la\eps_0,h_{\eps_0} s_0}(A_{{\H_{d',\de_0+\de_1}}^{(2)}})$. 
 \end{itemize}
Point (i) is then a consequence of the equality
$$[u']_{\eps_0,s_0}={q'}_{d,d',\de_0+\de_1,*}^{\eps_0,s_0}([u]_{\eps_0,s_0}).$$

  Notice that
  $$[Ad_{V_{\H_{d',\de_0+\de_1},M_n(A)}}(u')+I_n\ts Id_{\ell^2(X,\H_{d',\de_0+\de_1})}-I_n\ts V_{\H_{d',\de_0+\de_1}}V^*_{\H_{d',\de_0+\de_1}}]_{\eps_0,2d(s_0+1)}$$ is the image of
    $$[Ad_{V_{\H_{d},M_n(A)}}(u)+I_n\ts Id_{\ell^2(X,\H_d)}-I_n\ts V_{\H_d}V^*_{\H_d}]_{\eps_0,2d(s_0+1)}$$  under the morphism
    $$K_1^{\eps_0,2d(s_0+1)}(A_{\ell^2(X,\H_d)})\lto K_1^{\eps,2d(s_0+1)}(A_{\ell^2(X,\H_{d',\de_0+\de_1})}) $$ induced by the inclusion
    $A_{\ell^2(X,\H_d)})\hookrightarrow A_{\ell^2(X,\H_{d',\de_0+\de_1})}$ and hence 
     $$[Ad_{V_{\H_{d',\de_0+\de_1},M_n(A)}}(u')+I_n\ts Id_{\ell^2(X,\H_{d',\de_0+\de_1})}-I_n\ts V_{\H_{d',\de_0+\de_1}}V^*_{\H_{d',\de_0+\de_1}}]_{\eps,s}=0$$ in
     $K_1^{\eps,s}(A_{\ell^2(X,\H_{d',\de_0+\de_1})})$.

  Hence, in view of \cite[Proposition 2.12]{oy4}, and up to rescaling $(\la,h)$ and to enlarging  $n$, we see that there exists  two $\la\eps$-$h_{\eps}s$-unitaries $V_1$ and $V_2$ in $M_{n}(\widetilde{A_{\ell^2(X,\H_{d',\de_0+\de_1})}})$ such that
 \begin{itemize}
 \item       $V_i-I_{2n}\ts Id_{\ell^2(X,\H_{d',\de_0+\de_1})}$ is in        $$\bigoplus_{j=1}^{n_i}M_{n}(A_{\ell^2(X^{(i)}_{j,\de_0+\de_1},\H_{d',\de_0+\de_1})})$$ for $i=1,2$;
 \item $V_i$ is homotopic to $I_n\ts Id_{\ell^2(X,\H_{d',\de_0+\de_1})}$ as 
 a $\la\eps$-$h_{\eps}s$-unitary in $$\bigoplus_{j=1}^{n_i}M_{n}(A_{\ell^2(X^{(i)}_{j,\de_0+\de_1},\H_{d',\de_0+\de_1})})+\C\cdot I_n\ts Id_{\ell^2(X,\H_{d',\de_0+\de_1})}; $$
 \item $\|Ad_{V_{\H_{d',\de_0+\de_1},M_{n}(A)}}(u')+I_{n}\ts Id_{\ell^2(X,\H_{d',\de_0+\de_1})}-I_{n}\ts V_{\H_{d',\de_0+\de_1}}V^*_{\H_{d',\de_0+\de_1}}-V_1V_2\|<\la\eps$.
 \end{itemize}
 Let us set for $i=1,2$
 $$W_i=Ad_{V_{\H_{d',\de_0+\de_1},M_{n}(A)}}(w_i)+I_{n}\ts Id_{\ell^2(X,\H_{d',\de_0+\de_1})}-I_{n}\ts V_{\H_{d',\de_0+\de_1}}V^*_{\H_{d',\de_0+\de_1}}.$$ For $i=1,2$, then $W_i$ is a $\la\eps_0$-$2d'(h_{\eps_0}s_0+1)$-unitary in 
 $$M_n(A_{\ell^2(X,\H_{d',\de_0+\de_1})})+\C\cdot I_n\ts Id_{\ell^2(X,\H_{d',\de_0+\de_1})}$$ and moreover, since   the propagation of  $w_i$ is  less than $h_{\eps_0}s_0$ then
 $W_i-I_n\ts Id_{\ell^2(X,\H_{d',\de_0+\de_1})}$ lies in $$\bigoplus_{j=1}^{n_i}M_{n}(A_{\ell^2(X^{(i)}_{j,\de_0+\de_1},\H_{d',\de_0+\de_1})}).$$
 Furthermore,  we have $$\|V_1V_2-W_1W_2\|<2\la\eps$$ and hence
 $$\|V_2W_2^*-V_1^*W_1\|<12\la\eps.$$ 
 Using the $CIA$-property for the pair of algebras 
 $$\left(\bigoplus_{j=1}^{n_1} A_{\ell^2(X^{(1)}_{j,\de_0+\de_1},\H_{d',\de_0+\de_1})   },\bigoplus_{j=1}^{n_1} A_{\ell^2(X^{(2)}_{j,\de_0+\de_1},\H_{d',\de_0+\de_1})   }\right),$$  we  deduce that there exists a $12\la\eps$-$2(d'+1)(h_{\eps}s+1)$-unitary $U$ in the unitarization of
 $$\bigoplus_{j=1,\ldots,n_1;\,l=1,\ldots n_2}M_{n}(A_{\ell^2(X^{(1)}_{j,\de_0+\de_1}\cap X^{(2)}_{l,\de_0+\de_1},\H_{d',\de_0+\de_1})})$$  such that
 $$\|U-V_1^*W_1\|<12\la\eps$$ and 
 $$\|U-V_2W_2^*\|<12\la\eps.$$

 In particular, in view of point (ii) of Lemma \ref{lemma-almost-closed}, we deduce that
 \begin{itemize}
 \item $U$ and $W_1$ are homotopic as $48\la\eps$-$2(d'+1)(h_{\eps }s+1)$-unitaries in  $$\bigoplus_{k=1}^{n_i}M_{n}(A_{\ell^2(X^{(1)}_{k,\de_0+\de_1},\H_{d',\de_0+\de_1})})+\C\cdot I_n\ts Id_{\ell^2(X,\H_{d',\de_0+\de_1})};$$
 \item $U$ and $W^*_2$ are homotopic as $48\la\eps$-$2(d'+1)(h_{\eps }s+1)$-unitaries in $$\bigoplus_{k=1}^{n_i}M_{n}(A_{\ell^2(X^{(2)}_{k,\de_0+\de_1},\H_{d',\de_0+\de_1})})+\C\cdot I_n\ts Id_{\ell^2(X,\H_{d',\de_0+\de_1})};$$
 \end{itemize}
 Let us replace $48\la$ by $\la$   and let then set
$z=[U]_{\la\eps,2(d'+1)(h_{\eps}s+1)}$  (up to Morita equivalence) in          $K_1^{\la\eps,2(d'+1)(h_\eps s+1)}(A_{\ell^2(X^{(1)}_{\de_0+\de_1}\cap X^{(2)}_{\de_0+\de_1})})$. 
We have 
\begin{eqnarray*}
{\jmath'}^{\sharp,(1,2),\la\eps,2(d'+1)(h_{\eps}s+1)}_{X,A,*}(z)&=&[{\jmath'}^{\sharp,(1,2)}_{X,M_{n(A)}}(U)]_{\la\eps,2(d'+1)(h_{\eps}s+1)}\\
&=&[W_1]_{\la\eps,2(d'+1)(h_{\eps}s+1)}\\
&=&{\iota'}_*^{(1)-,\la\eps,2(d'+1)(h_{\eps}s+1)}\circ Ad^{\la\eps_0,h_{\eps_0} s_0}_{V_{{\H}_{d',\de_0+\de_1}^{(1)}},A,*}(y^{(1)})\end{eqnarray*}
and similarly
$${\jmath'}^{\sharp,(2,1),\la\eps,h_{\eps}s}_{X,A,*}(z)={\iota'}_*^{(2),-,\la\eps,2(d'+1)(h_{\eps}s+1)}\circ Ad^{\la\eps_0,h_{\eps_0} s_0}_{V_{{\H}_{d',\de_0+\de_1}^{(2)}},A,*}(y^{(2)}).$$ and hence we deduce points (ii) and (iii).

 \end{proof}
 {\it  End of the proof of the $QI$ statement.} 
 Let $\X$ be a family of finite metric spaces satisfying the assumptions of Theorem \ref{theorem-main}.
 Let $\eps$, $s$ and $d$ be positive numbers with $\eps<\frac{1}{200}$
 and $s\gq r_{d,\eps}$.  Let us show that there exists a positive number $d'$ such that for any $X$ in $\X$ and any $C^*$-algebra $A$, then $QI_{X,A,*}(d,d',s,\eps)$ holds. Using  controlled Bott Periodicity, we can reduce to the odd case. As before, we set  $(\al,k)=(\al_\mv,k_\mv)$ and for sake of simplicity, we shall also denote by 
   $(\al,k)$  the  control pairs of Lemmas  \ref{lemma1-preliminary-QI} and  \ref{lemma2-preliminary-QI}.

 Let us fix four positive numbers $r$, $\de_0$, $\de_1$ and $R$  such that  $k_{\eps}s<r$, $\de_0>2d$, $\de_1>(2d+1)(5r+1)$ and $R>2(\de_0+\de_1)+2r$. Let  $\Y$ be a family of finite metric spaces with uniformly bounded  geometry such that
 \begin{itemize}
 \item $\X$ is $R$-decomposable with respect to $\Y$;
 \item $\widetilde{\Y}$ satisfies uniformly the QAM-estimates.
 \end{itemize}

 For any  positive number $\delta$, let $\widetilde{\X_\de}$ be the set of all subsets of metric spaces in $\X$ thats contains a metric space in $\Y$ as a $\de$-net.
 According to Lemma \ref{lemma-s-net}, the family $\widetilde{\X_\de}$ satisfies uniformly the QAM-estimates.

 \medskip

 \medskip
 
{\bf Step 1: } Assume that $\eps<\frac{1}{200\al}$ (the general case will be obtained by rescaling using Proposition \ref{prop-compatibility-assembly-rescaling}).
Let $X$ be a metric space in $\X$ and let $X^{(1)}$ and  $X^{(2)}$ be two subsets of $X$ such that
\begin{itemize}
\item $X=X^{(1)}\cup X^{(2)}$;
\item $X^{(i)}=_R\bigsqcup_{j=1}^{n_i} X_j^{(i)}$ with $X_j^{(i)}$ in $\Y$ for $j=1,\ldots,n_i$ and $i=1,2$.
\end{itemize}
Recall from equation (\ref{eq-MV-ripps}) that
$$P_d(X)=\widering{{P_d(X_{\de_0+\de_1}^{(1)})}}\cup \widering{P_d(X_{\de_0+\de_1}^{(2)})}.$$
Let   $d'$ be a positive number with $d'\gq d$ such that
$QI_{X',A,*}(d,d',k_\eps s,\eps)$ holds for all metric spaces in $\widetilde{\X_{\de_0+\de_1}}$.
Notice that

$$K_*(P_d(X^{(1)}_{\de_0+\de_1}\cap X^{(2)}_{\de_0+\de_1}),A)\cong \bigoplus_{j=1,\ldots,n_1,l=1,\ldots n_2}K_*(P_d(X^{(1)}_{j,\de_0+\de_1}\cap X^{(2)}_{l,\de_0+\de_1}),A)$$ and 
$$K_*^{\al\eps,k_{\eps}s}(A_{X^{(1)}_{\de_0+\de_1}\cap X^{(2)}_{\de_0+\de_1}})\cong \bigoplus_{j=1,\ldots,n_1,l=1,\ldots n_2}K_*^{\al\eps,k_{\eps}s}(A_{X^{(1)}_{j,\de_0+\de_1}\cap X^{(2)}_{l,\de_0+\de_1}})$$ and that under these two identifications, then
\begin{equation}\label{equ-decom-assembly-map}\nu_{X^{(1)}_{\de_0+\de_1}\cap X^{(2)}_{\de_0+\de_1},A,*}^{\al\eps,k_{\eps}s,d}=\bigoplus_{j=1,\ldots,n_1,l=1,\ldots n_2} \nu_{X^{(1)}_{j,\de_0+\de_1}\cap X^{(2)}_{l,\de_0+\de_1},A,*}^{\al\eps,k_{\eps}s,d}.\end{equation}
Let $x$ be a element of $K_1(P_d(X),A)$  such that
$\nu_{X,A,*}^{\eps,s,d}(x)=0$ in $K_1^{\eps,s}(A_{X})$. 

Since $QI_{X^{(1)}_{\de_0+\de_1}\cap X^{(2)}_{\de_0+\de_1,A,*}}(d,d',k_\eps s,\al\eps)$ holds, we deduce from  Lemma \ref{lemma1-preliminary-QI} and from equation (\ref {equ-decom-assembly-map}) that
\begin{equation}\label{eq-step1}{q'}^{(1,2)}_{d,d',\de_0+\de_1,*}\circ\partial_{P_d(X^{(1)}_{\de_0+\de_1}),P_d(X^{(2)}_{\de_0+\de_1}),*}(x)=0.\end{equation}

\medskip

{\bf Step 2: }
Let $n_{d'}$ be an integer such that for any metric space in $\Y$, balls of radius $d'$ have cardinal bounded by $n_{d'}$.
In view of Theorem  \ref{thm-qindex-iso}, let $(0,1/200)\to [1,+\infty);\,\eps\mapsto h_\eps$ be a non increasing function such that for any finite simplicial complexe $Z$ with dimension less than $n_{d'}$ and for any $C^*$-algebra $A$, then $Ind_{Z,A,*}$ is a $h$-isomorphism.
Let us fix a non degenerate standard module $P_{d',\de_0+\de_1}(X)$-module $(\H_{d',\de_0+\de_1},\rho_{d',\de_0+\de_1})$ such that the restrictions to elements of  $\H_{d',\de_0+\de_1}$ with support in
$P_d(X)$, $P'_{d'}(X^{(1)}_{\de_0+\de_1})$, $P'_{d'}(X^{(2)}_{\de_0+\de_1})$ and $P'_{d'}\left(X^{(1)}_{\de_0+\de_1}\cap X^{(2)}_{\de_0+\de_1}\right)$ are standard. Let $(\H_d,\rho_d)$ be the restriction of $(\H_{d',\de_0+\de_1},\rho_{d',\de_0+\de_1})$ to element of $\H_{d',\de_0+\de_1}$ with support in
$P_d(X)$. Let us pick two positive numbers $\eps_0$ and $s_0$ with 
$\eps_0<\eps $, $s_0<\frac{1}{k_{\eps_0}h^2_{\al\eps_0}}$, $k_{\eps_0} s_0<s$ and $2d(k_{\eps_0} s_0+1)<s$ (recall that $s\gq r_{d,\eps}>2d$).
Let us set then $$y=Ind_{P_d(X),A,*}^{\eps_0,s_0}(x)$$ in $K_1^{\eps_0,s_0}(A_{\H_d})$. According to Proposition \ref{proposition-assembly-conjugate}, we see that
$$\iota_*^{-,\eps, s}\circ Ad^{\eps_0,s_0}_{V_{\H_d},A,*}(y)=0.$$
On the other hand,
\begin{equation*}\begin{split}
{q'}^{(1,2),\sharp,\al\eps_0,k_{\eps_0}s_0}_{d,d',\de_0+\de_1,A,*}\circ& \partial^{\eps_0,s_0}_{X,d,\de_0,\de_1,A,*}(y)\\
=&{q'}^{(1,2),\sharp,\al\eps_0,k_{\eps_0}}_{d,d',\de_0+\de_1,A,*}\circ \partial^{\eps_0,s_0}_{X,d,\de_0,\de_1,A,*}\circ Ind_{X,A,*}^{\eps_0,s_0}(x)\\
=& {q'}^{(1,2),\sharp,\al\eps_0,k_{\eps_0}s_0}_{d,d',\de_0+\de_1,A,*}\circ Ind^{\al\eps_0,k_{\eps_0}s_0}_{P_d(X^{(1)}_{\de_0+\de_1}\cap X^{(2)}_{\de_0+\de_1}),A,*}\circ\partial_{P_d(X^{(1)}_{\de_0+\de_1}),P_d(X^{(2)}_{\de_0+\de_1}),*}(x)\\
=&Ind^{\al\eps_0,k_{\eps_0}s_0}_{P_{d',\de_0+\de_1}X^{(1)}_{\de_0+\de_1}\cap X^{(2)}_{\de_0+\de_1},A,*}\circ {q'}^{(1,2)}_{d,d',*}\circ \partial_{P_d(X^{(1)}_{\de_0+\de_1}),P_d(X^{(2)}_{\de_0+\de_1}),*}(x)\\
=&0,\end{split}
\end{equation*}
where 
\begin{itemize}
\item the second equality is a consequence of  Proposition \ref{proposition-index-commutes}; 
\item the third  equality is a consequence of Remark \ref{rem-funct-inclusion};
\item  the last equality is a consequence of equation (\ref{eq-step1}).
\end{itemize}

\medskip

{\bf Step 3: }
Let  $y^{(1)}$ in $K_1^{\al\eps_0,k_{\eps_0} s_0}(A_{{\H}_{d',\de_0+\de_1}^{(1)}})$, $y^{(2)}$ in  $K_1^{\al\eps_0,k_{\eps_0} s_0}(A_{{\H}_{d',\de_0+\de_1}^{(2)}})$ and 
 $z$ in  $\displaystyle\bigoplus_{j=1,\ldots;n_1;\,l=1,\ldots,n_2} K_1^{\al\eps,2(d'+1)(k_{\eps}s+1)}(A_{X^{(1)}_{j,\de_0+\de_1}\cap X^{(2)}_{l,\de_0+\de_1}})$ be as in   Lemma \ref{lemma2-preliminary-QI} with respect to $y$. 
  Pick then 
 positive numbers $d''$ and $s'$ with  $s'\gq r_{d'',\al\eps}$ and $s'\gq 2(d'+1)(k_{\eps} s+1)$ and such that $$QS_{X',A,*}(d'',2(d'+1)(k_{\eps} s+1),s',\al\eps)$$ holds for any  metric space $X'$ in 
 $\widetilde{\X_{\de_0+\de_1}}$. 

 Then  there exists an element $x_z$ in $$K_1(P'_{d''}(X^{(1)}_{\de_0+\de_1}\cap X^{(2)}_{\de_0+\de_1}),A)\cong \bigoplus_ {j=1,\ldots,n_1,l=1,\ldots,n_2}K_1(P_{d''}(X^{(1)}_{j,\de_0+\de_1}\cap X^{(2)}_{k,\de_0+\de_1}),A)$$ such that under above identifications 
\begin{equation}\label{equ-z-xz}{\nu'}^{\al\eps,s',d''}_{X^{(1)}_{\de_0+\de_1}\cap X^{(2)}_{\de_0+\de_1},A,*}(x_z)=z',\end{equation} where $z'$ is the image of $z$ under the map
\begin{equation*}\begin{split}\bigoplus_{j=1,\ldots;n_1;\,l=1,\ldots,n_2} K_1^{\al\eps,2(d'+1)(k_{\eps}s+1)}&(A_{X^{(1)}_{j,\de_0+\de_1}\cap X^{(2)}_{l,\de_0+\de_1}})\\&  \lto \displaystyle\bigoplus_{j=1,\ldots;n_1;\,l=1,\ldots,n_2} K_1^{\al\eps,s'}(A_{X^{(1)}_{j,\de_0+\de_1}\cap X^{(2)}_{l,\de_0+\de_1}})\end{split}\end{equation*} (increasing control).

Since $k_{\eps_0}s_0<1/h^2_{\al\eps_0}$ and $Ind_{P'_{d'}(X^{(i)}_{\de_0+\de_1}),A,*}$ is a $h$-isomorphism, 
there exists  an element $x^{(i)}$ in $K_1(P'_{d'}(X^{(i)}_{\de_0+\de_1}),A)$ for $i=1,2$ such that 
\begin{equation}\label{eq-ind-yi}
Ind_{P'_{d'}(X^{(i)}_{\de_0+\de_1}),A,*}^{\al\eps_0,h_{\al\eps_0}k_{\eps_0}s_0}(x^{(i)})=\iota_*^{-,{\al\eps_0,h_{\al\eps_0}k_{\eps_0}s_0}}(y^{(i)}).\end{equation}
 For $i=1,2$, let ${x'}^{(i)}$ in $K_1(P'_{d''}(X^{(i)}_{\de_0+\de_1}),A)$ be the image of $x^{(i)}$  under the morphism induced by the inclusion
 $P'_{d'}(X^{(i)}_{\de_0+\de_1})\hookrightarrow P'_{d''}(X^{(i)}_{\de_0+\de_1})$.
 Then 
 \begin{eqnarray*}
 {\nu'}^{\al\eps,s',d''}_{X^{(i)}_{\de_0+\de_1},A,*}(x'^{(i)})&=&\iota'^{(i),-,\al\eps,s'}_{X,A,*}\circ Ad_{V_{\H^{(i)}_{d',\de_0+\de_1}},A,*}^{\al\eps_0,k_{\eps_0}s_0}(y^{(i)})\\
 &=&\jmath^{(1,2),\sharp,\al\eps,s'}_{P_{d'',\de_0+\de_1}(X),A,*}\circ \iota_*^{-,-}(z)\\
 &=&\jmath^{(1,2),\sharp,\al\eps,s'}_{P_{d'',\de_0+\de_1}(X),A,*}\circ{\nu'}^{\al\eps,s',d''}_{X^{(1)}_{\de_0+\de_1}\cap X^{(2)}_{\de_0+\de_1},A,*}(x_z). \end{eqnarray*}
\medskip 

{\bf Step 4: }
Let $d'''$ be a positive number with $d'''\gq d''$ and such that $QI_{X',A,*}(d'',d''',s',\al\eps)$ holds  for every metric space  $X'$ in 
 $\widetilde{\X_{\de_0+\de_1}}$. Since $$P'_{d''}(X^{(i)}_{\de_0+\de_1})= \bigsqcup_ {j=1,\ldots,n_i} P_{d''}(X^{(i)}_{j,\de_0+\de_1})$$ and according to Lemma \ref{lemma-subset} and to equation (\ref{equ-z-xz}), we deduce that in $K_1(P'_{d'''}(X^{(i)}_{\de_0+\de_1}),A)$, the image of ${x'}^{(i)}$    under the morphism induced by the inclusion
 $P'_{d''}(X^{(i)}_{\de_0+\de_1})\hookrightarrow P'_{d'''}(X^{(i)}_{\de_0+\de_1})$  and  the image of $x_z$  under the morphism induced by the inclusion $P'_{d''}(X^{(1)}_{\de_0+\de_1}\cap X^{(2)}_{\de_0+\de_1})\hookrightarrow P'_{d'''}(X^{(i)}_{\de_0+\de_1})$ coincides for $i=1,2$.
 As a consequence of equation \eqref{eq-ind-yi} and of the injectivity of $Ind_{P'_{d'}(X^{(i)}_{\de_0+\de_1}),A,*}^{\al\eps_0,h_{\al\eps_0}k_{\eps_0}s_0}$  and since $${\jmath}_{P_{d',\de_0+\de_1}(X),A,*}^{(1),\sharp,\al\eps_0,k_{\eps_0} s_0}(y^{(1)})-{\jmath}_{P_{d',\de_0+\de_1}(X),A,*}^{(2),\sharp,\al\eps_0,k_{\eps_0} s_0}(y^{(2)})=\iota^{-,\al\eps_0,k_{\eps_0}s_0}_*\circ {q'}_{d,d',\de_0+\de_1,A,*}^{\sharp,\eps_0,s_0}(y)$$ in $K_1^{\al\eps_0,k_{\eps_0} s_0}(A_{\H_{d',\de_0+\de_1}})$, we deduce from   Remark \ref{rem-funct-inclusion} that 
 $${\jmath}_{P^{(1)}_{d',\de_0+\de_1}(X),A,*}({x}^{(1)})-{\jmath}^{(2)}_{P_{d',\de_0+\de_1}(X),A,*}({x}^{(2)})={q'}_{d,d',\de_0+\de_1,A,*}(x)$$  and hence
$${\jmath}_{P^{(1)}_{d'',\de_0+\de_1}(X),A,*}({x'}^{(1)})-{\jmath}^{(2)}_{P_{d'',\de_0+\de_1}(X),A,*}({x'}^{(2)})={q'}_{d,d'',\de_0+\de_1,A,*}(x)$$ 
  and therefore
  $q_{d,d''',*}(x)=0$. Hence the condition $QI_{X,A,*}(d,d''',s,\eps)$ holds.
  
\subsection{The $QS$-statement}
Let $\X$ be a family of finite metric spaces and  let $\Y$ be a family of finite metric spaces with uniformly bounded  geometry such that for any positive number $R$,
 \begin{itemize}
 \item $\X$ is $R$-decomposable with respect to $\Y$;
 \item $\tilde{\Y}$ satisfies uniformly the QAM-estimates.
 \end{itemize}
  As for the $QI$-statement,  by using controlled Bott Periodicity we can reduce it  to the odd case and we set $(\al,k)=(\al_\mv,k_\mv)$.
 
 \medskip

 Let $\eps$ and $s$ be positive numbers with  $\eps<\frac{1}{200\al}$ (the general case will follow by rescaling).
 Let $X$ be a metric space in $\X$  and let $z$ be an element in $K_1^{\eps,s}(A_{X})$. Let us fix $\de$  and $R$ two  positive numbers such that
$10k_{\eps}s<10r<2\de<R$.
Let $X=X^{(1)}\cup X^{(2)}$ be a $R$-decomposition of $X$ in $\Y$ with $X^{(i)}=_R\bigsqcup_{j=1}^{n_i}X_j^{(i)}$ for $i=1,2$.
   Then $X_\de^{(i)}=_{R-2\de}\bigsqcup_{j=1}^{n_i}X_{j,\de}^{(i)}$  for $i=1,2$ and as mentioned in Section \ref{subsection-notations} (see equation (\ref{equ-MVpair1})), for any $C^*$-algebra $A$,
   $$(A_{X^{(1)},X_\de^{(1)},r},A_{X^{(2)},X_\de^{(2)},r},A_{X_\de^{(1)}},A_{X_\de^{(2)}})$$ is a $r$-controlled Mayer-Vietoris pair for $A_{X}$.
    Recall from  Lemma \ref{lemma-s-net} that  the family $\widetilde{\X_\de}$ satisfies uniformly the QAM-estimates.     
   \medskip
   
   {\bf Step 1: }Let $s'$ and $d$ be positive numbers with $s'\gq r_{d,\al\eps}$, $s'\gq k_{\eps}s$ and such that
   the condition $QS_{X',A,*}(k_{\eps}s,s',d,\al\eps)$ holds for every metric space $X'$ in $\widetilde{\X_\de}$.
Since $$K^{\eps,s}_*(A_{X_\de^{(1)}\cap X_\de^{(2)}})\cong\bigoplus_{k=1,\ldots,n_1; \,l=1,\ldots,n_2}
K^{\eps,s}_*(A_{X_{k,\de}^{(1)}\cap X_{l,\de}^{(2)}})$$ and under the identification 
$$K_0(P'_d(X_{\de}^{(1)}\cap X_{\de}^{(2)}),A)\cong \bigoplus_{j=1,\ldots,n_1;\, l=1,\ldots,n_2} K_0(P_d(X_{j,\de}^{(1)}\cap X_{l,\de}^{(2)}),A),$$ there exists an element $x'$ in $K_0(P'_d(X_{\de}^{(1)}\cap X_{\de}^{(2)}),A)$ such that

$$\nu'^{\al\eps,s',d}_{X_{\de}^{(1)}\cap X_{\de}^{(2)},A,*}(x')={\iota'}_*^{(1,2),-,\al\eps,s'}\circ 
\partial_{X_\de^{(1)},X_\de^{(2)},A,*}^{\eps,s}(z).$$
Since $$\jmath^{(1,2),\sharp}_{X,A,*}\circ \D _{X_\de^{(1)},X_\de^{(2)},A,*}=0$$ and 
$$\jmath^{(2,1),\sharp}_{X,A,*}\circ \D _{X_\de^{(1)},X_\de^{(2)},A,*}=0,$$ using Lemma \ref{lemma-subset} we deduce 
that 
\begin{equation}\label{equ1-QS}{\nu'}^{\al\eps,s',d}_{X_{\de}^{(1)},A,*}\circ \jmath^{(1,2)}_{P_{d,\de}(X),*}(x')=0\end{equation}
and 
\begin{equation}\label{equ2-QS}{\nu'}^{\al\eps,s',d}_{X_{\de}^{(2)},A,*}\circ \jmath^{(2,1)}_{P_{d,\de}(X),*}(x')=0,\end{equation}
where
 $$\jmath^{(1,2)}_{P_{d,\de}(X),*}:K_*(P'_d(X_\de^{(1)}\cap  X_\de^{(2)}),A)\lto K_*(P'_d(X_\de^{(1)}),A)$$ and 
 $$\jmath^{(2,1)}_{P_{d,\de}(X),*}:K_*(P'_d(X_\de^{(1)} \cap  X_\de^{(2)}),A)\lto K_*(P'_d(X_\de^{(2)}),A)$$ are respectively the morphism induced by the inclusions 
$$P'_d(X_\de^{(1)}\cap X_\de^{(2)})\hookrightarrow  P'_d(X_\de^{(1)})$$  and 
$$P'_d(X_\de^{(1)}\cap X_\de^{(2)})\hookrightarrow  P'_d(X_\de^{(2)}).$$

  \medskip
   
   {\bf Step 2: }Let $d'$ be a positive number with $d'\gq d$  and such that $QI_{X',A,*}(d,d',s',\al\eps)$ holds for every metric space $X'$ in $\widetilde{\X_\de}$.
   We deduce from equation (\ref{equ1-QS}) that 
   \begin{eqnarray*}
  \jmath^{(1,2)}_{P_{d',\de}(X),*}\circ {q'}^{(1,2)}_{d,d',\de,*}(x')&=&{q'}^{(1)}_{d,d',\de,*}\circ  \jmath^{(1,2)}_{P_{d,\de}(X),*}(x')\\
   &=&0,
   \end{eqnarray*}
   and similarly from   equation (\ref{equ2-QS}) that 
$$\jmath^{(2,1)}_{P_{d',\de}(X),*}\circ {q'}^{(2,1)}_{d,d',\de,*}(x')=0.$$   
Let $(\la,h)$ be the control pair of \cite[Lemma 3.8]{oy4}. For sake  of simplicity, we also denote by $h$  the non increasing function of  Step 2 of the proof  
of the $QI$-statement in Section \ref{subsection-QI-statement}.
Let us pick $\eps_0$ and $s_0$ two positive numbers such that 
$\eps_0<\frac{1}{200\la}$ and $s_0<\frac{1}{h^2_{\la\eps_0}h_{\eps_0}}$. Let us fix $(\H_{d',\de},\rho_{d',\de})$ a non degenerate standard
  $P_{d',\de}(X)$-module  with standard restriction to  ${\H}_{d',\de}^{(1)}$, ${\H}_{d',\de}^{(2)} $ and ${\H}_{d',\de}^{(1)}\cap{\H}_{d',\de}^{(2)}$.  
 Let us set then 
 $$y'=Ind^{\eps_0,s_0}_{P_{d'}(X^{(1)}_\de\cap X^{(1)}_\de),A,*}\circ {q'}^{(1,2)}_{d,d',\de,*}(x')$$ in $K_0^{\eps_0,s_0}(A_{{\H}_{d',\de}^{(1)}\cap{\H}_{d',\de}^{(2)}})$.
 From Remark \ref{rem-funct-inclusion}, we deduce
 that $$\jmath^{(1,2),\sharp,\eps_0,s_0}_{P_{d',\de}(X),A,*}(y')=0$$ in $K_0^{\eps_0,s_0}(A_{{\H}_{d',\de}^{(1)}})$ and 
 $$\jmath^{(2,1),\sharp,\eps_0,s_0}_{P_{d',\de}(X),A,*}(y')=0$$ in $K_0^{\eps_0,s_0}(A_{{\H}_{d',\de}^{(2)}})$.
 Let then $q$ be an $\eps_0$-$s_0$-projection in some $M_n(\widetilde{A_{{\H}_{d',\de}^{(1)}\cap{\H}_{d',\de}^{(2)}}})$ and let $l$ be an integer with $l\lq n$ such that $y'=[q,l]_{\eps_0,s_0}$ in $K^{\eps_0,s_0}_0(A_{{\H}_{d',\de}^{(1)}\cap{\H}_{d',\de}^{(2)}})$. We can assume without loss of generality that
 \begin{itemize}
 \item $q-\diag(I_l\ts Id_{\H_{d',\de}},0)$ is in $M_n({A_{{\H}_{d',\de}^{(1)}\cap{\H}_{d',\de}^{(2)}}})$(see \cite[Lemma 1.7]{oy4});
 \item $q$ is homotopic to $\diag(I_l\ts Id_{\H_{d',\de}},0)$ as an $\eps_0$-$s_0$-projection in  $M_n(\widetilde{A_{{\H}_{d',\de}^{(i)}}})$ for $i=1,2$.
 \end{itemize}
 Applying  \cite[Lemma 3.8]{oy4}), we see that for an integer $N$ with $N\gq n$, there exist $w_1$ and $w_2$ two 
 $\la\eps_0$-$h_{\eps_0}s_0$-unitaries   in  $M_N(\widetilde{A_{\H_{d',\de}}})$,  $u$  an $\la\eps_0$-$h_{\eps_0}s_0$-unitary in
 $M_n(\widetilde{A_{\H_{d',\de}}})$, $v$  an  $\la\eps_0$-$h_{\eps_0}s_0$-unitary in
 $M_{N-n}(\widetilde{A_{\H_{d',\de}}})$ such that
 \begin{itemize}
\item $w_i-I_N\ts Id_{\H_{d',\de}}$ is    in   $M_N({A_{{\H'}_{d',\de}^{(i)}}})$ for $i=1,2$;
\item $$\|w_1^*\diag(I_n\ts Id_{\H_{d',\de}},0)w_1-diag(q,0)\|<\la\eps_0$$ and
$$\|w_2\diag(I_n\ts Id_{\H_{d',\de}},0)w_2^*-\diag(q,0)\|<\la\eps_0.$$
 \item $\|\diag(u,v)-w_1w_2\|<\la\eps_0$.
 \end{itemize}
 \medskip
   
   {\bf Step 3: } Let $U$ be an $\eps$-$s$-unitary in some $M_m(\widetilde{A_{\ell^2(X,\H_{d',\de})}})$ such that  $z=-[U]_{\eps,s}$ in $K^{\eps,s}_1(A_{X})$ (up to Morita equivalence). According to \cite[Lemma 1.8]{oy4}, we can assume without loss of generality that 
  $U-I_m\ts Id_{\ell^2(X,\H_{d',\de})}$ is in $M_m({A_{\ell^2(X,\H_{d',\de})}})$. Up to stabilization and by construction of the Mayer-Vietoris controlled boundary map, there exist
   $W_1$ and $W_2$ two $\al\eps$-$k_{\eps}s$-unitaries in 
   $M_{2m}(\widetilde{A_{\ell^2(X,\H_{d',\de})}})$ and $Q$ an $\al\eps$-$k_\eps s$-projection  in 
   $M_{2m}(\widetilde{A_{\ell^2(X,\H_{d',\de})}})$  such that
   
    \begin{itemize}
\item $W_i-I_{2m}\ts Id_{\ell^2(X,\H_{d',\de})}$ is    in   $M_{2m}(A_{\ell^2(X_\de^{(i)},\H_{d',\de})})$ for $i=1,2$;
\item $Q-\diag(I_{m}\ts Id_{\ell^2(X,\H_{d',\de})},0)$ is in $M_{2m}(A_{\ell^2(X_\de^{(1)}\cap X_\de^{(2)},\H_{d',\de})})$
\item  $\|\diag(U,U^*)-W_1W_2\|<\al\eps$;
\item $$\|W_1^*\diag(I_{m}\ts Id_{\ell^2(X,\H_{d',\de})},0)W_1-Q\|<\al\eps$$ and
$$\|W_2\diag(I_n\ts Id_{\H_{d',\de}},0)W_2^*-Q\|<\al\eps;$$
 \item up  to Morita equivalence, then
  $$\partial^{\eps,s}_{X_\de^{(1)},X_\de^{(2)},A,*}(z)=-[Q,m]_{\al\eps,k_{\eps}s}$$ in 
 $K_0^{\al\eps,k_{\eps}s}(A_{X_\de^{(1)}\cap X_\de^{(2)}})$.
 \end{itemize}
 Let $s''$ be a positive number such that $s''\gq s'$ and $s''>2d'(h_{\eps_0}s_0+1)$.
 Applying \cite[Lemma 3.5]{oy4} to 
 $$\diag(Ad_{V_{\H_{d',\de},A}} u+I_n\ts Id_{\H_{d',\de}}-I_n\ts V_{\H_{d',\de}}V^*_{\H_{d',\de}},U),$$
  $$\diag(Ad_{V_{\H_{d',\de},A}} v+I_{N-n}\ts Id_{\H_{d',\de}}-I_{N-n}\ts V_{\H_{d',\de}}V^*_{\H_{d',\de}},U^*)$$ and to 
  the matrices obtained from
  $$\diag(Ad_{V_{\H_{d',\de},A}} w_1+I_N\ts Id_{\H_{d',\de}}-I_n\ts V_{\H_{d',\de}}V^*_{\H_{d',\de}},W_1),$$
 
 $$\diag(Ad_{V_{\H_{d',\de},A}} w_2+I_N\ts Id_{\H_{d',\de}}-I_n\ts V_{\H_{d',\de}}V^*_{\H_{d',\de}},W_2)$$  
 and 
  $$\diag(Ad_{V_{\H_{d',\de},A}} \diag(q,0)+I_N\ts Id_{\H_{d',\de}}-I_n\ts V_{\H_{d',\de}}V^*_{\H_{d',\de}},Q)$$

 by swapping the order
 of the coordinates $n+1,\ldots,N$ and $N+1,\ldots,N+m$, we see that for some control pair $(\la',h')$ that depends only on
 $(\al,k)$ and $(\la,h)$ with  $(\la',h')\gq (\la,h)$, there exists $U^{(1)}$ and $U^{(2)}$ two $\la'\eps$-$h'_{\eps}s''$-unitary in
 some $M_{n'}(\widetilde{A_{X}})$ such that
 \begin{itemize}
 \item $U^{(i)}-I_{n'}\ts Id_{\ell^2(X)}$ is in $\bigoplus_{k=1}^{n_i} M_{n'}\left(A_{\ell^2(X_{k,\de}^{(i)})}\right)$ for $i=1,2$;
 \item up to Morita equivalence, we have \begin{equation*}\begin{split}
 [U^{(1)}]_{\la'\eps,h'_{\eps}s''}&+[U^{(2)}]_{\la'\eps,h'_{\eps}s''}=\\&[Ad_{V_{\H_{d',\de},A}} u+I_n\ts Id_{\H_{d',\de}}-I_n\ts V_{\H_{d',\de}}V^*_{\H_{d',\de}}]_{\la'\eps,h'_{\eps}s''}-\iota_*^{\la'\eps,h'_{\eps}s''}(y)\end{split}\end{equation*} in $K_1^{\la'\eps,h'_{\eps}s''}(A_{\ell^2(X)})$.
 \end{itemize}
 
  \medskip
   
   {\bf Step 4: } Assume from now on that $\eps<\frac{1}{200\la'}$, the general case being deduced by rescaling using Proposition \ref{prop-compatibility-assembly-rescaling}. 
   Let $s'''$ and $d''$ be positive numbers with $d''\gq d'$, $s'''>h'_\eps s''$, $s'''> r_{d'',\la'\eps}$ and $s'''>2d''(h_{\eps_0}h_{\la\eps_0}s_0+1))$ such that
 the condition $QS_{X',A,*}(d'',h'_{\eps}s'',s''',\la'\eps)$ holds for any metric space $X'$ in $\widetilde{\X_\de}$ 	and any  $C^*$-algebra $A$.
 Then for $i=1,2$, there exists $x'^{(i)}$ in $K_1(P'_{d''}(X_\de^{(i)}),A)\cong \bigoplus_{j=1}^{n_i}
 K_1(P_{d''}(X_{j,\de}^{(i)}),A)$ such that up to Morita equivalence, then
 $${\nu'}_{X_\de^{(i)},A,*}^{\la'\eps,s''',d''}(x'^{(i)})=[U^{(i)}]_{\la'\eps,s''',d''}$$ in $K_1^{\la'\eps,s'''}(\bigoplus_{j=1}^{n_i}A_{X^{(i)}_{j,\de}})$. Since
 $Ind_{P_{d',\de}(X),A,*}$ is a $h$-isomorphism, there exists a unique  element $x'$ in $K_1(P_{d',\de}(X),A)$ such that
 $$Ind^{\la\eps_0,h_{\la\eps_0}h_{\eps_0}s_0}_{P_{d',\de}(X),A,*}(x')=[u]_{\la\eps_0,h_{\la\eps_0}h_{\eps_0}s_0}$$ in $K_1^{\la\eps_0,h_{\la\eps_0}h_{\eps_0}s_0}(A_{\H_{d',\de}})$. 
 Let $x$ in $K_1(P_{d''}(X),A)$ be the image of $x'$ under the morphism induced by the inclusion 
 $P_{d',\de}(X)\hookrightarrow P_{d''}(X)$. By choosing a non degenerate standard module for  $P_{d''}(X)$ that restricts to
 $(\H_{d',\de},\rho_{d',\de})$ on   $P_{d',\de}(X)$, it is straightforward to check using Lemma  \ref{lemma-subset} and Proposition 
 \ref{proposition-assembly-conjugate} that we have up  to Morita equivalence
 \begin{equation}\label{eq-surjectivity-step4}\nu^{\la'\eps,s''',d''}_{X,A,*}(x)=[Ad_{V_{\H_{d',\de},A}} u+I_n\ts Id_{\H_{d',\de}}-I_n\ts V_{\H_{d',\de}}V^*_{\H_{d',\de}}]_{\la'\eps,s'''}\end{equation}
 in $K_1^{\la'\eps,s'''}(A_X)$.
 For $i=1,2$, let $x^{(i)ø}$ be the image of  $x'^{(i)}$ under the map $K_1(P'_{d''}(X_\de^{(i)}),A)\to K_1(P_{d''}(X_\de^{(i)}),A)$ induced by the inclusion 
 $P'_{d''}(X_\de^{(i)})\hookrightarrow P_{d''}(X_\de^{(i)})$.
 We deduce from equation \eqref{eq-surjectivity-step4} that 
 \begin{equation*}\begin{split}
 \nu^{\la'\eps,s''',d''}_{X,A,*}&(x-\jmath_{P_{d''}(X_\de^{(1)}),A,*}(x^{(1)})-\jmath_{P_{d''}(X_\de^{(2)}),A,*}(x^{(2)}))\\
 =& [Ad_{V_{\H_{d',\de},A}} u+I_n\ts Id_{\H_{d',\de}}-I_n\ts V_{\H_{d',\de}}V^*_{\H_{d',\de}}]_{\la'\eps,s'''}-[U^{(1)}]_{la'\eps,s'''}-
 [U^{(2)}]_{\la'\eps,s'''}\\
 =& \iota^{-,\la'\eps,s'''}_*(y)\end{split},\end{equation*}where the first equality holds by Lemma  \ref{lemma-subset}.
  Thus the condition
 $QS_{X,A,*}(d'',s,s''',\eps)$ holds for any  $C^*$-algebra $A$.
 \qed
 \bibliographystyle{plain}

\end{document}